\documentclass[11pt]{amsart}
\usepackage{amsfonts,amssymb,amsmath,euscript}
\usepackage{graphicx}
\usepackage[matrix,arrow,curve]{xy}
\usepackage{wrapfig}

    \newtheorem{thm}{Theorem}                     [section]
    \newtheorem{thm*}{Theorem}
    \newtheorem{prop}[thm]{Proposition}
    \newtheorem{lemma}[thm]{Lemma}
    \newtheorem{cor}[thm]{Corollary}

    \newtheorem{lemma*}{Lemma}    
    
    \newtheorem{assump}[thm]{Assumption}

    \newtheorem{defn}[thm]{Definition}                 
    \newtheorem{conj}{Conjecture}                 
    \newtheorem{rems}[thm]{Remark}                     
    
    \newtheorem{rems*}{Remark}   

\newcommand{\ndef}{\newcommand*}
\def\rndef{\renewcommand}

\ndef{\myaddress}[1]{\begin{center} \it\small #1 \end{center}}

\ndef{\clA}{{\mathcal A}} \ndef{\rmA}{{\mathrm A}} \ndef{\mbA}{{\mathbb A}} \ndef{\bfA}{{\mathbf A}} \ndef{\euA}{{\EuScript A}} \ndef{\frA}{{\mathfrak A}}
\ndef{\clB}{{\mathcal B}} \ndef{\rmB}{{\mathrm B}} \ndef{\mbB}{{\mathbb B}} \ndef{\bfB}{{\mathbf B}} \ndef{\euB}{{\EuScript B}} \ndef{\frB}{{\mathfrak B}}
\ndef{\clC}{{\mathcal C}} \ndef{\rmC}{{\mathrm C}} \ndef{\mbC}{{\mathbb C}} \ndef{\bfC}{{\mathbf C}} \ndef{\euC}{{\EuScript C}} \ndef{\frC}{{\mathfrak C}}
\ndef{\clD}{{\mathcal D}} \ndef{\rmD}{{\mathrm D}} \ndef{\mbD}{{\mathbb D}} \ndef{\bfD}{{\mathbf D}} \ndef{\euD}{{\EuScript D}} \ndef{\frD}{{\mathfrak D}}
\ndef{\clE}{{\mathcal E}} \ndef{\rmE}{{\mathrm E}} \ndef{\mbE}{{\mathbb E}} \ndef{\bfE}{{\mathbf E}} \ndef{\euE}{{\EuScript E}} \ndef{\frE}{{\mathfrak E}}
\ndef{\clF}{{\mathcal F}} \ndef{\rmF}{{\mathrm F}} \ndef{\mbF}{{\mathbb F}} \ndef{\bfF}{{\mathbf F}} \ndef{\euF}{{\EuScript F}} \ndef{\frF}{{\mathfrak F}}
\ndef{\clG}{{\mathcal G}} \ndef{\rmG}{{\mathrm G}} \ndef{\mbG}{{\mathbb G}} \ndef{\bfG}{{\mathbf G}} \ndef{\euG}{{\EuScript G}} \ndef{\frG}{{\mathfrak G}}
\ndef{\clH}{{\mathcal H}} \ndef{\rmH}{{\mathrm H}} \ndef{\mbH}{{\mathbb H}} \ndef{\bfH}{{\mathbf H}} \ndef{\euH}{{\EuScript H}} \ndef{\frH}{{\mathfrak H}}
\ndef{\clI}{{\mathcal I}} \ndef{\rmI}{{\mathrm I}} \ndef{\mbI}{{\mathbb I}} \ndef{\bfI}{{\mathbf I}} \ndef{\euI}{{\EuScript I}} \ndef{\frI}{{\mathfrak I}}
\ndef{\clJ}{{\mathcal J}} \ndef{\rmJ}{{\mathrm J}} \ndef{\mbJ}{{\mathbb J}} \ndef{\bfJ}{{\mathbf J}} \ndef{\euJ}{{\EuScript J}} \ndef{\frJ}{{\mathfrak J}}
\ndef{\clK}{{\mathcal K}} \ndef{\rmK}{{\mathrm K}} \ndef{\mbK}{{\mathbb K}} \ndef{\bfK}{{\mathbf K}} \ndef{\euK}{{\EuScript K}} \ndef{\frK}{{\mathfrak K}}
\ndef{\clL}{{\mathcal L}} \ndef{\rmL}{{\mathrm L}} \ndef{\mbL}{{\mathbb L}} \ndef{\bfL}{{\mathbf L}} \ndef{\euL}{{\EuScript L}} \ndef{\frL}{{\mathfrak L}}
\ndef{\clM}{{\mathcal M}} \ndef{\rmM}{{\mathrm M}} \ndef{\mbM}{{\mathbb M}} \ndef{\bfM}{{\mathbf M}} \ndef{\euM}{{\EuScript M}} \ndef{\frM}{{\mathfrak M}}
\ndef{\clN}{{\mathcal N}} \ndef{\rmN}{{\mathrm N}} \ndef{\mbN}{{\mathbb N}} \ndef{\bfN}{{\mathbf N}} \ndef{\euN}{{\EuScript N}} \ndef{\frN}{{\mathfrak N}}
\ndef{\clO}{{\mathcal O}} \ndef{\rmO}{{\mathrm O}} \ndef{\mbO}{{\mathbb O}} \ndef{\bfO}{{\mathbf O}} \ndef{\euO}{{\EuScript O}} \ndef{\frO}{{\mathfrak O}}
\ndef{\clP}{{\mathcal P}} \ndef{\rmP}{{\mathrm P}} \ndef{\mbP}{{\mathbb P}} \ndef{\bfP}{{\mathbf P}} \ndef{\euP}{{\EuScript P}} \ndef{\frP}{{\mathfrak P}}
\ndef{\clQ}{{\mathcal Q}} \ndef{\rmQ}{{\mathrm Q}} \ndef{\mbQ}{{\mathbb Q}} \ndef{\bfQ}{{\mathbf Q}} \ndef{\euQ}{{\EuScript Q}} \ndef{\frQ}{{\mathfrak Q}}
\ndef{\clR}{{\mathcal R}} \ndef{\rmR}{{\mathrm R}} \ndef{\mbR}{{\mathbb R}} \ndef{\bfR}{{\mathbf R}} \ndef{\euR}{{\EuScript R}} \ndef{\frR}{{\mathfrak R}}
\ndef{\clS}{{\mathcal S}} \ndef{\rmS}{{\mathrm S}} \ndef{\mbS}{{\mathbb S}} \ndef{\bfS}{{\mathbf S}} \ndef{\euS}{{\EuScript S}} \ndef{\frS}{{\mathfrak S}}
\ndef{\clT}{{\mathcal T}} \ndef{\rmT}{{\mathrm T}} \ndef{\mbT}{{\mathbb T}} \ndef{\bfT}{{\mathbf T}} \ndef{\euT}{{\EuScript T}} \ndef{\frT}{{\mathfrak T}}
\ndef{\clU}{{\mathcal U}} \ndef{\rmU}{{\mathrm U}} \ndef{\mbU}{{\mathbb U}} \ndef{\bfU}{{\mathbf U}} \ndef{\euU}{{\EuScript U}} \ndef{\frU}{{\mathfrak U}}
\ndef{\clV}{{\mathcal V}} \ndef{\rmV}{{\mathrm V}} \ndef{\mbV}{{\mathbb V}} \ndef{\bfV}{{\mathbf V}} \ndef{\euV}{{\EuScript V}} \ndef{\frV}{{\mathfrak V}}
\ndef{\clW}{{\mathcal W}} \ndef{\rmW}{{\mathrm W}} \ndef{\mbW}{{\mathbb W}} \ndef{\bfW}{{\mathbf W}} \ndef{\euW}{{\EuScript W}} \ndef{\frW}{{\mathfrak W}}
\ndef{\clX}{{\mathcal X}} \ndef{\rmX}{{\mathrm X}} \ndef{\mbX}{{\mathbb X}} \ndef{\bfX}{{\mathbf X}} \ndef{\euX}{{\EuScript X}} \ndef{\frX}{{\mathfrak X}}
\ndef{\clY}{{\mathcal Y}} \ndef{\rmY}{{\mathrm Y}} \ndef{\mbY}{{\mathbb Y}} \ndef{\bfY}{{\mathbf Y}} \ndef{\euY}{{\EuScript Y}} \ndef{\frY}{{\mathfrak Y}}
\ndef{\clZ}{{\mathcal Z}} \ndef{\rmZ}{{\mathrm Z}} \ndef{\mbZ}{{\mathbb Z}} \ndef{\bfZ}{{\mathbf Z}} \ndef{\euZ}{{\EuScript Z}} \ndef{\frZ}{{\mathfrak Z}}

\ndef{\tA}{{\widetilde A}} \ndef{\tcA}{{\widetilde\clA}} \ndef{\ttcA}{\widetilde{\tcA}} \ndef{\sfA}{{\textsf A}} \ndef{\ttA}{\widetilde{\tA}} \ndef{\dzA}{{A^\sharp}}
\ndef{\tB}{{\widetilde B}} \ndef{\tcB}{{\widetilde\clB}} \ndef{\ttcB}{\widetilde{\tcB}} \ndef{\sfB}{{\textsf B}} \ndef{\ttB}{\widetilde{\tB}} \ndef{\dzB}{{B^\sharp}}
\ndef{\tC}{{\widetilde C}} \ndef{\tcC}{{\widetilde\clC}} \ndef{\ttcC}{\widetilde{\tcC}} \ndef{\sfC}{{\textsf C}} \ndef{\ttC}{\widetilde{\tC}} \ndef{\dzC}{{C^\sharp}}
\ndef{\tD}{{\widetilde D}} \ndef{\tcD}{{\widetilde\clD}} \ndef{\ttcD}{\widetilde{\tcD}} \ndef{\sfD}{{\textsf D}} \ndef{\ttD}{\widetilde{\tD}} \ndef{\dzD}{{D^\sharp}}
\ndef{\tE}{{\widetilde E}} \ndef{\tcE}{{\widetilde\clE}} \ndef{\ttcE}{\widetilde{\tcE}} \ndef{\sfE}{{\textsf E}} \ndef{\ttE}{\widetilde{\tE}} \ndef{\dzE}{{E^\sharp}}
\ndef{\tF}{{\widetilde F}} \ndef{\tcF}{{\widetilde\clF}} \ndef{\ttcF}{\widetilde{\tcF}} \ndef{\sfF}{{\textsf F}} \ndef{\ttF}{\widetilde{\tF}} \ndef{\dzF}{{F^\sharp}}
\ndef{\tG}{{\widetilde G}} \ndef{\tcG}{{\widetilde\clG}} \ndef{\ttcG}{\widetilde{\tcG}} \ndef{\sfG}{{\textsf G}} \ndef{\ttG}{\widetilde{\tG}} \ndef{\dzG}{{G^\sharp}}
\ndef{\tH}{{\widetilde H}} \ndef{\tcH}{{\widetilde\clH}} \ndef{\ttcH}{\widetilde{\tcH}} \ndef{\sfH}{{\textsf H}} \ndef{\ttH}{\widetilde{\tH}} \ndef{\dzH}{{H^\sharp}}
\ndef{\tI}{{\widetilde I}} \ndef{\tcI}{{\widetilde\clI}} \ndef{\ttcI}{\widetilde{\tcI}} \ndef{\sfI}{{\textsf I}} \ndef{\ttI}{\widetilde{\tI}} \ndef{\dzI}{{I^\sharp}}
\ndef{\tJ}{{\widetilde J}} \ndef{\tcJ}{{\widetilde\clJ}} \ndef{\ttcJ}{\widetilde{\tcJ}} \ndef{\sfJ}{{\textsf J}} \ndef{\ttJ}{\widetilde{\tJ}} \ndef{\dzJ}{{J^\sharp}}
\ndef{\tK}{{\widetilde K}} \ndef{\tcK}{{\widetilde\clK}} \ndef{\ttcK}{\widetilde{\tcK}} \ndef{\sfK}{{\textsf K}} \ndef{\ttK}{\widetilde{\tK}} \ndef{\dzK}{{K^\sharp}}
\ndef{\tL}{{\widetilde L}} \ndef{\tcL}{{\widetilde\clL}} \ndef{\ttcL}{\widetilde{\tcL}} \ndef{\sfL}{{\textsf L}} \ndef{\ttL}{\widetilde{\tL}} \ndef{\dzL}{{L^\sharp}}
\ndef{\tM}{{\widetilde M}} \ndef{\tcM}{{\widetilde\clM}} \ndef{\ttcM}{\widetilde{\tcM}} \ndef{\sfM}{{\textsf M}} \ndef{\ttM}{\widetilde{\tM}} \ndef{\dzM}{{M^\sharp}}
\ndef{\tN}{{\widetilde N}} \ndef{\tcN}{{\widetilde\clN}} \ndef{\ttcN}{\widetilde{\tcN}} \ndef{\sfN}{{\textsf N}} \ndef{\ttN}{\widetilde{\tN}} \ndef{\dzN}{{N^\sharp}}
\ndef{\tO}{{\widetilde O}} \ndef{\tcO}{{\widetilde\clO}} \ndef{\ttcO}{\widetilde{\tcO}} \ndef{\sfO}{{\textsf O}} \ndef{\ttO}{\widetilde{\tO}} \ndef{\dzO}{{O^\sharp}}
\ndef{\tP}{{\widetilde P}} \ndef{\tcP}{{\widetilde\clP}} \ndef{\ttcP}{\widetilde{\tcP}} \ndef{\sfP}{{\textsf P}} \ndef{\ttP}{\widetilde{\tP}} \ndef{\dzP}{{P^\sharp}}
\ndef{\tQ}{{\widetilde Q}} \ndef{\tcQ}{{\widetilde\clQ}} \ndef{\ttcQ}{\widetilde{\tcQ}} \ndef{\sfQ}{{\textsf Q}} \ndef{\ttQ}{\widetilde{\tQ}} \ndef{\dzQ}{{Q^\sharp}}
\ndef{\tR}{{\widetilde R}} \ndef{\tcR}{{\widetilde\clR}} \ndef{\ttcR}{\widetilde{\tcR}} \ndef{\sfR}{{\textsf R}} \ndef{\ttR}{\widetilde{\tR}} \ndef{\dzR}{{R^\sharp}}
\ndef{\tS}{{\widetilde S}} \ndef{\tcS}{{\widetilde\clS}} \ndef{\ttcS}{\widetilde{\tcS}} \ndef{\sfS}{{\textsf S}} \ndef{\ttS}{\widetilde{\tS}} \ndef{\dzS}{{S^\sharp}}
\ndef{\tT}{{\widetilde T}} \ndef{\tcT}{{\widetilde\clT}} \ndef{\ttcT}{\widetilde{\tcT}} \ndef{\sfT}{{\textsf T}} \ndef{\ttT}{\widetilde{\tT}} \ndef{\dzT}{{T^\sharp}}
\ndef{\tU}{{\widetilde U}} \ndef{\tcU}{{\widetilde\clU}} \ndef{\ttcU}{\widetilde{\tcU}} \ndef{\sfU}{{\textsf U}} \ndef{\ttU}{\widetilde{\tU}} \ndef{\dzU}{{U^\sharp}}
\ndef{\tV}{{\widetilde V}} \ndef{\tcV}{{\widetilde\clV}} \ndef{\ttcV}{\widetilde{\tcV}} \ndef{\sfV}{{\textsf V}} \ndef{\ttV}{\widetilde{\tV}} \ndef{\dzV}{{V^\sharp}}
\ndef{\tW}{{\widetilde W}} \ndef{\tcW}{{\widetilde\clW}} \ndef{\ttcW}{\widetilde{\tcW}} \ndef{\sfW}{{\textsf W}} \ndef{\ttW}{\widetilde{\tW}} \ndef{\dzW}{{W^\sharp}}
\ndef{\tX}{{\widetilde X}} \ndef{\tcX}{{\widetilde\clX}} \ndef{\ttcX}{\widetilde{\tcX}} \ndef{\sfX}{{\textsf X}} \ndef{\ttX}{\widetilde{\tX}} \ndef{\dzX}{{X^\sharp}}
\ndef{\tY}{{\widetilde Y}} \ndef{\tcY}{{\widetilde\clY}} \ndef{\ttcY}{\widetilde{\tcY}} \ndef{\sfY}{{\textsf Y}} \ndef{\ttY}{\widetilde{\tY}} \ndef{\dzY}{{Y^\sharp}}
\ndef{\tZ}{{\widetilde Z}} \ndef{\tcZ}{{\widetilde\clZ}} \ndef{\ttcZ}{\widetilde{\tcZ}} \ndef{\sfZ}{{\textsf Z}} \ndef{\ttZ}{\widetilde{\tZ}} \ndef{\dzZ}{{Z^\sharp}}

\ndef{\bfa}{{\mathbf a}}
\ndef{\bfb}{{\mathbf b}}
\ndef{\bfc}{{\mathbf c}}
\ndef{\bfd}{{\mathbf d}}

\ndef{\euu}{{\EuScript u}}

  \ndef{\eps}{\varepsilon}

\let\geq\geqslant
\let\leq\leqslant

\ndef{\lims}[1]{\lim\limits_{#1}}
\ndef{\sums}[1]{\sum\limits_{#1}}
\ndef{\ints}[1]{\int_{#1}}
\ndef{\sups}[1]{\sup\limits_{#1}}
\ndef{\liminfty}[1]{\lims{#1\to\infty}}
\ndef{\suminf}[1]{\sums{#1=1}^\infty}

\ndef{\limo}[1]{\omega\mbox{-}\!\!\!\lims{#1\to\infty}}          
\ndef{\limL}[1]{\rmL\mbox{-}\!\!\!\lims{#1\to\infty}}            
\ndef{\limLOne}[1]{\clL_1\mbox{-}\!\lims{#1}}
\ndef{\tildelimo}[1]{\tilde\omega\mbox{-}\!\!\!\lims{#1\to\infty}}
\ndef{\slim}{\mathrm{s}\mbox{-}\!\!\lim}          
\ndef{\wlim}{\mathrm{w}\mbox{-}\!\lim}          

\ndef{\Aut}{\operatorname{Aut}}      
\ndef{\Ch}{\operatorname{ch}}        
\ndef{\End}{\operatorname{End}}      
\ndef{\Hom}{\operatorname{Hom}}      
\rndef{\ker}{\operatorname{ker}}      
\ndef{\coker}{\operatorname{coker}}      
\ndef{\im}{\operatorname{im}}        
\ndef{\Log}{\operatorname{Log}}      
\ndef{\OP}{\operatorname{OP}}        
\ndef{\Op}{\operatorname{Op}}        
\ndef{\Symb}{\operatorname{Symb}}    
\ndef{\Tr}{\operatorname{Tr}}        
\ndef{\Wres}{\operatorname{Wres}}    
\ndef{\cl}{\operatorname{cl}}        
\ndef{\com}{\operatorname{com}}
\ndef{\const}{\operatorname{const}}  
\ndef{\conv}{\operatorname{conv}}    
\rndef{\det}{\operatorname{det}}     
\ndef{\Var}{\operatorname{Var}}
\ndef{\Cov}{\operatorname{Cov}}

\ndef{\detFK}[1]{\Delta\brs{#1}} 
\ndef{\detFKrel}[2]{\Delta_{#2}\brs{#1}} 

\ndef{\adj}{\operatorname{adj}}    
\ndef{\diag}{\operatorname{diag}}    
\ndef{\dist}{\operatorname{dist}}    
\ndef{\dom}{\operatorname{dom}}      
\ndef{\ec}{\operatorname{ec}}        
\ndef{\id}{\mathrm{Id}}                        
\ndef{\ind}{\operatorname{ind}}      
\ndef{\mydeg}{\operatorname{deg}}    
\ndef{\op}{\operatorname{op}}
\ndef{\rank}{\operatorname{rank}}
\ndef{\res}{\operatorname{res}}      
\ndef{\rng}{\operatorname{ran}}      
\ndef{\sflow}{\operatorname{sf}}     
\ndef{\isf}{\operatorname{isf}}      
\ndef{\sign}{\operatorname{sign}}    
\ndef{\sgn}{\operatorname{sgn}}      
\ndef{\sing}{\operatorname{sing}}    
\ndef{\supp}{\operatorname{supp}}    
\ndef{\tr}{\operatorname{tr}}        
\ndef{\var}{\operatorname{var}}      
\ndef{\vol}{\operatorname{vol}}      
\ndef{\wn}{\operatorname{wn}}        
\ndef{\wres}{\operatorname{wres}}    
\rndef{\Im}{\operatorname{Im}}       
\rndef{\Re}{\operatorname{Re}}       

\ndef{\prng}[1]{\mathrm R_{#1}} 
\ndef{\pker}[1]{\mathrm N_{#1}} 
\ndef{\rprng}[2]{\mathrm R_{#1}^{#2}}           
\ndef{\rpker}[2]{\mathrm N_{#1}^{#2}}           
\ndef{\rsupp}[1]{\supp_r(#1)}
\ndef{\lsupp}[1]{\supp_l(#1)}
\ndef{\rslv}[1]{R_z(#1)}      
\ndef{\HH}{H}                 
\ndef{\tHH}{\tilde \HH}       
\ndef{\VV}{V}                 
\ndef{\Rz}{R_z}               
\ndef{\tRz}{\tR_z}            
\ndef{\psif}[1]{#1^{[1]}} 
\ndef{\WPlus}[1]{W_{#1}(\mbR)} 

\newcommand{\xia}{\xi^{(a)}}
\newcommand{\xis}{\xi^{(s)}}

\ndef{\bndl}{\xi}                         
\ndef{\bndlA}{\eta}                       
\ndef{\GlueMap}{\varphi}                  
\ndef{\ChartMap}{h}                       
\ndef{\chern}{\ensuremath{\mathrm{ch}}}
\ndef{\hilb}{\clH}                     
\ndef{\hilba}{\clH^{(a)}}                    
\ndef{\hilbs}{\clH^{(s)}}                    
   \ndef{\hilbasargument}{(\hilb)} 
\ndef{\LpH}[1]{\clL_{#1}\hilbasargument}       
\ndef{\saLpH}[1]{\clL_{sa}^{#1}\hilbasargument}       
\ndef{\clBH}{\clB\hilbasargument}              
\ndef{\ubBH}{\clB_1\hilbasargument}            
\ndef{\clCH}{\clC\hilbasargument}              
\ndef{\clKH}{\clK\hilbasargument}              
\ndef{\clFH}{\clF\hilbasargument}              
\ndef{\clUH}{\clU\hilbasargument}              
\ndef{\clCFH}{{\clC\clF}\hilbasargument}       
\ndef{\saBH}{\clB_{sa}\hilbasargument}         
\ndef{\saCH}{\clC_{sa}\hilbasargument}         
\ndef{\saFH}{\clF_{sa}\hilbasargument}         
\ndef{\saKH}{\clK_{sa}\hilbasargument}         
\ndef{\saCFH}{\clC\clF_{sa}\hilbasargument}    
\ndef{\clUFH}{\clU\clF\hilbasargument}         
\ndef{\Uinj}{\clU_{inj}\hilbasargument}        
\ndef{\UFinj}{\clU\clF_{inj}\hilbasargument}   

\ndef{\spproj}[2]{E^{#1}_{#2}}                      
\ndef{\spprojb}[2]{E^{#2}_{#1}}                     

\ndef{\LpN}[1]{\clL^{#1}(\clN,\tau)}     
\ndef{\saLpN}[1]{\clL^{#1}_{sa}(\clN,\tau)} 
\ndef{\rLpN}[1]{L^{#1}(\clN,\tau)}       
\ndef{\clAND}{(\clA,\clN,D)}             
\ndef{\clBA}{{\clB(\clA)}}
\ndef{\saKN}{{\clK_{sa}(\clN,\tau)}}          
\ndef{\clKN}{{\clK(\clN,\tau)}}          
\ndef{\clKtN}{{\clK(\tilde\clN,\tau)}}   
\ndef{\clFN}{{\clF(\clN,\tau)}}          
\ndef{\saFN}{{\clF_{sa}(\clN,\tau)}}     
\ndef{\clPN}{\clP(\clN)}                 
\ndef{\clQN}{\clQ(\clN,\tau)}            
\ndef{\infPN}{{\clP_\tau^\infty(\clN)}}  
\ndef{\clOF}[2]{\clF_{#1\mbox{-}#2}(\clN,\tau)}         
\ndef{\oind}[2]{{\rm \tau\mbox{-}ind}_{#1\mbox{-}#2}}   
\ndef{\tind}{\tau\mbox{-}\ind}                  
\ndef{\DInd}{\ind_{\clD,\tau}}           
\ndef{\BF}{Breuer-Fredholm}              
\ndef{\skewfred}[2]{$(#1\cdot #2)$ $\tau$\tire Fredholm}   
\ndef{\affl}{\eta}                       
\ndef{\vNa}{von Neumann algebra}         
\ndef{\nsf}{faithful normal semifinite } 
\ndef{\taubrs}[1]{\tau\brackets{#1}}     
\ndef{\sqbrs}[1]{[#1]}        
\ndef{\Sqbrs}[1]{\big[#1\big]}        
\ndef{\SqBrs}[1]{\Big[#1\Big]}        

\ndef{\domd}{\bigcap\limits_{n\ge 0} \dom\;\delta^n}         
\ndef{\DiffOP}{{\rm \clD}}
\ndef{\ADA}{\clA \cup [\clD,\clA]}
\ndef{\DixIdeal}[1]{\LpH{#1,\infty}}               
\ndef{\dixideal}{\ell^{1,\infty}}                  
\ndef{\WDixIdeal}{\LpH{1,\mathrm w}}               
\ndef{\DixIdealPos}[1]{\DixIdeal{#1}_+}            
\ndef{\DixIdealN}[1]{\LpN{#1,\infty}}              
\ndef{\DixIdealNPar}[2]{\clL^{#1,\infty}_{#2}(\clN,\tau)}    
\ndef{\DixIdealNPos}[1]{\LpN{#1,\infty}_+}                   
\ndef{\TrD}{\Tr_\omega}                                      
\ndef{\tauD}{{\tau_\omega}}                                  
\ndef{\ILogN}{\frac 1{\log(1+N)}}
\ndef{\DixNorm}[1]{\norm{#1}_{(1,\infty)}}                   
\ndef{\DixInt}[1]{\ints 0^t \mu_s(#1)\,ds}
\ndef{\DixIntL}[1]{\ints 0^{\lambda_{1/t}(#1)}\mu_s(#1)\,ds}
    \ndef{\SmallIdeal}{{\clL^{1, \mathrm w}}}
    \ndef{\SmallIdealMeas}{{\clL^{1, \mathrm w}_m}}
    \ndef{\DixIntII}[1]{\int_0^t \mu_s(#1)\,ds}
    \ndef{\DixIntf}[1]{\Phi_t(#1)}
    \ndef{\DixIntg}[1]{\Psi_t(#1)}

\ndef{\lpi}{\clL^{1,\pi}(\clN,\tau)}

\ndef{\strl}[1]{\stackrel \longrightarrow {#1}}
\ndef{\IIinfty}{$\mathrm{II}_\infty$\ }

\ndef{\fourier}[1]{\clF(#1)}          
\ndef{\HaarMeasBohrs}{\nu}            
\ndef{\BrownsMeas}{\mu}               
\ndef{\BohrCont}[1]{\tilde{#1}}       
\ndef{\APMean}{{M}}                   
\ndef{\CDSS}{{\clA_B}}                
\ndef{\matr}{{\rm Mat}}               
\ndef{\seque}[1]{\ensuremath{\{#1_n\}_{n=1}^\infty}}    
\ndef{\sequen}[2]{\ensuremath{\{#1_#2\}_{#2=1}^\infty}}    
\ndef{\Seque}[1]{\ensuremath{\left(#1_0,#1_1,#1_2,\dots\right)}}    
\ndef{\Cesaro}{H}                           
\ndef{\CesaroRPlus}{M}                      
\ndef{\Dilation}{D}                         
\ndef{\Shift}{T}                            

\ndef{\norm}[1]{\left\Vert#1\right\Vert}    
\ndef{\TrNorm}[1]{\norm{#1}_1}              
\ndef{\HSNorm}[1]{\norm{#1}_2}              
\ndef{\InftyNorm}[1]{\norm{#1}_\infty}      
\ndef{\normQN}[1]{\norm{#1}_{\clQN}}        
\ndef{\clLpnorm}[2]{\norm{#2}_{\clL^{#1}}}    
\ndef{\clLnorm}[1]{\clLpnorm{1}{#1}}    

\ndef{\ccurve}{\gamma}                      

\ndef{\abs}[1]{\left\lvert#1\right\rvert}   
\ndef{\set}[1]{\left\{#1\right\}}           
\ndef{\brackets}[1]{\left(#1\right)}        
\ndef{\brs}[1]{\brackets{#1}}               
\ndef{\Brs}[1]{\big(#1\big)}                
\ndef{\BRS}[1]{\Big(#1\Big)}                
\ndef{\scal}[2]{\left\langle #1,#2\right\rangle}               
\ndef{\precprec}{\prec\!\!\!\prec}
\ndef{\qeq}{\stackrel?=}
\ndef{\spectrum}[1]{\sigma_{#1}} 
\ndef{\spectruma}[1]{\sigma^{(a)}_{#1}} 
\rndef{\emptyset}{\varnothing}                              
\ndef{\csupp}{c}                           
\ndef{\closure}[1]{\overline{#1}}
\ndef{\linspan}[1]{\mathrm{span}\ {#1}}
\ndef{\bddborel}[1]{B(#1)}                 
\ndef{\charfunc}{\chi}
\ndef{\FrDer}{\euD}                        
\ndef{\LieDer}[1]{\pounds_{#1}\,}          
\ndef{\dds}{\left.\frac d{ds} \right|_{s = 0}}
\ndef{\ortcmp}[1]{#1^{\scriptscriptstyle \perp}}            
\ndef{\Laplace}{\Delta}                    

\ndef{\matrPQ}[3]
{
    \left(
      \begin{array}{cc}
        #1_{11} & #1_{12} \\
        #1_{21} & #1_{22}
      \end{array}
    \right)_{[#2,#3]}
}

\ndef{\margOK}{\marginpar{\bf \small OK}}

\newcounter{margcomcount}
\ndef{\margcom}[1]{\marginpar{\bf \small #1} \addtocounter{margcomcount}{1}
   \index{\indexcom{{\bf COMMENT: #1}}}}

\ndef{\mytimes}{\!\times\!}
\ndef{\sss}[1]{\subsubsection{}\label{#1}}

\rndef{\phi}{\varphi} \ndef{\OpenUnitDisk}{D}
\ndef{\RHS}{RHS}                            
\ndef{\LHS}{LHS} 
\ndef{\ttt}{\Leftrightarrow}
\ndef{\then}{\Rightarrow}
\ndef{\tto}{\longrightarrow}
\ndef{\nno}{\nonumber\\}
\ndef{\newn}[1]{\index{#1} {\bfseries #1}}       
\ndef{\la}{\langle}
\ndef{\ra}{\rangle}
\ndef{\dbar}{{\;\bar{\phantom{o}} \!\!\!\! d}}
\ndef{\stl}[1]{\stackrel{\vbox to 0pt{\vss\hbox{$\scriptstyle #1$}}}}
\ndef{\mathcomment}[1]{{\hfill \qquad\qquad\qquad\text{by (#1)}}}        
\ndef{\mathcomm}[1]{{\hfill \qquad\qquad\qquad\qquad\qquad\text{#1}}}        
\ndef{\details}[1]{\smallskip\begin{center} {\bf Here:}
#1\end{center}\medskip} \ndef{\indexcom}[1]{ --- #1}
\ndef{\longsim}{\ \sim \ }              
\ndef{\tire}{-}              
\ndef{\intinfinf}{\int_{-\infty}^\infty}
     \ndef{\npartial}{\slash\!\!\!\partial}
     \ndef{\Heis}{\operatorname{Heis}}
     \ndef{\Solv}{\operatorname{Solv}}
     \ndef{\Spin}{\operatorname{Spin}}
     \ndef{\SO}{\operatorname{SO}}
     \ndef{\Index}{\operatorname{index}}

             \ndef{\p}{\partial}
             \ndef{\dd}{|\clD|}
             \ndef{\n}{\parallel}

\let\LatexCite=\cite  

\let\ifnumref\iffalse 

\ndef{\ifuncited}[4]{\expandafter\ifx\csname used#4\endcsname\relax}

\ndef{\ifcited}[4]{\expandafter\ifx\csname used#4\endcsname\relax\else}

  \ndef{\papertitle}[1]{ \emph{#1}, }
  \ndef{\paperauthor}[2]{#2}  
  \ndef{\pbbi}[9]{%
      \ifcited{#1}{#2}{#3}{#5}%
        \ifnumref%
          \bibitem{#5}\paperauthor{#1}{#6},\papertitle{#7}#8.%
        \else%
          \advance #9 by 1%
          \ifnum#9<1%
            \bibitem[#4]{#5}\paperauthor{#1}{#6}, \papertitle{#7}#8.%
          \else%
            \bibitem[#4$_{\the#9}$]{#5}\paperauthor{#1}{#6},\papertitle{#7}#8.%
          \fi%
        \fi%
      \fi%
  }
  \ndef{\mbbi}[8]{%
     \ifcited{#1}{#2}{#3}{#5}%
        \ifnumref%
          \bibitem{#5}\paperauthor{#1}{#6},\papertitle{#7}#8.%
        \else%
          \bibitem[#4]{#5}\paperauthor{#1}{#6},\papertitle{#7}#8.%
        \fi%
     \fi%
  }

\ndef{\AddCite}[1]{%
   \ifuncited{0}{0}{0}{#1}%
     \expandafter\gdef\csname used#1\endcsname {}%
   \fi%
}

\def\ProcessCite#1,{%
     \ifx\relax#1%
         \let\next=\relax%
     \else%
         \AddCite{#1}%
         \let\next=\ProcessCite%
     \fi%
     \next%
}

\ndef{\AddCites}[1]{\ProcessCite#1,\relax,}

\ndef{\CiteWithoutExtension}[1]{%
   \AddCites{#1}%
   \LatexCite{#1}%
}

\def\CiteWithExtension[#1]#2{%
   \AddCites{#2}%
   \LatexCite[#1]{#2}%
}

\ndef{\CleverCite}{%
    \ifx\NChar[ %
       \let\MyCite=\CiteWithExtension %
    \else %
       \let\MyCite=\CiteWithoutExtension %
    \fi %
    \MyCite%
}

\renewcommand{\cite}{\futurelet\NChar\CleverCite}

      \ndef{\volume}[1]{{\bf #1}}
      \ndef{\VolYearPP}[3]{\ifnum#2=0 (to appear)\else\volume{#1} (#2), #3\fi}
      \ndef{\VolNoYearPP}[4]{\ifnum#3=0 (to appear)\else\volume{#1} #2 (#3), #4\fi}
      \ndef{\libcode}[1]{}

\ndef{\jnActaMath}[3]{Acta Math. \VolYearPP{#1}{#2}{#3}}                       
\ndef{\jnAdvMath}[3]{Adv. in~Math. \VolYearPP{#1}{#2}{#3}}                     
\ndef{\jnAlgAnal}[3]{Algebra i~Analiz \VolYearPP{#1}{#2}{#3}}
\ndef{\jnAmerJMath}[3]{Amer. J. Math. \VolYearPP{#1}{#2}{#3}}                  
\ndef{\jnAmerMathMonth}[3]{Amer. Math. Monthly \VolYearPP{#1}{#2}{#3}}         
\ndef{\jnAnnMath}[4]{Ann. of~Math. \VolNoYearPP{#1}{#2}{#3}{#4}}               
\ndef{\jnAnalMath}[3]{J. Anal. Math. \VolYearPP{#1}{#2}{#3}}                   
\ndef{\jnArchRatMechAnal}[3]{Arch. Rational Mech. Anal. \VolYearPP{#1}{#2}{#3}}                   
\ndef{\jnBullLondMathSoc}[3]{Bull. London Math. Soc. \VolYearPP{#1}{#2}{#3}}   
\ndef{\jnBullAMS}[3]{Bull. Amer. Math. Soc. \VolYearPP{#1}{#2}{#3}}   
\ndef{\jnCanMathBull}[3]{Canad. Math. Bull. \VolYearPP{#1}{#2}{#3}}            
\ndef{\jnCanMath}[3]{Canad. J.~Math. \VolYearPP{#1}{#2}{#3}}             
\ndef{\jnCommMathPhys}[3]{Comm. Math. Phys. \VolYearPP{#1}{#2}{#3}}             
\ndef{\jnCommPDE}[3]{Comm. Partial Differential Equations \VolYearPP{#1}{#2}{#3}}             
\ndef{\jnComptRendue}[3]{C.\,R.~Acad. Sci. Paris S\'er. A-B \VolYearPP{#1}{#2}{#3}}      
\ndef{\jnContMath}[3]{Contemporary Math. \VolYearPP{#1}{#2}{#3}}               %
\ndef{\jnDukeMJ}[3]{Duke Math. J. \VolYearPP{#1}{#2}{#3}}
\ndef{\jnDiffGeom}[3]{J.~Diff. Geom. \VolYearPP{#1}{#2}{#3}}                   
\ndef{\jnErgodicTheory}[3]{Ergodic Theory and Dynamical Systems \VolYearPP{#1}{#2}{#3}} 
\ndef{\jnFuncAnal}[3]{J.~Functional Analysis \VolYearPP{#1}{#2}{#3}}           
\ndef{\jnFunkAnalPril}[4]{Funct. Anal. Appl. \VolNoYearPP{#1}{#2}{#3}{#4}}  
\ndef{\jnGAFA}[3]{GAFA \VolYearPP{#1}{#2}{#3}}                                 
\ndef{\jnIHES}[3]{IHES Publ. Math. (Paris) \VolYearPP{#1}{#2}{#3}}             
\ndef{\jnIEOT}[3]{Integral Equations Operator Theory   \VolYearPP{#1}{#2}{#3}} 
\ndef{\jnIsrMath}[3]{Israel J.~Math. \VolYearPP{#1}{#2}{#3}}                   
\ndef{\jnKTheory}[3]{K-Theory \VolYearPP{#1}{#2}{#3}}                          
\ndef{\jnLetMathPhys}[3]{Lett. Math. Phys. \VolYearPP{#1}{#2}{#3}}             
\ndef{\jnMathAnn}[3]{Math. Ann. \VolYearPP{#1}{#2}{#3}}                        
\ndef{\jnMathAnalAppl}[3]{J.~Math. Anal. and Appl. \VolYearPP{#1}{#2}{#3}}     
\ndef{\jnMathNachr}[3]{Math. Nachr. \VolYearPP{#1}{#2}{#3}}
\ndef{\jnMathPhys}[3]{J. Math. Phys. \VolYearPP{#1}{#2}{#3}}
\ndef{\jnMathSocJap}[3]{J. Math. Soc. Japan \VolYearPP{#1}{#2}{#3}}
\ndef{\jnOperTheory}[3]{J.~Operator Theory \VolYearPP{#1}{#2}{#3}}             
\ndef{\jnPacJMath}[3]{Pacific J.~Math. \VolYearPP{#1}{#2}{#3}}                  
\ndef{\jnPositivity}[3]{Positivity \VolYearPP{#1}{#2}{#3}}
\ndef{\jnProcAmerMS}[3]{Proc. Amer. Math. Soc. \VolYearPP{#1}{#2}{#3}}         
\ndef{\jnProcCambPhilSoc}[3]{Math. Proc. Camb. Phil. Soc. \VolYearPP{#1}{#2}{#3}}
\ndef{\jnReineAngew}[3]{J.~Reine Angew. Math. \VolYearPP{#1}{#2}{#3}}          
\ndef{\jnTokyoMath}[3]{Tokyo J.~Math. \VolYearPP{#1}{#2}{#3}}
\ndef{\jnTopology}[3]{Topology \VolYearPP{#1}{#2}{#3}}
\ndef{\jnTransAmerMathSoc}[3]{Trans. Amer. Math. Soc. \VolYearPP{#1}{#2}{#3}}
\ndef{\jnIzvANSSSR}[3]{Izv. Akad. Nauk SSSR, Ser. Mat. \VolYearPP{#1}{#2}{#3}}
\ndef{\jnIzvVyshUchZav}[3]{Izv. Vyssh. Uch. Zav., Mat. \VolYearPP{#1}{#2}{#3} (Russian)}
\ndef{\jnIzdatLenUniv}[2]{Izdat. Leningrad. Univ., Leningrad, (#1), #2 (Russian)}
\ndef{\jnFieldsInsComm}[3]{Fields Inst. Comm. \VolYearPP{#1}{#2}{#3}}
\ndef{\jnDoklANSSSR}[3]{Dokl. Akad. Nauk SSSR \VolYearPP{#1}{#2}{#3}}
\ndef{\jnMatZametki}[3]{Matem. zametki \VolYearPP{#1}{#2}{#3}}
\ndef{\jnRussMathSurvey}[3]{Russian Math. Surveys \VolYearPP{#1}{#2}{#3}}
\ndef{\jnSibMathJ}[3]{Sib. Math.~J. \VolYearPP{#1}{#2}{#3}}
\ndef{\jnSovMath}[3]{J.~Soviet math. \VolYearPP{#1}{#2}{#3}}
\ndef{\jnTransMoscMathSoc}[3]{Trans. Moscow Math. Soc. \VolYearPP{#1}{#2}{#3}}
\ndef{\jnUMN}[3]{Uspekhi Mat. Nauk \VolYearPP{#1}{#2}{#3}}

\ndef{\bkTransMathMon}[2]{Trans. Math. Monographs, AMS, \volume{#1}, #2}

\ndef{\pbBirkhauser}[1]{Birkh\"auser, Boston, #1}
\ndef{\pbFactorial}[1]{Moscow, Factorial, #1}
\ndef{\pbGauthier}[1]{Gauthier-Villars, Paris, #1}
\ndef{\pbNauka}[1]{Moscow, Nauka, #1 (Russian)}
\ndef{\pbNaukaR}[1]{Москва, Наука, #1}
\ndef{\pbPrinceton}[1]{Princeton University Press, Princeton, New Jersey, #1}
\ndef{\pbPublPerish}[1]{Publish or Perish Inc., Berkeley, #1}
\ndef{\pbSpringer}[1]{Springer-Verlag, #1}

\ndef{\myauthor}[1]{\mbox{#1}}

\ndef{\Agmon}{\myauthor{Sh.\,Agmon}}
\ndef{\Ahiezer}{\myauthor{N.\,I.\,Ahiezer}}
\ndef{\Arazy}{\myauthor{J.\,Arazy}}
\ndef{\Aronszajn}{\myauthor{N.\,Aronszajn}}
\ndef{\Astashkin}{\myauthor{S.\,V.\,Astashkin}}
\ndef{\Atiyah}{\myauthor{M.\,Atiyah}}
\ndef{\Avron}{\myauthor{J.\,E.\,Avron}}
\ndef{\Azamov}{\myauthor{N.\,A.\,Azamov}}
\ndef{\Banach}{\myauthor{S.\,Banach}}
\ndef{\Benameur}{\myauthor{M-T.\,Benameur}}
\ndef{\Bennett}{\myauthor{C.\,Bennett}}
\ndef{\Berezin}{\myauthor{F.\,A.\,Berezin}}
\ndef{\Berline}{\myauthor{N.\,Berline}}
\ndef{\Birman}{\myauthor{M.\,Sh.\,Birman}}
\ndef{\Blackadar}{\myauthor{B.\,Blackadar}}
\ndef{\Bogolyubov}{\myauthor{N.\,N.\,Bogolyubov}}
\ndef{\Bonsall}{\myauthor{F.\,F.\,Bonsall}}
\ndef{\Bony}{\myauthor{J.\,F.\,Bony}}
\ndef{\BoosBavnbek}{\myauthor{B.\,Boo$\beta$-Bavnbek}}
\ndef{\Bott}{\myauthor{R.\,Bott}}
\ndef{\Branges}{\myauthor{L.\,de Branges}}
\ndef{\Bratteli}{\myauthor{O.\,Bratteli}}
\ndef{\Bredon}{\myauthor{G.\,E.\,Bredon}}
\ndef{\Breuer}{\myauthor{M.\,Breuer}}
\ndef{\Brown}{\myauthor{L.\,G.\,Brown}}
\ndef{\Bruneau}{\myauthor{V.\,Bruneau}}
\ndef{\Buslaev}{\myauthor{V.\,S.\,Buslaev}}
\ndef{\Carey}{\myauthor{A.\,L.\,Carey}}
\ndef{\CareyRW}{\myauthor{R.\,W.\,Carey}} 
\ndef{\Cartan}{\myauthor{H.\,Cartan}}
\ndef{\Chilin}{\myauthor{V.\,I.\,Chilin}}
\ndef{\Coburn}{\myauthor{L.\,A.\,Coburn}}
\ndef{\Connes}{\myauthor{A.\,Connes}}
\ndef{\Cornfeld}{\myauthor{I.\,P.\,Cornfeld}}
\ndef{\Daletskii}{\myauthor{Yu.\,L.\,Daletski\u\i}}   
\ndef{\Dixmier}{\myauthor{J.\,Dixmier}}
\ndef{\DoddsPG}{\myauthor{P.\,G.\,Dodds}}
\ndef{\DoddsTK}{\myauthor{T.\,K.\,Dodds}}
\ndef{\Douglas}{\myauthor{R.\,G.\,Douglas}}
\ndef{\Dubrovin}{\myauthor{B.\,A.\,Dubrovin}}
\ndef{\Dugundji}{\myauthor{J.\,Dugundji}}
\ndef{\Duncan}{\myauthor{J.\,Duncan}}
\ndef{\Dunford}{\myauthor{N.\,Dunford}}
\ndef{\Dykema}{\myauthor{K.\,J.\,Dykema}}
\ndef{\Edwards}{\myauthor{R.\,E.\,Edwards}}
\ndef{\Eilenberg}{\myauthor{S.\,Eilenberg}}
\ndef{\Entina}{\myauthor{S.\,B.\,\`Entina}}
\ndef{\Fack}{\myauthor{T.\,Fack}} 
\ndef{\Faddeev}{\myauthor{L.\,D.\,Faddeev}}
\ndef{\Farber}{\myauthor{M.\,Farber}}
\ndef{\Farforovskaya}{\myauthor{Yu.\,B.\,Farforovskaya}}
\ndef{\Federer}{\myauthor{H.\,Federer}}
\ndef{\Fedosov}{\myauthor{B.\,V.\,Fedosov}}
\ndef{\Figiel}{\myauthor{T.\,Figiel}} 
\ndef{\Figueroa}{\myauthor{H.\,Figueroa}}
\ndef{\Fillmore}{\myauthor{P.\,A.\,Fillmore}}
\ndef{\Fomenko}{\myauthor{A.\,T.\,Fomenko}} 
\ndef{\Fomin}{\myauthor{S.\,V.\,Fomin}}
\ndef{\Frohlich}{\myauthor{J.\,Fr\"ohlich}}
\ndef{\Fuglede}{\myauthor{B.\,Fuglede}}
\ndef{\Furutani}{\myauthor{K.\,Furutani}}
\ndef{\Gelfand}{\myauthor{I.\,M.\,Gelfand}}
\ndef{\Gesztesy}{\myauthor{F.\,Gesztesy}}     
\ndef{\Getzler}{\myauthor{E.\,Getzler}} 
\ndef{\Gilkey}{\myauthor{P.\,B.\,Gilkey}}
\ndef{\Gitler}{\myauthor{S.\,Gitler}}
\ndef{\Glazman}{\myauthor{I.\,M.\,Glazman}}
\ndef{\Glimm}{\myauthor{J.\,Glimm}}
\ndef{\Gohberg}{\myauthor{I.\,C.\,Gohberg}}
\ndef{\Goldshtein}{\myauthor{Ya.\,Goldshtein}}
\ndef{\Golze}{\myauthor{F.\,Golze}}
\ndef{\GraciaBondia}{\myauthor{J.\,M.\,Gracia-Bond\'{i}a}}
\ndef{\Greenleaf}{\myauthor{F.\,P.\,Greenleaf}}
\ndef{\Gromov}{\myauthor{M.\,Gromov}}
\ndef{\Gunning}{\myauthor{R.\,C.\,Gunning}}
\ndef{\Haagerup}{\myauthor{U.\,Haagerup}}
\ndef{\Haag}{\myauthor{R.\,Haag}}
\ndef{\Halmos}{\myauthor{P.\,R.\,Halmos}}
\ndef{\Hardy}{\myauthor{G.\,H.\,Hardy}}
\ndef{\Herbst}{\myauthor{I.\,W.\,Herbst}}
\ndef{\Higson}{\myauthor{N.\,Higson}}  
\ndef{\Hoermander}{\myauthor{L.\,H\"ormander}} 
\ndef{\Hoffman}{\myauthor{K.\,Hoffman}} 
\ndef{\Ito}{\myauthor{K.\,Ito}}
\ndef{\Ikebe}{\myauthor{T.\,Ikebe}}
\ndef{\Jaffe}{\myauthor{A.\,Jaffe}}
\ndef{\James}{\myauthor{I.\,M.\,James}}
\ndef{\Javrjan}{\myauthor{V.\,A.\,Javrjan}}
\ndef{\Jitomirskaya}{\myauthor{S.\,Jitomirskaya}}
\ndef{\Kadison}{\myauthor{R.\,V.\,Kadison}}
\ndef{\Kalton}{\myauthor{N.\,J.\,Kalton}} 
\ndef{\Kato}{\myauthor{T.\,Kato}} 
\ndef{\Kobayashi}{\myauthor{S.\,Kobayashi}}
\ndef{\Koplienko}{\myauthor{L.\,S.\,Koplienko}}
\ndef{\Korotyaev}{\myauthor{E.\,Korotyaev}}
\ndef{\Kosaki}{\myauthor{H.\,Kosaki}}
\ndef{\Kostrykin}{\myauthor{V.\,Kostrykin}}
\ndef{\Kotani}{\myauthor{S.\,Kotani}}
\ndef{\Krein}{\myauthor{Kre\u\i n}}
\ndef{\KreinMG}{\myauthor{M.\,G.\,Kre\u\i n}}
\ndef{\KreinSG}{\myauthor{S.\,G.\,Kre\u\i n}}
\ndef{\Kuroda}{\myauthor{S.\,T.\,Kuroda}}
\ndef{\Leichtnam}{\myauthor{E.\,Leichtnam}}
\ndef{\Lesch}{\myauthor{M.\,Lesch}}
\ndef{\Lesniewski}{\myauthor{A.\,Lesniewski}}
\ndef{\Levitan}{\myauthor{B.\,M.\,Levitan}}
\ndef{\Lidskii}{\myauthor{V.\,B.\,Lidskii}}
\ndef{\Lifshitz}{\myauthor{I.\,M.\,Lifshitz}}
\ndef{\Lindenstrauss}{\myauthor{J.\,Lindenstrauss}}
\ndef{\Loday}{\myauthor{J.-L.\,Loday}}
\ndef{\Lord}{\myauthor{S.\,Lord}}      
\ndef{\Lorentz}{\myauthor{G.\,Lorentz}}
\ndef{\Magnus}{\myauthor{W.\,Magnus}}
\ndef{\Makarov}{\myauthor{K.\,A.\,Makarov}}
\ndef{\MakarovN}{\myauthor{N.\,Makarov}}
\ndef{\Mathai}{\myauthor{V.\,Mathai}}         
\ndef{\McKean}{\myauthor{H.\,P.\,McKean}}
\ndef{\Mishchenko}{\myauthor{A.\,S.\,Mishchenko}}
\ndef{\Molchanov}{\myauthor{S.\,A.\,Molchanov}}
\ndef{\Moore}{\myauthor{C.\,C.\,Moore}}
\ndef{\Moscovici}{\myauthor{H.\,Moscovici}}  
\ndef{\Motovilov}{\myauthor{A.\,K.\,Motovilov}}
\ndef{\Moyer}{\myauthor{R.\,D.\,Moyer}}
\ndef{\Naboko}{\myauthor{S.\,N.\,Naboko}}
\ndef{\Narasimhan}{\myauthor{R.\,Narasimhan}}
\ndef{\Nomizu}{\myauthor{K.\,Nomizu}}
\ndef{\Novikov}{\myauthor{S.\,P.\,Novikov}}
\ndef{\Osterwalder}{\myauthor{K.\,Osterwalder}}
\ndef{\Patodi}{\myauthor{V.\,Patodi}}
\ndef{\Pagter}{\myauthor{B.\,de~Pagter}}  
\ndef{\Pastur}{\myauthor{L.\,A.\,Pastur}}  
\ndef{\Pavlov}{\myauthor{B.\,S.\,Pavlov}}
\ndef{\Pedersen}{\myauthor{G.\,K.\,Pedersen}}
\ndef{\Peller}{\myauthor{V.\,V.\,Peller}}
\ndef{\Perera}{\myauthor{V.\,S.\,Perera}}
\ndef{\Petunin}{\myauthor{Ju.\,I.\,Petunin}}
\ndef{\Phillips}{\myauthor{J.\,Phillips}}  
\ndef{\Piazza}{\myauthor{P.\,Piazza}}   
\ndef{\Pincus}{\myauthor{J.\,D.\,Pincus}}   
\ndef{\Poincare}{Poincar\'e}
\ndef{\Postnikov}{\myauthor{M.\,M.\,Postnikov}} 
\ndef{\Povzner}{\myauthor{A.\,Ya.\,Povzner}}
\ndef{\Prinzis}{\myauthor{R.\,Prinzis}}
\ndef{\Privalov}{\myauthor{I.\,I.\,Privalov}}
\ndef{\Pushnitski}{\myauthor{A.\,B.\,Pushnitski}} 
\ndef{\Raeburn}{\myauthor{I.\,Raeburn}}
\ndef{\Raikov}{\myauthor{G.\,Raikov}}
\ndef{\Reed}{\myauthor{M.\,Reed}}
\ndef{\Rennie}{\myauthor{A.\,Rennie}}
\ndef{\Rickart}{\myauthor{C.\,E.\,Rickart}}
\ndef{\Riesz}{\myauthor{F.\,Riesz}}
\ndef{\Ringrose}{\myauthor{J.\,Ringrose}}
\ndef{\Rio}{\myauthor{R.\,del Rio}}
\ndef{\Robinson}{\myauthor{D.\,Robinson}}
\ndef{\Rossi}{\myauthor{H.\,Rossi}}
\ndef{\Rudin}{\myauthor{W.\,Rudin}}
\ndef{\Ruelle}{\myauthor{D.\,Ruelle}}
\ndef{\Ruzhansky}{\myauthor{M.\,Ruzhansky}}
\ndef{\Sakai}{\myauthor{Sh.\,Sakai}}
\ndef{\Sargsjan}{\myauthor{I.\,S.\,Sargsjan}}
\ndef{\Sato}{\myauthor{H.\,Sato}}
\ndef{\Schaeffer}{\myauthor{D.\,G.\,Schaeffer}}
\ndef{\Schluchtermann}{\myauthor{G.\,Schluchtermann}}
\ndef{\Schochet}{\myauthor{C.\,Schochet}}
\ndef{\SchroedingerE}{\myauthor{E.\,Schr\"odinger}}
\ndef{\Schroedinger}{\myauthor{Schr\"odinger}}
\ndef{\Schrohe}{\myauthor{E.\,Schrohe}}
\ndef{\Schwartz}{\myauthor{J.\,T.\,Schwartz}}
\ndef{\Sedaev}{\myauthor{A.\,A.\,Sedaev}}
\ndef{\Seiler}{\myauthor{R.\,Seiler}}
\ndef{\Semenov}{\myauthor{E.\,M.\,Semenov}}
\ndef{\Shabat}{\myauthor{B.\,V.\,Shabat}}
\ndef{\Shafarevich}{\myauthor{I.\,R.\,Shafarevich}}
\ndef{\Sharpley}{\myauthor{R.\,Sharpley}}
\ndef{\Shilov}{\myauthor{G.\,E.\,Shilov}}
\ndef{\Shirkov}{\myauthor{D.\,V.\,Shirkov}}
\ndef{\Shubin}{\myauthor{M.\,A.\,Shubin}}
\ndef{\Silverman}{\myauthor{H.\,Silverman}}
\ndef{\Simon}{\myauthor{B.\,Simon}}
\ndef{\Sinai}{\myauthor{Ya.\,G.\,Sinai}}
\ndef{\Singer}{\myauthor{I.\,M.\,Singer}}
\ndef{\Solomyak}{\myauthor{M.\,Z.\,Solomyak}}
\ndef{\Soloviev}{\myauthor{Yu.\,P.\,Soloviev}}
\ndef{\Spivak}{\myauthor{M.\,Spivak}}
\ndef{\Stein}{\myauthor{E.\,M.\,Stein}}
\ndef{\Stenkin}{\myauthor{V.\,V.\,Sten'kin}}
\ndef{\Stratila}{\myauthor{S.\,Stratila}}
\ndef{\Sucheston}{\myauthor{L.\,Sucheston}}
\ndef{\Sukochev}{\myauthor{F.\,A.\,Sukochev}}
\ndef{\Switzer}{\myauthor{R.\,M.\,Switzer}}
\ndef{\SzNagy}{\myauthor{B.\,Sz.-Nagy}}
\ndef{\Takesaki}{\myauthor{M.\,Takesaki}}
\ndef{\Taylor}{\myauthor{M.\,E.\,Taylor}}
\ndef{\Treves}{\myauthor{F.\,Treves}}
\ndef{\Troitsky}{\myauthor{E.\,V.\,Troitsky}}
\ndef{\Tzafriri}{\myauthor{L.\,Tzafriri}}
\ndef{\Varilly}{\myauthor{J.\,C.\,V\'{a}rilly}}
\ndef{\Vergne}{\myauthor{M.\,Vergne}}
\ndef{\Vladimirov}{\myauthor{V.\,S.\,Vladimirov}}
\ndef{\Voiculescu}{\myauthor{D.\,Voiculescu}}
\ndef{\Weiss}{\myauthor{G.\,Weiss}}
\ndef{\Wells}{\myauthor{R.\,O.\,Wells}}
\ndef{\Williams}{\myauthor{J.\,P.\,Williams}}
\ndef{\Winkler}{\myauthor{S.\,Winkler}}
\ndef{\Witten}{\myauthor{E.\,Witten}}
\ndef{\Wodzicki}{\myauthor{M.\,Wodzicki}}
\ndef{\Wojciechowski}{\myauthor{K.\,P.\,Wojciechowski}}
\ndef{\Yafaev}{\myauthor{D.\,R.\,Yafaev}}
\ndef{\Yosida}{\myauthor{K.\,Yosida}}
\ndef{\Zsido}{\myauthor{L.\,Zsido}}

\numberwithin{equation}{section}

\newcommand{\twovector}[2]{\left(\begin{matrix} #1 \\ #2 \end{matrix} \right)}
\newcommand{\catn}{\divideontimes}
\newcommand{\rlmb}{r_\lambda}

\newcommand{\aaa}{\mathfrak{a}}

\newcommand{\Rset}{\euR(\lambda)}

\ndef{\hlambda}{{\mathfrak h_\lambda}}
\rndef{\iff}{\Leftrightarrow} 
\begin{document}
\title{Spectral flow and resonance index}
\author{Nurulla Azamov}
\address{School of Computer Science, Engineering and Mathematics
   \\ Flinders University Tonsley
   \\ 5042 Clovelly Park SA Australia}
\email{nurulla.azamov@flinders.edu.au}

\keywords{spectral flow, resonance index, Robbin-Salamon axioms}
 \subjclass[2000]{ 
     Primary 47A40, 47A55 
 }
\begin{abstract}
It has been shown recently that spectral flow admits a natural integer-valued
extension to essential spectrum. This extension admits four different interpretations;
two of them are singular spectral shift function and total resonance index.
In this work we study resonance index outside essential spectrum.

Among results of this paper are the following.

1. Total resonance index satisfies Robbin-Salamon axioms for spectral flow.

2. Direct proof of equality ``total resonance index = intersection number''.

3. Direct proof of equality ``total resonance index = total Fredholm index''.

4. (a) Criteria for a perturbation~$V$ to be tangent to the~resonance set at a point~$H,$
where the resonance set is the infinite-dimensional variety of self-adjoint perturbations of the initial 
self-adjoint operator~$H_0$ which have~$\lambda$ as an eigenvalue. 
(b) Criteria for the order of tangency of a perturbation~$V$ to the resonance set.

5. Investigation of the root space of the compact operator $(H_0+sV-\lambda)^{-1}V$
corresponding to an eigenvalue $(s-r_\lambda)^{-1},$ where $H_0+r_\lambda V$ is a point of the resonance set. 

This analysis gives a finer information about behaviour of discrete spectrum compared to spectral flow. 

Finally, many results of this paper are non-trivial even in finite dimensions, in which case they can be 
and were tested in numerical experiments. 
\end{abstract}
\begin{center} 
   \tiny \sf arXiv version, \ v\,17.6
\end{center}
\maketitle

\tableofcontents

\section{Introduction}

\subsection{Introduction}
The spectral flow of a continuous path of self-adjoint operators
$\set{H_r\colon r \in [0,1]}$ through a point~$\lambda$ which does
not belong to the common essential spectrum $\sigma_{ess}$ of
operators~$H_r$ is naively understood as the number of eigenvalues
of $H_r$ which cross~$\lambda$ from left to right minus the number
of eigenvalues of~$H_r$ which cross~$\lambda$ from right to left
as the variable~$r$ moves from~$0$ to~$1$ \cite{APS76}. This
naive definition was given a rigorous basis in \cite{RoSa}. The
spectral flow can also be defined as total Fredholm index of the 
path $\set{H_r},$ \cite{Ph96CMB,Ph97FIC}, \cite[Section 4]{BCPRSW}. Further,
the total Fredholm index can be interpreted as an integral of a
one-form defined on some real affine space of
self-adjoint operators \cite{Ge93Top,CP98CJM,CP2,ACS,BCPRSW}, 
the possibility of such interpretation of spectral flow was first 
suggested by \Singer\ in 1974. The spectral flow can also be interpreted as
Maslov index \cite{RoSa}. Finally, outside the
essential spectrum the spectral flow is equal to the spectral shift function
\cite{Li52UMN,Kr53MS}, see e.g. \cite{ACDS,ACS,Pu08AMST}.

Any of these definitions of spectral flow is applicable
only for numbers~$\lambda$ outside the common essential spectrum
of a path of self-adjoint operators~$H_r,$
with the exception of the spectral shift function. The
spectral shift function is not integer-valued inside the essential
spectrum and therefore it cannot be considered as a proper
analogue of spectral flow for essential spectrum. In
\cite{Az,Az2,Az3v6} an analogue of Lebesgue decomposition
$m=m^{(a)}+m^{(s)}$ of a measure~$m$ into its absolutely
continuous and singular parts was suggested for the spectral shift
function~$\xi.$ It was shown that the singular part $\xis$ of the spectral shift function~$\xi,$
which can be correctly defined by formula 
\begin{equation} \label{F: def-n of xis}
  \xis(\lambda) = \frac d{d\lambda} \int_0^1 \Tr(V E_\lambda^{H^{(s)}_r})\,dr, \ \text{a.e.} \ \lambda,
\end{equation}
where $H_r^{(s)}$ is the singular part of $H_r = H_0+r V,$
is a function which takes integer values for a.e. value of the spectral parameter~$\lambda,$
including those in $\sigma_{ess},$ and
which coincides with the spectral shift function~$\xi,$ and thus with spectral flow, outside the essential spectrum.
Apparently, it is difficult to work directly with the definition (\ref{F: def-n of xis})
of the singular spectral shift function~$\xis.$ 
In \cite{Az7} (see also \cite[Section 6]{Az9}) it was found that for trace 
class perturbations the singular spectral shift function
can be interpreted as total resonance index (TRI),\label{Page: total res index} 
\begin{equation} \label{F: xis=tot res index}
  \xis(\lambda) = \sum_{{r_\lambda} \in [0,1]} \ind_{res}(\lambda; H_{r_\lambda}, V),
\end{equation}
where $\ind_{res}(\lambda; H_{r_\lambda}, V)$ is the so-called \emph{resonance index} of the triple 
$(\lambda; H_{r_\lambda}, V),$ and where the sum is taken over real \emph{resonance points}~$r_\lambda$ of the triple 
$(\lambda; H_{0}, V),$ which belong to $[0,1].$ 
One of the ways to define the resonance points and the resonance index is as follows.
Let 
$$ 
  \sigma_z^1(s), \ \sigma_z^2(s), \ \sigma_z^3(s), \ \ldots
$$ 
be the list of eigenvalues of the compact operator 
$$
  (H_0+s V-z)^{-1}V, \ s \in \mbR.
$$ 
It is not difficult to show that 
for each of these eigenvalues $\sigma_z^j(s)$ there exists a number $r_z^j$ such that 
$$
  \sigma_z^j(s) = (s - r_z^j)^{-1}.
$$
The numbers $r_z^j$ are resonance points of the triple $(z,H_{0},V).$ 
A real number~$r_\lambda$ is a real resonance point of the triple 
$(\lambda,H_{0},V),$ if at least one of the resonance points $r_z^j$ approaches~$r_\lambda$
as $z = \lambda+iy \to \lambda+i0.$ The resonance index\label{Page: res index} of the triple 
$(\lambda,H_{\rlmb},V)$ is the integer 
$$
  N_+-N_-,
$$ 
where $N_+$ (respectively, $N_-$), 
is the number of resonance points which approach~$r_\lambda$ in the upper complex half-plane 
(respectively, lower complex half-plane). The resonance points $r_z^j$ in this definition should be 
counted according to their algebraic multiplicities, which are transferred 
from the corresponding eigenvalues~$\sigma_z^j(s).$ 

\smallskip
If~$\lambda$ does not belong to the essential spectrum, then $\xis(\lambda) = \xi(\lambda),$
as the definition (\ref{F: def-n of xis}) of the singular SSF turns into a well-known 
Birman-Solomyak formula for SSF. Thus, outside the essential spectrum 
the formula (\ref{F: xis=tot res index})
turns into 
$$
  \text{spectral flow through} \ \lambda = \sum_{{r_\lambda} \in [0,1]} \ind_{res}(\lambda; H_{r_\lambda}, V).
$$
Once this formula is obtained, one may choose to forget its origin and consider the right hand side
as a new definition of spectral flow. Apart from the fact that, unlike other definitions of spectral flow, 
this definition makes perfect sense inside the essential spectrum, it has two other advantages. 
Firstly, it requires minimum assumptions in order to be defined, in particular, 
it includes as special cases the spectral flow for operators with compact resolvent and spectral shift function
for relatively trace class perturbations. Secondly, it is defined in the language of complex analysis
and it can be investigated using tools of complex analysis. It is well known that proofs based on 
complex analysis are as a rule considerably simpler (and also more beautiful, but this depends 
on one's taste) and therefore more natural. As a historical example, Franz Rellich's perturbation 
theory of isolated eigenvalues was essentially simplified by introduction of Riesz idempotents. 

In this paper we study spectral flow from the point of view of resonance index. 

\subsection{Basic assumption}

In the following Assumption we collect basic assumptions, notation
and terminology  which will be assumed and used throughout this paper.
\begin{assump} \label{A: Main Assumption} \rm
  \mbox{ }
  \begin{enumerate}
  \item $H_0$ is a self-adjoint operator on a separable complex Hilbert 
  space~$\hilb$ with dense domain~$\euD.$
  \item \label{AI: 2}~$\clA_0$ is a real vector space of self-adjoint \label{Page: clA0}
  operators~$V$ which are relatively compact with respect to~$H_0.$
  The last means by definition that domain of~$V$ contains $\euD$ and the product
  $$
    R_z(H_0)V := (H_0-z)^{-1}V
  $$
  is compact for some and thus for any complex number~$z$ which
  does not belong to the spectrum of~$H_0.$ Elements of the real
  vector space~$\clA_0$ will also be called \emph{directions.}
  \item~$\clA$ is the real affine space $H_0+\clA_0$ of \label{Page: clA}
  self-adjoint operators. Elements of the real
  affine space~$\clA$ will also be called \emph{points.}
  \item It follows from (\ref{AI: 2}) and the second resolvent identity, that all directions from~$\clA_0$ are
  relatively compact with respect to any point from~$\clA.$
  Therefore, by Weyl's theorem (see e.g. \cite{RS4}) all points from~$\clA$ have the same
  essential spectrum. This common essential spectrum we shall denote by~$\sigma_{ess}$
  and refer to it as the essential spectrum of the affine space~$\clA.$
  \item There exists at least one real number~$\lambda$  
  which does not belong to the essential spectrum~$\sigma_{ess}.$
  Most of the time, the real number~$\lambda$ will be fixed.
  \end{enumerate}
\end{assump}

These are the only assumptions which we shall make in this paper.
These assumptions are quite generic in Hilbert space perturbation theory. 
As a special case they include the case of self-adjoint operators~$H_{0}$ with compact resolvent
and a vector space of bounded perturbations~$V,$ which is the main setting of spectral flow theory 
in differential geometry and global analysis. 

We shall consistently use the words ``point'' and ``direction'' instead of self-adjoint operator
from~$\clA$ and a self-adjoint perturbation operator from~$\clA_0.$
Elements of a vector space associated with an affine space are usually called vectors,
but we shall not use this word for this purpose to avoid confusion. 

\smallskip
The set of all points~$H$ from~$\clA$ for which~$\lambda$
is an eigenvalue we call \emph{resonance set} and denote it by $\Rset.$
Elements of the set $\Rset$ will be called \emph{$\lambda$-resonant operators}\label{Page: resonance point}
or \emph{$\lambda$-resonant points.} Since the real number~$\lambda$ will be fixed for most 
of the time,~$\lambda$-resonant points will often be called resonant points.
We say that a resonance point~$H$ is \emph{simple},\label{Page: simple point} if~$\lambda$ is an
eigenvalue of~$H$ of geometric multiplicity~1. 
An analytic path $H(s)$ will be said to be a \emph{resonant path}\label{Page: H(s) 2}
\label{Page: resonant path}
if $H(s) \in \Rset$ for all $s.$ Otherwise we say that $H(s)$ is a \emph{regular path}.
\label{Page: regular path}
A regular path $H(s)$ may have resonant points on it, but the set of such points is discrete.
A~direction~$V$ at a resonant point~$H$ will be said to be \emph{regular},
\label{Page: regular direction} if the straight line $H+sV$ is a regular path.
In my previous papers, whether published or not, regular directions were called regularising.

\smallskip
The terminology ``$\lambda$-resonant operator'' was used in \cite{Az9}
in a study of spectral flow inside essential spectrum $\sigma_{ess},$ but for a real number~$\lambda$
outside the essential spectrum the definition of~$\lambda$-resonant operator reduces to the one given above.

\smallskip
A function defined on a finite-dimensional real affine space is called \emph{analytic}
if it is given by a real analytic function in some and thus in any affine system of coordinates.
A subset~$R$ of a finite-dimensional real affine space is called an \emph{analytic set}
if~$R$ is the set of zeros of one or several real analytic functions.
A subset~$R$ of any real affine space is called an \emph{analytic set}
if the intersection of~$R$ with any finite-dimensional affine subspace is an analytic set.
By an affine space from now on we will always mean a real affine space.
Finite subsets of an affine space and an affine space itself are analytic sets.
Intersections and finite unions of analytic sets are also analytic.
The resonance set~$\Rset$ is an analytic set, proof of this assertion follows 
verbatim that of \cite[Theorem 4.2.5]{Az3v6}.

\subsection{Preliminaries}
The second resolvent identity
\begin{equation} \label{F: second resolvent identity (1)}
    R_z(H_0+V) - R_z(H_0) = - R_z(H_0+V)V R_z(H_0)
                              = - R_z(H_0)V R_z(H_0+V).
\end{equation}
holds for any pair of closed operators~$H_0$ and~$V$ 
provided the sum~$H_0+V$ is well-defined and~$z$ belongs 
to the resolvent sets of both operators~$H_0$ and~$H_0+V.$
This identity can be rewritten as follows:
\begin{equation} \label{F: second resolvent identity (2)}
     R_z(H_0+V) = R_z(H_0)\brs{1+V R_z(H_0)}^{-1}  
                  = \brs{1+R_z(H_0)V}^{-1}R_z(H_0).
\end{equation}

Let
\begin{equation*}
  A_z(s) = R_z(H_s)V \quad \text{and} \quad B_z(s) = V R_z(H_s),
\end{equation*}
where 
$$
  H_s = H_0 + sV.
$$
By Assumption~\ref{A: Main Assumption}, the operators $A_z(s)$ and $B_z(s)$ are compact.
The second resolvent identity (\ref{F: second resolvent identity (2)}) implies that 
for $s,r \in \mbC$
\begin{equation} \label{F: second resolvent identity (3)}
     A_z(s) = \brs{1+(s-r)A_z(r)}^{-1}A_z(r),
\end{equation}
and a similar equality holds for $B_z(s).$ This equality shows that $A_z(s)$ is a meromorphic 
function of~$s.$

\smallskip 
We start with a recap of some material of \cite[Section 3]{Az9}.

Let $H_0 \in \clA,$~$V \in \clA_0$ and let $H_s = H_0 + sV,$ where $s \in \mbC.$
Let~$z = \lambda + iy \in \mbC \setminus \sigma_{ess}.$ 
A point~$r_z$ is a \emph{resonance point}\label{Page: res point} of the triple $(z; H_0,V),$ 
if any one of the following equivalent conditions hold:
\begin{enumerate}
  \item~$r_z$ is a pole of the meromorphic function $\mbC \ni s \mapsto A_z(s).$ 
  \item~$r_z$ is a pole of the meromorphic function $\mbC \ni s \mapsto B_z(s).$
  \item The operator $1+(r_z-s)A_z(s)$ has a non-zero kernel for some $s \in \mbC.$
   The kernel $\Upsilon^1_z(r_z; H_0,V)$ of this operator does not depend on~$s.$
  \item The operator $1+(r_z-s)B_z(s)$ has a non-zero kernel. 
        The kernel $\Psi^1_z(r_z; H_0,V)$ of this operator also does not depend on~$s.$
  \item The number~$z$ is an eigenvalue of the operator $H_0+r_z V =: H_{r_z}.$
\end{enumerate}

With every resonance point~$r_z$ of the triple~$(z; H_{0},V),$
one can associate an idempotent operator~$P_z(r_z),$
the vector space~$\Upsilon_z(r_z) := \im P_z(r_z)$\label{Page Upsilon z rz} and a nilpotent operator~$\bfA_z(r_z),$
which is reduced by the vector space~$\Upsilon_z(r_z).$ They can be defined by formulas\label{Page: Pz(rz)}
\begin{equation} \label{F: Pz=oint A(s)}
  P_z(r_z) = \frac 1{2\pi i} \oint _{C_{r_z}} A_z(s)\,ds 
\end{equation}
and \label{Page: bfA(z)}
\begin{equation} \label{F: A(lamb)=oint s A(s)}
  \bfA_z(r_z) = \frac 1{2\pi i} \oint _{C_{r_z}} (s-r_z) A_z(s)\,ds,
\end{equation}
where $C_{r_z}$ is a contour encircling~$r_z$ and no other resonance points. 
Similarly one defines an idempotent operator $Q_z(r_z)$ and a nilpotent operator $\bfB_z(r_z),$
by replacing $A_z(s)$ in (\ref{F: Pz=oint A(s)}) and (\ref{F: A(lamb)=oint s A(s)}) by $B_z(s).$ 
For these operators we have 
$$
  \bfB_z(r_z) V = V \bfA_z(r_z) \quad \text{and} \quad \brs{\bfA_z(r_z)}^* = \bfB_{\bar z}(\bar r_z).
$$
\noindent 
The Laurent expansion of the meromorphic function $A_z(s)$ at a resonance point~$r_z$ 
has the following form \cite[(3.3.16)]{Az9}:
\begin{equation} \label{F: Az (3.3.16)}
  A_z(s) = \tilde A_z(s) + (s-r_z)^{-1}P_z(r_z) + (s-r_z)^{-2}\bfA_z(r_z) +  
     \ldots + (s-r_z)^{-d}\bfA^{d-1}_z(r_z),
\end{equation}
where $d$ is the order of the resonance point~$r_z$ (see below) and $\tilde A_z(s)$ is the holomorphic part 
of the Laurent series. 
The idempotent operators $P_z(r_z)$ and $Q_z(r_z)$ have the following properties~\cite{Az9}. 
If~$r_z^{(1)}$ and~$r_z^{(2)}$ are two different resonance points corresponding to~$z,$ 
then 
\begin{equation} \label{F: Pz(1)Pz(2)=0}
  P_z(r_z^{(1)})P_z(r_z^{(2)})=0.
\end{equation}
Further, 
\begin{equation} \label{F: P*=Q}
  P^*_z(r_z) = Q_{\bar z}(\bar r_z),
\end{equation}
and 
\begin{equation} \label{F: VP = QV}
  V P_z(r_z) = Q_z(r_z) V.
\end{equation}
The operator $Q_{\bar z}(\bar r_z)V P_z(r_z)$
is a finite-rank self-adjoint operator which in \cite{Az9} is called \emph{resonance matrix}.
If~$\lambda$ and $r_\lambda$ are real and $\lambda \notin \sigma_{ess},$ then 
$$
  Q_{\lambda}(r_\lambda)V P_\lambda(r_\lambda) = V P_\lambda(r_\lambda).
$$

The vector space $\Upsilon_z(r_z)$ consists of all vectors~$\chi$ such that for some positive integer~$k$
\begin{equation} \label{F: res eq-n of order k}
  (1+(r_z-s)A_z(s))^k \chi = 0,
\end{equation} 
where~$s$ is any number which is not a pole of $A_z(s).$ 
This definition does not depend on the choice of~$s.$ A vector from $\Upsilon_z(r_z)$ will be called 
a \emph{resonance vector}. \label{Page: res vector}
The equality (\ref{F: res eq-n of order k}) will be called \emph{resonance equation} of order~$k.$
The smallest integer~$k$ such that the equality (\ref{F: res eq-n of order k})
holds for some (and thus for any)~$s$ will be called the \emph{order} of~$\chi.$ 
\label{Page: order of res vector} 
The vector space of resonance vectors of order $k$ we denote $\Upsilon^k_z(r_z).$
The operator $\bfA_z(r_z)$ maps $\Upsilon^k_z(r_z)$ onto $\Upsilon^{k-1}_z(r_z),$
that is, $\bfA_z(r_z)$ lowers the order of a resonance vector by~$1.$
A resonance vector $\chi$ has \emph{depth} at least~$k,$ 
if $\chi \in \im \bfA^k_z(r_z).$\label{Page: depth of vector}

Dimensions of the vector spaces $\Upsilon_z(r_z)$ 
and $\Upsilon^1_z(r_z)$ we denote by~$N$ and~$m$ respectively. 
The vector space $\Upsilon^1_z(r_z)$ is the eigenspace of $H_{r_z} = H_{0} + r_z V$ corresponding to 
eigenvalue~$z,$ and we also denote it by $\clV_z(r_z)$ or $\clV_z$ if there is no danger of confusion. 

The largest of positive integers $d$ such that $\Upsilon_z(r_z) = \Upsilon^d_z(r_z)$ 
will be called the \emph{order}
of the perturbation $V$ at $H_{r_z}$ and the order of the resonance point~$r_z.$
A regular direction~$V$ is said to be \emph{simple}\label{Page: simple direction} 
at a~$\lambda$-resonance point~$H,$ if~$V$ has order~$1.$

\medskip
Jordan decomposition of the nilpotent operator $\bfA_z(r_z)$ consists of~$m$ Jordan blocks;
we use lower case Greek letters~$\nu$ and~$\mu$ to enumerate them. The size of~$\nu$th block 
we denote~$d_\nu$ and assume that $d_1 \geq d_2 \geq \ldots \geq d_m.$ 
A basis 
$$
  \chi^{(j)}_\nu, \quad \nu=1,\ldots,m, \ j=0,1,\ldots,d_\nu-1
$$ 
of $\Upsilon_z(r_z)$ is a \emph{Jordan basis} if 
$$
  \bfA_z(r_z) \chi^{(j)}_\nu = \chi^{(j-1)}_\nu,
$$
where it is assumed that~$\chi^{(-1)}_\nu=0.$ 
In particular, $\ker (\bfA_z(r_z)) \cap \Upsilon_z(r_z) = \Upsilon^1_z(r_z) = \clV_z.$ 
A Jordan basis can be depicted 
by either of the two Young diagrams shown next: 

\smallskip
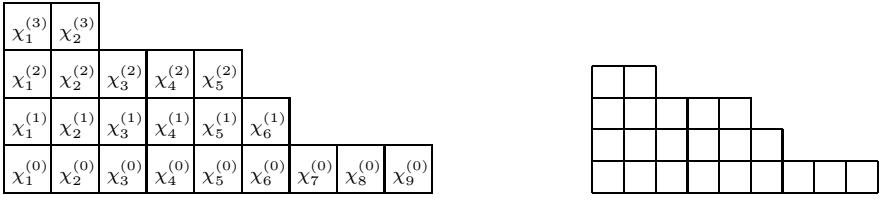
\begin{picture}(162,72)
\put(0,0){\line(0,1){72}}\put(18,0){\line(0,1){72}} \put(36,0){\line(0,1){72}} \put(54,0){\line(0,1){54}} 
\put(72,0){\line(0,1){54}} \put(90,0){\line(0,1){54}} \put(108,0){\line(0,1){36}} \put(126,0){\line(0,1){18}} 
\put(144,0){\line(0,1){18}} \put(162,0){\line(0,1){18}} 
\put(0,0){\line(1,0){162}}\put(0,18){\line(1,0){162}}\put(0,36){\line(1,0){108}}\put(0,54){\line(1,0){90}}
\put(0,72){\line(1,0){36}}
\put(3,5){\tiny $\chi_{1}^{(0)}$}\put(3,23){\tiny $\chi_{1}^{(1)}$}\put(3,41){\tiny $\chi_{1}^{(2)}$}\put(3,59){\tiny $\chi_{1}^{(3)}$}
\put(21,5){\tiny $\chi_{2}^{(0)}$}\put(21,23){\tiny $\chi_{2}^{(1)}$}\put(21,41){\tiny $\chi_{2}^{(2)}$}\put(21,59){\tiny $\chi_{2}^{(3)}$}
\put(39,5){\tiny $\chi_{3}^{(0)}$}\put(39,23){\tiny $\chi_{3}^{(1)}$}\put(39,41){\tiny $\chi_{3}^{(2)}$}
\put(57,5){\tiny $\chi_{4}^{(0)}$}\put(57,23){\tiny $\chi_{4}^{(1)}$}\put(57,41){\tiny $\chi_{4}^{(2)}$}
\put(75,5){\tiny $\chi_{5}^{(0)}$}\put(75,23){\tiny $\chi_{5}^{(1)}$}\put(75,41){\tiny $\chi_{5}^{(2)}$}
\put(93,5){\tiny $\chi_{6}^{(0)}$}\put(93,23){\tiny $\chi_{6}^{(1)}$}
\put(111,5){\tiny $\chi_{7}^{(0)}$}
\put(129,5){\tiny $\chi_{8}^{(0)}$}
\put(147,5){\tiny $\chi_{9}^{(0)}$}
\end{picture}
\hskip 2 cm  
\begin{picture}(108,48)
\put(0,0){\line(0,1){48}}\put(12,0){\line(0,1){48}} \put(24,0){\line(0,1){48}} \put(36,0){\line(0,1){36}} 
\put(48,0){\line(0,1){36}} \put(60,0){\line(0,1){36}} \put(72,0){\line(0,1){24}} \put(84,0){\line(0,1){12}} 
\put(96,0){\line(0,1){12}} \put(108,0){\line(0,1){12}} 
\put(0,0){\line(1,0){108}}\put(0,12){\line(1,0){108}}\put(0,24){\line(1,0){72}}\put(0,36){\line(1,0){60}}
\put(0,48){\line(1,0){24}}
\end{picture}

\smallskip
In such diagram each square represents a resonance vector from a Jordan basis,
and the height of the square is the order of the vector. The number of squares is~$N,$
the width is~$m,$ and the height is~$d.$ Each Jordan basis defines a direct sum decomposition 
of the resonance vector space $\Upsilon_z(r_z),$ 
the~$\nu$th summand of which we denote by $\Upsilon^{[\nu]}_z(r_z).$
Thus,
$$
  \Upsilon_z(r_z) = \Upsilon^{[1]}_z(r_z) \dotplus \ldots \dotplus \Upsilon^{[m]}_z(r_z).
$$
The height of~$\nu$th column in the Young diagram is $d_\nu = \dim \Upsilon^{[\nu]}_z(r_z).$

The vector space~$\Upsilon^1_z(z; H_0,V)$ 
depends only on the operator $H_{r_z} = H_0+r_z V$ and does not depend on~$V,$
but the vector spaces $\Upsilon^k_z(r_z; H_0,V), \ k\geq 2,$ depend on both~$H_{r_z}$ and~$V.$

\smallskip
A complex number $r_z$ is a resonance point iff 
$$
  \sigma_z(s) = (s-r_z)^{-1}
$$
is an eigenvalue of $A_z(s).$ 
The eigenvalue $\sigma_z(s)$ has algebraic multiplicity~$N$ and geometric multiplicity~$m.$ 

Though~$z$ can be any complex number outside the essential spectrum, we are mainly interested 
in the case where~$z$ and the corresponding resonance point~$r_z$ are real numbers. 
In this case, if a real number~$\lambda$ is shifted to~$\lambda+iy$
with small $y>0,$ then the real eigenvalue $\sigma_\lambda(s)$ of
$A_\lambda(s)$ splits into $N_+$ and $N_-$ (where $N_\pm \geq 0,$ $N_+ + N_-\geq 1$)
eigenvalues in~$\mbC_+$ and~$\mbC_-$ respectively, and all 
shifted eigenvalues are non-real.
Resonance index of the triple $(\lambda; H_{r_\lambda},V)$ is by definition the difference
$$
  \ind_{res}(\lambda; H_{r_\lambda},V) = N_+-N_-.
$$

The objects such as $P_z(r_z),$ $\Upsilon_z(r_z),$ etc, depend on $H_0$ and $V$ too,
but since for the most part the operators $H_0$ and $V$ are fixed, 
usually we do not indicate this dependence. If necessary we write 
$P_z(r_z; H_0,V),$ etc, or $P_z(H_{r_z},V),$ etc, where $H_{r_z}=H_0+r_z V.$ The notation, such as
$P_z(H_{r_z},V)$ is not ambiguous, since $P_z(H_{r_z},V)$ depends on operators $H_{r_z}$ and $V$ but not on $H_0.$ 

\subsection{Description of results}

\subsubsection{Section~\ref{S: 2x2}}
Let $H_{\rlmb}$ be a resonance point, let $V$ be a regular direction
and let $H_s = H_{\rlmb} + (s-\rlmb)V.$ 
We use the following notation:
$\clV_\lambda$ is the eigenspace of $H_{\rlmb}$ corresponding to eigenvalue~$\lambda,$
$\hat \hilb = \clV_\lambda^\perp,$ 
$\hat P$ is the orthogonal projection onto $\hat \hilb,$
$\hat H_s = \hat P H_s \hat P,$
$\hat V = \hat P V \hat P,$
$v = \hat P V \hat P^\perp,$
$R_\lambda(\hat H_s)$ is the resolvent 
of $\hat H_s$ on $\clV_\lambda^\perp$ and zero on $\clV_\lambda,$
$S_\lambda$ is the operator $R_\lambda(\hat H_{\rlmb})V,$ 
$\hat A_\lambda(s) = R_\lambda(\hat H_s)\hat V.$

\begin{thm} (Theorem~\ref{T: thm 2.4}) 
Let $k\geq 2.$ If $\hat \phi \in \hat \hilb$ is a resonance vector of order $k$ 
then $\hat \phi$ belongs to the linear subspace 
$$
  \im R_\lambda(\hat H_{\rlmb})v \dotplus \im \hat A_\lambda(\rlmb)R_\lambda(\hat H_{\rlmb})v
   \dotplus \ldots \dotplus \im \hat A^{k-2}_\lambda(\rlmb)R_\lambda(\hat H_{\rlmb})v.
$$
\end{thm}  

The following theorem provides a criterion for a resonance vector to have depth at least 1.
This criterion is used several times in the remaining sections.
\begin{thm} (Theorem~\ref{T: TFAE depth 1 criterion}) 
For a resonance vector~$\chi$ the following three assertions are equivalent: 
(i) $V \chi \perp \clV_\lambda,$ (ii) $\chi$ has depth at least 1 and (iii)
$\bfA_\lambda(r_\lambda) S_\lambda \chi = -\chi.$ In particular, if 
$V \chi \perp \clV_\lambda,$ then $S_\lambda \chi$ is a resonance vector. 
\end{thm}  

The operator $S_\lambda$ satisfies the following equality.
\begin{thm} (Theorem \ref{T: -S A(j)=A(j-1)}) 
$$
  - S_\lambda \bfA_\lambda^j(r_\lambda) = \hat P \bfA^{j-1}_\lambda(r_\lambda). 
$$
\end{thm}
The last two theorems show that the operator $-S_\lambda$ behaves to a certain extent 
as the inverse of the nilpotent operator $\bfA_\lambda(r_\lambda),$
in particular, the operator $S_\lambda$ increases the order of a resonance vector by~$1,$
if there is a room for that. Since the operator $\bfA_\lambda(r_\lambda)$ decreases order by~$1$
and increases depth by~$1,$ this raises a natural question of whether $S_\lambda$ decreases depth by~$1.$
This assertion is not proved, but Theorem \ref{T: TFAE prop A}
provides a criterion for this property.
\begin{thm} (Theorem \ref{T: TFAE prop A})
The following assertions are equivalent: 
\begin{enumerate}
  \item For all $j=1,2,\ldots,d-1$ \
$\im\brs{S_\lambda \bfA_\lambda^{j}} \subset \im\brs{\bfA_\lambda^{j-1}}.$ 
  \item For all $j=1,2,\ldots,d-1$ \
$\im\brs{\hat P^\perp \bfA_\lambda^{j-1}} \subset \im\brs{\bfA_\lambda^{j-1}}.$
\end{enumerate}
\end{thm}
That is, the operator $S_\lambda$ decreases depth of resonance vectors 
by~$1$ iff the orthogonal projection onto the eigenspace $\clV_\lambda$ preserves depth. 

\subsubsection{Section~\ref{S: On phi[nu]}}
The eigenvalue~$\lambda$ of multiplicity~$m$ of the self-adjoint operator $H_{r_\lambda}$
splits into~$m$ analytic eigenvalue functions~$\lambda_\nu(s),$ $\nu=1,\ldots,m,$ 
of the operator $H_s = H_{\rlmb} + (s-r_\lambda)V,$ 
so that~$\lambda_\nu(\rlmb) = \lambda$ for all $\nu = 1,\ldots,m.$ 
The corresponding eigenvector functions we denote by $\phi_\nu(s).$
The eigenvalue functions $\lambda_\nu(s)$ are not necessarily distinct, in which case we list 
them according to their multiplicities, but in any case the analytic eigenvector functions $\phi_\nu(s)$
can be chosen to be pairwise orthogonal: for any $s$ and any $\nu\neq \mu,$ \ 
$\scal{\phi_\nu(s)}{\phi_\mu(s)}=0$. We assume such a choice throughout this paper. 

\begin{thm} \label{IT: order of phi(nu)} (Theorem~\ref{T: TFAE for (Uk)})
Let $k\geq 2,$ let $\phi(s)$ be an analytic path of eigenvectors of the path 
$H_s = H_0+sV.$
The following assertions are equivalent:
\begin{enumerate}
  \item[(i)] The vectors $V\phi(\rlmb), \ V\phi'(\rlmb), \ \ldots, \ V\phi^{(k-2)}(\rlmb)$ 
              \ are orthogonal to the eigenspace $\clV_\lambda.$ 
  \item[(ii)] The vectors 
$
  V\phi(\rlmb), \ V\phi'(\rlmb), \ \ldots, \ V\phi^{(k-2)}(\rlmb) 
$
are orthogonal to the vector $\phi(\rlmb).$ 
  \item[(iii)] The equalities 
$
  \lambda'(\rlmb) = 0, \ \ldots, \ \lambda^{(k-1)}(\rlmb) = 0,
$
hold, where~$\lambda(s)$ is an analytic path of eigenvalues of $H_s$ which corresponds to $\phi(s).$
  \item[(iv)] for all $j=1,2,\ldots,k-1$ \
  $
    (H_{\rlmb}-\lambda)\phi^{(j)}(\rlmb) = - j V \phi^{(j-1)}(\rlmb).
  $
  \item[(v)] for all $j=1,2,\ldots,k-1$ \
   $
     \bfA_\lambda(\rlmb) \phi^{(j)}(\rlmb) = j \phi^{(j-1)}(\rlmb).
   $
  \item[(vi)] $\phi(\rlmb)$ is an eigenvector of depth at least~$k-1.$
\end{enumerate}
\end{thm}

An eigenpath $\phi(s)$ will be said to have order at least $k,$ if it has any of these properties. 
This definition is correct in the sense that an eigenpath $\phi(s)$ has order $k$ if and only
if an eigenpath $a(s) \phi(s)$ has order $k$ for any non-zero analytic function $a(s),$
such that $a(\rlmb) \neq~0.$ 

\begin{thm} (Lemma~\ref{L: nice one}) If an eigenpath $\phi(s)$ has order at least $k,$ then the vectors 
$\phi(\rlmb),$ $\phi'(\rlmb),$ $\ldots,$ $\phi^{(k-1)}(\rlmb)$ are resonance vectors of orders respectively
$1,2,\ldots,k.$ 
\end{thm}

\subsubsection{Section~\ref{S: char-n order k dir-ns}}
It is interesting to find out conditions under which a straight line of operators $H_s = H_{\rlmb} + (s-\rlmb)V$
is tangent to the resonance set at a resonance point $H_{\rlmb}.$ 
We say that a direction~$V$ is \emph{tangent} 
to the resonance set~$\Rset$ at~$H_{\rlmb}$ \emph{to order at least~$k,$}
if there exists a resonant path $\set{H(s)} \subset \euR(\lambda)$\label{Page: H(s) 1}
such that for some (necessarily real) numbers $c_2,c_3,\ldots,c_{k-1}$
\begin{equation} \label{IF: H(s)=H0+sV+c2 s2 V+...}
  H(s) = H_{\rlmb} + (s-\rlmb)V + \sum_{j=2}^{k-1}c_j (s-\rlmb)^j V + O((s-\rlmb)^k),
  \ s \to \rlmb.
\end{equation}
In this case we also say that the path $H(s)$ is tangent to~$V$ at $H(\rlmb)=H_{\rlmb}$ to order at least~$k.$
The \emph{order of tangency} of a direction~$V$ to the resonance set $\euR(\lambda)$
is the largest of positive integers $k$ such that for some resonance path $H(s)$ 
(\ref{IF: H(s)=H0+sV+c2 s2 V+...}) holds.
We say that a direction~$V$ is \emph{tangent} \label{Page: tangent direction(0)}
at~$H_{\rlmb}$ if~$V$ is tangent to order at least~$2.$
If a direction~$V$ is tangent only to order 1 at $H_{\rlmb} \in \Rset,$ then we say that~$V$
is \emph{transversal} at~$H_{\rlmb}.$ \label{Page: transversal direction(0)}

\begin{thm} (Theorem~\ref{T: tangent V to order k})
Assume Assumption~\ref{A: Main Assumption}. Let~$k \geq 1,$ let
$H_{\rlmb}$ be a resonance point and~$V$ be a regular direction at~$H_{\rlmb}.$
If~$H(s)$ is a resonant path tangent to~$V$ at~$H_{\rlmb}$ to order at least~$k$
and if~$\chi(s)$ is a corresponding analytic resonant eigenpath,
then 
\begin{enumerate}
  \item[(i)] vectors~$\chi(\rlmb),\chi'(\rlmb),\ldots,\chi^{(k-1)}(\rlmb)$ have orders respectively $1,2,\ldots, k,$
  \item[(ii)] the direction~$V$ has order at least~$k,$ 
  \item[(iii)] for any $j=1,2,\ldots,k$ 
$$
  \bfA_\lambda(r_\lambda) \chi^{(k-1)}(\rlmb) 
        = (k-1)\chi^{(k-2)}(\rlmb) 
                  + \sum_{j=2}^{k-1} j!{k-1 \choose j} c_j \chi^{(k-1-j)}(\rlmb), 
$$
where the numbers $c_2, \ldots, c_k$ are as in (\ref{IF: H(s)=H0+sV+c2 s2 V+...}), and 
  \item[(iv)] the eigenvector~$\chi(\rlmb)$ has depth at least $k-1.$
\end{enumerate}
\end{thm}

If a resonant path $H(s)$ is tangent to~$V$ to order~$k$ at a resonant point $H_{\rlmb},$
then changing the parameter~$s$ if necessary we can always make
the operators $H''(\rlmb),\ldots,H^{(k-1)}(\rlmb)$ equal to zero, so that the path $H(s)$ takes the form
\begin{equation*} 
  H(s) = H_{\rlmb} + (s-\rlmb)V + O((s-\rlmb)^k).
\end{equation*}
A path of this form 
will be called \emph{standard}.\label{Page: standard path(0)}

According to Theorem~\ref{T: tangent V to order k},
with a resonant path tangent to order~$k$ we can associate a set of resonance vectors~$\chi_0, \ldots, \chi_{k-1}$
of respective orders $1,\ldots,k;$ namely, the first~$k$ coefficients
\begin{equation*} 
  \chi_j = \frac{1}{j!} \chi^{(j)}(\rlmb), \ j = 0, 1, 2, \ldots
\end{equation*}
of the Taylor expansion of a resonant eigenpath~$\chi(s).$

\begin{prop} (Proposition~\ref{P: amazing proposition})
Assume Assumption~\ref{A: Main Assumption}. 
Let~$V$ be a regular direction at~$H_{\rlmb}$
and let $H(s)$ be a resonant path tangent to~$V$ at~$H_{\rlmb}$ to
order~$k.$ The path $H(s)$ is standard if and only if
for all~$j=1,2,\ldots,k-1$ there holds the equality 
\begin{equation*} 
  \bfA_\lambda(r_\lambda) \chi_j(r_\lambda) = \chi_{j-1}(r_\lambda),
\end{equation*}
where~$\chi(s)$ is a corresponding analytic path of eigenvectors.
\end{prop}

\smallskip

Given a direction~$V$ of order~$d$ it is possible to present a resonant path~$H(s)$ which is tangent to~$V$
to order~$d.$ Namely, let~$\chi$ be an eigenvector of~$H(\rlmb)$ corresponding to the eigenvalue~$\lambda,$
and let $W = \scal{\chi}{\cdot}\chi.$ 
We consider the intersection of the two-dimensional real affine plane 
$$
  \alpha = H_{\rlmb} + \mbR V + \mbR W
$$
with the resonance set~$\euR(\lambda).$ 
A sufficiently small neighbourhood of $H_{\rlmb}$ in the intersection $\alpha \cap \euR(\lambda)$ 
consists of one and only one simple curve (Theorem~\ref{T: intersection is a curve}). 
We denote this analytic curve by~$\gamma_\chi.$ The curve $\gamma_\chi$ can be normalised so that $\gamma_\chi(\rlmb) = \chi$ 
(Theorem~\ref{T: chi(0) is exactly that}).
\begin{thm} (Theorem~\ref{T: high order then tangent}) If~$\chi$ has depth at least $k-1$ then 
\begin{enumerate}
  \item the analytic curve~$\gamma_\chi$ is tangent to~$V$ to order~$k,$
  \item for any parametrisation $\gamma(s)$ of $\gamma_\chi,$ 
  the vectors $\gamma(\rlmb), \ \gamma'(\rlmb), \ \ldots, \gamma^{(k-1)}(\rlmb)$ have 
  orders respectively $1,2,\ldots,k,$
  \item if the parametrisation $\gamma(s)$ of $\gamma_\chi$ is standard then 
      for all $j = 1,2,\ldots,k-1$ \ $\bfA_\lambda(\rlmb) \gamma^{(j)}(\rlmb) = j \gamma^{(j-1)}(\rlmb).$
\end{enumerate}
\end{thm}
\noindent 
Theorems~\ref{T: tangent V to order k} and~\ref{T: high order then tangent} 
provide the following geometric interpretation of order of a direction~$V.$
\begin{thm} (Theorem~\ref{T: V is k-tangent iff order >=k}) 
The order of tangency of a regular direction $V$ at a resonance point $H_{\rlmb}$ 
to the resonance set is equal to the order of the direction~$V.$
\end{thm}

\subsubsection{Section~\ref{S: res points as f-ns of s}}
Assume that~$r_\lambda$ is a resonance point of algebraic multiplicity~$N$
and geometric multiplicity~$m.$
A resonance point~$r_z$ depends analytically on~$z$ outside essential spectrum,
and as~$\lambda$ varies the resonance point~$r_\lambda$ splits into up to~$N$ 
resonance points~$r_z^j.$ Theorem~\ref{T: r(z) has order 1} asserts 
that all these resonance points~$r_z^j$
have geometric multiplicity~$1$ in some deleted neighbourhood of~$\lambda.$
When~$z$ makes one round around~$\lambda,$ these~$N$ resonance points undergo a permutation. 
Theorem \ref{T: depth of phi[nu] is d[nu]-1} asserts that this permutation 
is a product of~$m$ disjoint cycles of lengths~$d_1, \ldots, d_m.$ We denote these cycles by 
$$
  r_\nu^{(j)}(z), \ \nu = 1,\ldots,m, \ \ j=0,\ldots, d_\nu-1.
$$
Functions of each cycle $r_\nu^{(\cdot)}(z)$ represent branches of a multivalued
holomorphic function. The idempotents $P_z(r_\nu^{(j)}(z))$ which correspond to these resonance points
are also multivalued, but the sum 
$$
  P_z^{[\nu]} = \sum_{j=0}^{d_\nu-1}P_z(r_\nu^{(j)}(z))  
$$
is single-valued in a neighbourhood of~$\lambda.$ Proposition~\ref{P: P[nu](z) is holomorphic} asserts 
that this function admits analytic continuation to the point $\lambda.$
Thus, the limit operator 
$
  P_\lambda^{[\nu]}(r_\lambda) 
$
is defined. Similarly, one can define operators $Q_z^{[\nu]},$ or they also can be defined by formula
$Q_z^{[\nu]} = \brs{P_{\bar z}^{[\nu]}}^*.$
The operators $P_z^{[\nu]},$ including the case of $z = \lambda,$ have the following properties:
$$
  P_z^{[\nu]} P_z^{[\mu]}  = \delta_{\nu\mu} P_z^{[\nu]}, 
  \quad \bfA_z(r_\lambda) P_z^{[\nu]} = P_z^{[\nu]} \bfA_z(r_\lambda),
  \quad P_z(r_\lambda) = \sum_{\nu=1}^{m} P_z^{[\nu]},
  \quad V P_\lambda^{[\nu]} = P_\lambda^{[\nu]} V.
$$
Here $P_z(r_\lambda)$\label{Page: Pz(r lmd)}\label{Page: Az(r lmd)} (respectively, $\bfA_z(r_\lambda)$) 
is the sum of idempotents $P_z(r_\nu^{(j)}(z))$ 
(respectively, nilpotent operators $\bfA_z(r_\nu^{(j)}(z))$) over all resonance points 
$r_\nu^{(j)}(z)$ of the group of~$r_\lambda.$
In particular, the image $\Upsilon_\lambda^{[\nu]}$ of the operator 
$P_\lambda^{[\nu]}$ reduces the operator $\bfA_\lambda(r_\lambda).$ This reduction 
we denote by \label{Page: bfA(nu)1}
$
  \bfA_\lambda^{[\nu]}.
$
Restriction of this operator to $\Upsilon_\lambda^{[\nu]}$ is cyclic. 

\smallskip 
Proposition~\ref{P: first Puiseux coef is non-zero} asserts that 
in the Puiseux series (\ref{F: Puiseux for rz})
$$
  r_\nu^{(j)}(z) = \sum_{k=0}^\infty r_{k/d_\nu}\eps_{d_\nu}^{kj} (z-\lambda)^{k/d_\nu}, \ \ j = 0,\ldots,d_\nu-1,
$$
of the function $r_\nu^{(\cdot)}(z)$ the coefficients $r_{k/d_\nu}$ are real and $r_{1/d_\nu} \neq 0.$

Theorem~\ref{T: about res cycles}
and Proposition~\ref{P: number of cycles = m} assert that there is a natural one-to-one
correspondence between cycles $r_\nu^{(\cdot)}(z)$ and eigenvalue functions $\lambda_\nu(\cdot),$
and therefore, with eigenpaths $\phi_\nu(\cdot).$ Namely, 
restriction of one of the functions $r_\nu^{(0)}(z)$ of a cycle $r_\nu^{(\cdot)}(z)$
to at least one of the half-intervals
$[\lambda,\lambda+\eps)$ or $(\lambda-\eps,\lambda]$ takes real values and 
it is the inverse of $\lambda_\nu(s)$ in some left or right neighbourhood of~$r_\lambda.$
Such a function $r_\nu^{(0)}(z)$ is unique in the sense that there is no other branch 
of $r_\nu^{(\cdot)}(z)$ which takes real values when restricted to the same half-interval as 
the function $r_\nu^{(0)}(z).$ Theorem~\ref{T: about res cycles} provides more properties of this 
one-to-one correspondence. 

Theorem~\ref{T: depth of phi[nu] is d[nu]-1} 
provides further information about this correspondence. It asserts that for each $\nu=1,\ldots,m$
the following numbers are equal:
\begin{enumerate}
  \item order of the eigenpath $\phi_\nu(s)$ (see Theorem \ref{IT: order of phi(nu)}),
  \item the size of the cycle $r_\nu^{(\cdot)}(z),$
  \item the size of the $\nu$th Jordan cell of $R_\lambda(H_s)V$ corresponding to eigenvalue $(s-r_\lambda)^{-1}.$
\end{enumerate}
The number~$d_\nu$ is their common value.

\smallskip
The Puiseux series of the idempotent $P_z(r_\nu^{(j)}(z))$ has the form (Proposition~\ref{P: p=d-1})
$$
  P_z(r_\nu^{(j)}(z))
      = \tilde P^{(j)}_\nu(z) 
        + \sum_{l=0}^{d_\nu-1} e^{\frac{-2\pi i l j}{d_\nu}}(z-\lambda)^{-l/d_\nu}P_{-l/d_\nu}, 
$$
where $\tilde P^{(j)}_\nu(z)$ is continuous at $z = \lambda.$
The informative part of this formula is the upper summation limit $d_\nu-1.$
Further, for each $\nu=1,\ldots,m$ and for all $k\geq 0$ (Proposition \ref{P: what a croc})
$$
  \lim_{z \to \lambda} \sum_{j=0}^{d_\nu-1} (r_\nu^{(j)}(z)-r_\lambda)^k P_z(r_\nu^{(j)}(z)) 
      = P^{[\nu]}_\lambda(r_\lambda)\bfA_\lambda^k (r_\lambda).
$$

\smallskip 

Theorem~\ref{T: (i) and (ii) bfA phi(j)=phi(j-1)} provides finer information about the connection 
of the eigenpath~$\phi_\nu$ to the vector space $\Upsilon_\lambda^{[\nu]}(r_\lambda).$

\begin{thm} (Theorem~\ref{T: (i) and (ii) bfA phi(j)=phi(j-1)} and Corollary~\ref{C: dim Upsilon[nu]=d(nu)})
Let $\phi_\nu(s)$ be an eigenpath which corresponds to a cycle~$\nu$ of length~$d_\nu.$
The set of vectors 
\begin{equation} \label{IF: phi(nu)(j)}
 \phi_\nu(\rlmb), \ \phi'_\nu(\rlmb), \ \ldots, \ \phi^{(d_\nu-1)}_\nu(\rlmb)
\end{equation}
is a basis of the vector space $\Upsilon_\lambda^{[\nu]}(r_\lambda)$
and for any $j=1,2,\ldots,d_\nu-1$
$$
  \bfA_\lambda(r_\lambda) \phi_\lambda^{(j)}(r_\lambda) = j \phi_\lambda^{(j-1)}(r_\lambda).
$$
\end{thm}

This theorem implies that the set of vectors (\ref{IF: phi(nu)(j)})
is a Jordan basis of $\Upsilon_\lambda(r_\lambda).$ 



The following theorem provides a relationship between different eigenpaths $\phi_\nu(s).$
\begin{thm} \label{IT: (phi,V phi) =0} 
(Theorem~\ref{T: phi(j,mu)(0) is self-dual}, Corollary~\ref{C: second proof}) 
Let $\phi_\nu(s)$ and $\phi_\mu(s)$ be two different eigenpaths of~$H_s,$ $s \in \mbR.$
For all $j = 0,1,\ldots,d_\nu-1$ and all $k = 0,1,\ldots,d_\mu-1,$
$$
  \scal{\phi_\nu^{(j)}(\rlmb)}{V\phi_\mu^{(k)}(\rlmb)} = 0.
$$
\end{thm}
This theorem is a stronger version of Theorem~\ref{T: phi(j,mu)(0) is self-dual}
in that it shows that the numbers~$d_\nu$ from its statement add up to~$N.$

\smallskip

Proposition~\ref{P: Prop 6.20} gives an expression for the restriction of a power of the operator 
$\bfA_\lambda(r_\lambda)$
to the vector space $\Upsilon_\lambda^{[\nu]}(r_\lambda)$ via Puiseux coefficients of $r_z^{(j)}(z)$
and $P_z(r_z^{(j)}(z)):$ for all $\nu=1,\ldots,m$ and for all $k=1,2,\ldots,d_\nu-1$ 
\begin{equation*} 
    P_\lambda^{[\nu]}(r_\lambda) \bfA_\lambda^k(r_\lambda) 
          = d_\nu \cdot \sum_{l=k}^{d_\nu-1} \brs{\sum_{m_1+\ldots+m_k=l} 
                r_{m_1/d_\nu}\ldots r_{m_k/d_\nu}} P_{-l/d_\nu},
\end{equation*}
where in the sum $m_1, \ldots, m_k \geq 1.$

\smallskip

Inside the essential spectrum
geometric meaning of the total resonance index is obscure, but outside the essential spectrum 
it has a clear geometric interpretation, which is allowed by the fact 
that~$r_\lambda$ depends on~$\lambda$
analytically. Assume that $r_0$ is a resonance point corresponding to~$\lambda_0,$
that is,~$\lambda_0$ is an eigenvalue of $H_0 + r_0 V.$ 
If~$r_\lambda$ is an increasing function of~$\lambda$ in a neighbourhood~$I$ of~$\lambda_0,$ 
and therefore, if~$\lambda(r)$ is an increasing function of~$r,$ then intuitively contribution 
of the function~$\lambda(r)$ to the spectral flow through~$\lambda_0$ is $+1.$
Since~$r_\lambda$ is an increasing function, the derivative $\frac {d}{d \lambda}r_\lambda$
is positive on $I.$ Therefore, according to the geometric interpretation 
of the derivative of a holomorphic function, when $z=\lambda_0$ is perturbed to $z=\lambda_0 + iy$ with small $y>0,$
the real value $r_0$ of the function~$r_z$ rotates towards the half-plane~$\mbC_+.$ This gives a contribution 
to the total resonance index of~$+1.$ Similarly, decreasing function~$r_\lambda$ contributes $-1$ to the TRI.

A more interesting situation occurs if at some value~$r_0$ of the coupling constant 
the derivative~$\lambda'(r_0)$ vanishes. Geometrically, this means that 
the eigenvalue~$\lambda(r)$ of $H_r$ stops at~$\lambda_0$ when the coupling constant attains the value $r_0$ 
and may either turn back or go through~$\lambda,$ thus contributing one of the three numbers 
$-1, 0,$ or $+1$ to the spectral flow through~$\lambda_0.$  
Since the derivative~$\lambda'(r_0)$ vanishes, the inverse function $r(z)$
is not single-valued in a neighbourhood of $z = \lambda_0.$ Accordingly, when 
$\lambda_0$ is shifted to~$\lambda_0+iy$ with $y>0,$ the resonance point $r_0$ splits into two or more resonance 
points. About one half $d_+$ of those resonance points goes towards~$\mbC_+$ 
and another half $d_-$ goes towards~$\mbC_-,$ resulting in the appropriate 
contribution to the total resonance index $d_+-d_-.$ 
The overall number $d=d_+ + d_-$ of the resonance points equals the smallest of positive numbers~$d$ 
such that $\frac {d^d \lambda}{dr^d}\Big|_{r=r_0}\neq 0.$ Moreover, those~$d$ resonance points 
undergo a cyclic permutation when~$z$ makes one round around~$\lambda_0.$ Combining this with the fact 
that~$r_z$ cannot be real for non-real~$z,$ one can already infer that $d_+-d_-$ is to be equal to 
the contribution of~$\lambda(r)$ to the spectral flow. If~$\lambda$ has geometric multiplicity $m>1,$
then this argument applies to each of the~$m$ eigenvalue functions $\lambda_\nu(s).$
This formal observation is made precise 
in Theorems~\ref{T: sign of nu} and~\ref{T: intersection number = TRI}.

\begin{thm} (Theorem~\ref{T: sign of nu}) 
For each $\nu = 1,\ldots,m$ and for all small enough $\eps>0$ and $y>0,$ signs 
of the following real numbers coincide:
\begin{enumerate}
  \item $\lambda_\nu(r_\lambda + \eps) - \lambda_\nu(r_\lambda),$
  \item $\scal{\phi_\nu(r_\lambda)}{V\phi_\nu^{(d_\nu-1)}(r_\lambda)},$
  \item $\Im r_\nu^{(0)}(z+iy),$ for all $0<y << 1,$ and for all $z \in I,$ where $I$ is one of the two 
intervals $(\lambda, \lambda+\eps)$ or $(\lambda-\eps, \lambda),$ on which the branch $r_\nu^{(0)}(z+iy)$ 
takes positive values (such an interval and such a branch exist and are unique).
\end{enumerate}
\end{thm}
This sign will be called the \emph{sign of a cycle}~$\nu$ and denoted $\sign(\nu).$ 
Let 
$$
  b_\nu = \left\{ \begin{matrix} 0, & \text{if} \ d_\nu \ \text{is even}, \\
                                1, & \text{if} \ d_\nu \ \text{is odd}. \end{matrix} \right.
$$
\noindent
The intersection number through~$\lambda$ for a resonance point $r_\lambda$
can be defined by equality 
\begin{equation} \label{IF: inters-n number}
  \sum_{\nu=1}^m b_\nu \sign(\nu).
\end{equation}

\begin{thm} (Theorem~\ref{T: intersection number = TRI})
The sum of the intersection numbers (\ref{IF: inters-n number}) of resonance points 
of a path $H_r, \ r \in [0,1],$ through~$\lambda$ is equal to the total resonance index. 
\end{thm}

\smallskip
Theorem \ref{IT: (phi,V phi) =0} allows to prove the following two theorems. 
The first of these theorems is an explicit formula for the idempotent $P_\lambda(r_\lambda),$
the second theorem asserts that the resonance index equals the signature of the so-called resonance matrix. 
The second theorem holds for~$\lambda$ inside essential spectrum too \cite[Section 9]{Az9}, 
but here we provide a new and simpler proof in the case $\lambda \notin \sigma_{ess}.$ 

\begin{thm} (Theorem~\ref{T: P = sum sum})
The idempotent operator~$P_\lambda(r_\lambda)$ can be written in the form 
$$
  P_\lambda(r_\lambda) = \sum_{\mu=1}^m\sum_{\nu=1}^m \sum_{k=0}^{d_\mu-1} \sum_{j=0}^{d_\nu-1}
      \frac 1{k!j!} \alpha_{\mu\nu}^{kj} \scal{V \phi_\mu^{(k)}(r_\lambda)}{\cdot} \phi_\nu^{(j)}(r_\lambda),
$$
where the $N \times N$ matrix $\alpha$ is a direct 
sum of self-adjoint skew-upper triangular Hankel matrices of sizes $d_1, \ldots, d_m.$ 
\end{thm}

\begin{thm} (Theorem~\ref{T: ind res=sign VP}) 
The equality 
$$
  \ind_{res}(\lambda; H_{r_\lambda},V) = \sign V P_\lambda(r_\lambda)
$$
holds. 
\end{thm}

\subsubsection{Section~\ref{S: On stability of res-index}}

In section~\ref{S: On stability of res-index}
we study dependence of resonance index $\ind_{res}(\lambda; H_{\rlmb},V)$ 
on direction~$V.$ This requires some topology in~$\clA.$
We postulate that the topology of~$\clA$ 
has the following properties:
for some non-real complex number~$z$ and for some $H_0 \in \clA$ (1) the product $V R_z(H_0)$ 
continuously depends on~$V$ 
and (2) the product $V_1 R_z(H_0) V_2$ is compact and jointly continuously depends on~$V_1$ and~$V_2.$
These properties hold if~$\clA_0$ consist of bounded operators and the topology of~$\clA_0$
is the uniform topology. 

The sets of regular and simple directions are open in norm of~$\clA_0$ (Lemma~\ref{L: regular dir-s are open}).
Restrictions of the mappings 
$$
  \clA_0 \ni V \mapsto P_\lambda(H_{\rlmb},V)
\quad \text{and} \quad
  \clA_0 \ni V \mapsto V P_\lambda(H_{\rlmb},V)
$$
to the open set of simple directions are continuous in the norm of~$\clA_0$ 
(Lemma~\ref{L: P(H0,V) is cont-s}). Further, the resonance index
$$
  \clA_0 \ni V \mapsto \ind_{res}(\lambda; H_{\rlmb},V)
$$
is a locally constant function on the set of simple 
directions (Theorem~\ref{T: stability of res index for order 1 V}). 

For regular directions these assertions are not true. But the following theorem holds.
In this theorem for simplicity we assume that~$r_\lambda = 0.$
\begin{thm} (Theorems~\ref{T: total res. index for (H0,W)} and~\ref{T: total res. index for (H1,V)})  
Let~$V$ be a regular direction at a resonance point~$H_{0}.$
Let~$W$ be a small (in the norm of~$\clA_0$) perturbation of~$V$ and let 
$H'_{0}$ be a small perturbation of~$H_{0}.$ Let 
$r^1_\lambda(H'_{0},W), \ r^2_\lambda(H'_{0},W), \ \ldots$
be resonance points of the triple $(\lambda; H'_{0},W)$ which belong to the group of the resonance
point $s = 0$ of the triple $(\lambda; H_0,V),$ where $H_r = H_{0} + rV.$ 
Then the resonance index 
$\ind_{res} (\lambda; H_{0},V)$
is equal to the sum of resonance indices
$$
  \sum_{j} \ind_{res} (\lambda; H'_{r_\lambda^j},W),
$$
where the sum is taken over real resonance points of the group of $s = 0.$
\end{thm}

This theorem allows to prove homotopy stability of the total resonance index. 
\begin{thm} (Theorem~\ref{T: homotopy stability of res index}) 
Let $H_0, H_1$ be two operators from~$\clA$ such that
$H_0$ and $H_1$ are not resonant at~$\lambda \notin \sigma_{ess}.$ Then there exist neighbourhoods
$\euU_0$ and $\euU_1$ of~$H_0$ and $H_1$ respectively such that for all $H'_0 \in \euU_0$
and all $H'_1 \in \euU_1$
$$
  \sum_{r \in [0,1]} \ind_{res}(\lambda; H_r,V) = \sum_{r \in [0,1]} \ind_{res}(\lambda; H'_r,V'),
$$
where~$V' = H'_1 - H'_0$ and $H'_r = H'_0+rV'.$
\end{thm}

Theorem~\ref{T: Tot res index meets R-S axioms} asserts that the total 
resonance index satisfies Robbin-Salamon axioms for spectral flow. 
Since Robbin-Salamon axioms uniquely identify spectral 
flow (Theorem~\ref{T: Robbin-Salamon uniqueness theorem}), this proves the equality 
of the TRI and the spectral flow.

In the last subsection of section~\ref{S: On stability of res-index}
we give proof of some well-known properties of the resonance set $\euR(\lambda):$
$\mathrm{codim}\,\euR(\lambda) = 1,$ the set~$\euR(\lambda)$ has no cusps and that any plane section of 
$\euR(\lambda)$ consists of no more than~$m$ simple curves. 

\subsubsection{Section~\ref{S: res index and Fredholm index}}
In this section we also assume that the resonance point~$r_\lambda$ is~$0.$
In subsection~\ref{SS: res matrix as dir-n reduction}
we observe that the finite-rank self-adjoint operator $VP_\lambda(H_{0},V),$
which is called resonance matrix of the triple $(\lambda; H_{0},V),$
preserves many properties of the initial direction~$V:$
if $V$ is regular then so is $V P_\lambda$ (Theorem~\ref{T: VP is regular}),
further, the operators $P_\lambda$ and $\bfA_\lambda$ are the same for the triples
$(\lambda; H_{0},V)$ and $(\lambda; H_{0},VP_\lambda)$ (Theorem~\ref{T: P(H0,VP)=P and A(H0,VP)=A}),
and the resonance matrices of the directions~$V$ and~$VP_\lambda$ 
are equal (Theorem~\ref{T: res matrices of V and VP}). 
In particular, resonance indices of the triples 
$(\lambda; H_{0},V)$ and $(\lambda; H_{0},VP_\lambda)$ coincide (Theorem~\ref{T: ind(V)=ind(VP)}):
$$
  \ind_{res}(\lambda; H_{0},V) = \ind_{res}(\lambda; H_{0},VP_\lambda).
$$
All these assertions in essence follow from the following observation (Theorem~\ref{T: R(H0+sVP)VP=R(H0+sV)VP}):
$$
  R_\lambda(H_{0} + s V P_\lambda) V P_\lambda = R_\lambda(H_{0} + s V) V P_\lambda.
$$
Further, the operators $V$ and $VP_\lambda$ are plane homotopic, that is, a direction $V$ 
can be deformed to a direction $VP_\lambda$ in the affine plane generated by these directions
without crossing the resonance set (Theorem~\ref{T: V and VP are plane homotopic}). 

In subsection~\ref{SS: res ind and Fred ind} we give a direct proof of equality of TRI and total Fredholm index,
Theorem~\ref{T: TRI=TFI}. 

\subsubsection{Section~\ref{S: res ind=SSF}}
In section~\ref{S: res ind=SSF} we give a direct proof of equality of the total resonance index
and spectral shift function. This assertion is a special case of equality of TRI and singular spectral shift function
\cite[Section 6]{Az9}, \cite{Az7}.
It also follows from the fact that the spectral shift function satisfies Robbin-Salamon axioms. 
The proof presented in this subsection highlights a key idea of the proof of the more general result given in 
\cite[Section 6]{Az9}, \cite{Az7}.

\subsection{Acknowledgements}
I thank Mr Tom Daniels for a scrupulous reading of this paper which resulted
in countless improvements including a countable subset of fixed articles. 

\section{$2\times 2$ matrix representations}
\label{S: 2x2}
\subsection{$2\times 2$ representation of $A_z(s)$}
Recall that we are working in the setting of Assumption~\ref{A: Main Assumption}.
Let~$H_{\rlmb}$ be a resonance point of multiplicity~$m.$
Accordingly, let
$$
  \hilb = \hat \hilb \oplus \clV_\lambda
$$
be the orthogonal decomposition of the Hilbert space $\hilb,$ on which~$H_{\rlmb}$ acts,
into the sum of the~$m$-dimensional eigenspace \label{Page: clV(lamb)}
$$
  \clV_\lambda = \set{\chi \in \hilb \colon H_{\rlmb}\chi = \lambda \chi}
$$ 
and its orthogonal complement which we hereby denote $\hat \hilb.$\label{Page: hat hilb}
Operators acting on the Hilbert space $\hilb = \hat \hilb \oplus \clV_\lambda$ 
can be written as~$2\times 2$ matrices. 
The operator~$H_{\rlmb}$ has the matrix representation\label{Page: hat H0}
\begin{equation} \label{F: H0 as 2x2 matrix}
  H_{\rlmb} = \left(\begin{matrix}
    \hat H_{\rlmb} & 0 \\
    0 & \lambda I_m
  \end{matrix}\right).
\end{equation}
The operator~$V$ has the form
\label{Page: aaa}\label{Page: hat V}\label{Page: v}
\begin{equation} \label{F: V as 2x2 matrix}
  V = \left(\begin{matrix}
    \hat V & v \\
    v^* & \aaa
  \end{matrix}\right),
\end{equation}
where $\hat V$ is a self-adjoint operator in~$\hat \hilb,$ $v$ 
is an operator from~$\clV_\lambda$ to $\hat \hilb$ and $\aaa$ is a
self-adjoint operator on~$\clV_\lambda.$
In \cite[p.\,14]{RoSa} the operator $\aaa$ is called the \emph{crossing operator}.

We agree to identify an element $\hat \chi$ of the Hilbert space $\hat \hilb$ 
with an element $\twovector{\hat \chi}{0}$ of $\hilb.$
Analogously, an operator $v \colon \clV_\lambda \to \hat \hilb$
will also be considered as an operator from $\hilb \to \hilb.$ This remark applies to other operators
such as~$\hat H_s,$ $v^*$ and~$\aaa.$ 

\label{Page: hat P}By $\hat P$ we denote the operator of orthogonal 
projection from $\hilb$ onto $\hat \hilb.$
Sometimes we write $\twovector{\chi}{0}$ for $\hat P \chi$ and 
$\twovector{0}{\chi}$ for $\hat P^\perp \chi,$ and thus~$\chi$ and 
$\twovector{\chi}{\chi}$ are two ways to write the same vector. 
Similarly, the number $1$ is treated as the identity operator on $\hilb,$ or 
$\hat \hilb$ or $\hat \hilb^\perp,$ depending on the context in which it appears. 
The components of the operator $V$ in its $2 \times 2$ representation (\ref{F: V as 2x2 matrix})
can be defined by formulas 
$$
  \hat V = \hat P V \hat P, \ \ v = \hat P V \hat P^\perp \ \ \text{and} 
  \ \ \aaa = \hat P^\perp V \hat P^\perp.
$$


$2 \times 2$ representation of the operator 
$$
  H_s = H_{\rlmb} + (s-\rlmb)V
$$ 
is given by
\begin{equation*} \label{F: Hr as 2x2 matrix}
  H_{s} = \left(\begin{matrix}
    \hat H_{s} & (s-\rlmb)v \\
    (s-\rlmb)v^* & \lambda + (s-\rlmb)\aaa
  \end{matrix}\right),
\end{equation*}
where $\hat H_{s} = \hat H_{\rlmb} + (s-\rlmb)\hat V.$
According to our agreements made above, we can rewrite this formula 
as $$H_s = \hat H_s + \lambda\hat P^\perp + (s-\rlmb)(v + v^*) + (s-\rlmb)\aaa.$$ 
We shall often use both ways of writing these kind of formulas, but mostly we prefer
the matrix formulas. 

The following lemma is a well-known fact of linear algebra.
\begin{lemma} \label{L: inverse of block matrix}
Assume that~$A$ is an invertible operator with bounded inverse. 
A block operator 
\begin{equation*}
  \brs{\begin{matrix} A & B \\ C & D \end{matrix}}
\end{equation*}
has bounded inverse if and only if the operator $D - CA^{-1}B$ has bounded inverse.
In this case the inverse of the block operator is given by 
\begin{equation*}
  \brs{\begin{matrix} A & B \\ C & D \end{matrix}}^{-1}
    =
  \brs{\begin{matrix}
        A^{-1} + A^{-1}B \euD CA^{-1} & -A^{-1}B\euD \\
        -\euD CA^{-1} & \euD
       \end{matrix}
  },
\end{equation*}
where
$$
  \euD = (D - CA^{-1}B)^{-1}.
$$
\end{lemma}

It follows from this lemma that the inverse of the operator $H_s - z$ is given by
\begin{equation} \label{F: (Hr-z)^(-1) as 2x2 matrix}
  \begin{split}
    (H_s - z)^{-1} & =
    \left(
       \begin{matrix}
          \hat H_s-z & (s-\rlmb)v \\
          (s-\rlmb)v^* & \lambda - z + (s-\rlmb)\aaa
       \end{matrix}
    \right)   ^{-1}
    \\ & =
    \left(
       \begin{matrix}
          R_z(\hat H_s) + (s-\rlmb)^2 R_z(\hat H_s) v \euD_z(s) v^* R_z(\hat H_s) & (\rlmb-s) R_z(\hat H_s) v \euD_z(s) \\
          (\rlmb-s)\euD_z(s)  v^* R_z(\hat H_s) &  \euD_z(s)
       \end{matrix}
    \right),
  \end{split}
\end{equation}
where
\begin{equation} \label{F: euDz}
  \euD_z(s) = (\lambda - z + (s-\rlmb)\aaa - (s-\rlmb)^2 v^* R_z(\hat H_s) v)^{-1}.
\end{equation}
We use notation
$$
  \hat A_z(s) = R_z(\hat H_s) \hat V.
$$
We have,
\begin{equation} \label{F: 2x2 for Az(s)}
  \begin{split}
    A_z(s) & := (H_s - z)^{-1}V
    \\ & =
    \left(
       \begin{matrix}
          R_z(\hat H_s) + (s-\rlmb)^2 R_z(\hat H_s) v \euD_z(s) v^* R_z(\hat H_s) & (\rlmb-s) R_z(\hat H_s) v \euD_z(s) \\
          (\rlmb-s)\euD_z(s)  v^* R_z(\hat H_s) &  \euD_z(s)
       \end{matrix}
    \right)
   \left(\begin{matrix}
    \hat V & v \\
    v^* & \aaa
  \end{matrix}\right)
  \\ & =
    \left(
       \begin{matrix}
          \hat A_z(s) + (\rlmb-s) R_z(\hat H_s)v \euD_z(s) v^*\euF_z(s)
                              &  R_z(\hat H_s)v(1 + (\rlmb-s )[\ldots] )  \\
         \euD_z(s)  v^* \euF_z(s)            &   [\ldots]
       \end{matrix}
    \right),
  \end{split}
\end{equation}
where 
$$
  [\ldots]=\euD_z(s)\SqBrs{\aaa + (\rlmb-s)v^* R_z(\hat H_s)v},
$$ 
and\label{Page: euF(r)=1-r Az(r)}
\begin{equation} \label{F: euF(r)=1-r Az(r)}
  \euF_z(s) = 1 + (\rlmb-s) \hat A_z(s),
\end{equation}
where $1$ is the identity operator on $\hat \hilb.$ 
The operator $\euF_z(s)$ is invertible and
\begin{equation} \label{F: euF(-1)=...}
  \euF^{-1}_z(s) = 1 + (s-\rlmb)\hat A_z(\rlmb).
\end{equation}
The second resolvent identity implies that 
\begin{equation} \label{F: euF(r)Az(0)=Az(r)}
  \euF_z(s) \hat A_z(\rlmb) = \hat A_z(s).
\end{equation}


\smallskip

From now on we consider the case of $z = \lambda.$

For a regular direction~$V$ we have from (\ref{F: euDz})\label{Page: euD lambda(r)}
\begin{equation} \label{F: euD lambda(r)}
  \euD_\lambda(s) = ((s-\rlmb)\aaa - (s-\rlmb)^2 v^* R_\lambda(\hat H_s)v)^{-1}.
\end{equation}

\begin{lemma} A direction~$V$ given by (\ref{F: V as 2x2 matrix}) is regular 
at a resonance point~$H_{\rlmb}$ given by (\ref{F: H0 as 2x2 matrix}) if and only if
the matrix 
$$
  \aaa + (\rlmb-s)v^* R_\lambda(\hat H_s)v
$$
is defined and is invertible for some real value of~$s.$
\end{lemma}
\begin{proof} This follows from Lemma~\ref{L: inverse of block matrix},
2 x 2 representation (\ref{F: (Hr-z)^(-1) as 2x2 matrix}) of the resolvent $(H_s-\lambda)^{-1},$
and from the definition (\ref{F: euD lambda(r)}) of $\euD_\lambda(s).$
\end{proof}

\begin{cor} \label{C: criterion of regularity for V} 
If a direction~$V$ given by (\ref{F: V as 2x2 matrix}) is regular at~$H_{\rlmb}$
then for any non-zero eigenvector~$\chi$ of~$H_{\rlmb}$ at least one of the two vectors 
$\aaa \chi$ or $v \chi$ is non-zero. 
That is, if~$V$ is regular then for any non-zero eigenvector~$\chi$ the vector $V\chi$ 
is also non-zero.
\end{cor}
This necessary condition of regularity of a direction~$V$ is not sufficient.
A simple three-dimensional example, which demonstrates this, can be found in \cite[\S 14.6.1]{Az9}.

\smallskip
Using $2\times 2$ representation (\ref{F: 2x2 for Az(s)}) of $A_\lambda(s)$ 
and the equality (\ref{F: euD lambda(r)}), a direct calculation 
gives the formulas 
\begin{equation} \label{F: 2x2 repr-n of A lamb(r)}
   A_\lambda(s) =
    \left(
       \begin{matrix}
          \hat A_\lambda(s) +(\rlmb-s) R_\lambda(\hat H_s)v \euD_\lambda(s) v^*\euF_\lambda(s)
                              &  0  \\
         \euD_\lambda(s)  v^* \euF_\lambda(s)            &  (s-\rlmb)^{-1}
       \end{matrix}
    \right)
\end{equation}
and
\begin{equation} \label{F: 2x2 repr-n of 1-r A lamb(r)}
    1+(\rlmb-s)A_\lambda(s) =
    \left(
       \begin{matrix}
           \SqBrs{1 + (s-\rlmb)^2 R_\lambda(\hat H_s)v \euD_\lambda(s) v^*}\euF_\lambda(s)
                              &  0  \\
           (\rlmb-s)\euD_\lambda(s)  v^* \euF_\lambda(s)            &  0
       \end{matrix}
    \right).
\end{equation}
The function $A_\lambda(s)$ is not holomorphic at~$s=\rlmb.$ 
As can be seen from (\ref{F: 2x2 repr-n of A lamb(r)}),
apart from the $(2,2)$-entry of $A_\lambda(s),$
the only factor which violates holomorphicity 
of this function is $\euD_\lambda(s),$ the other terms are holomorphic at~$s=\rlmb.$

\smallskip
One can note that a vector $\hat \phi \in \hat \hilb$ has order 2 if and only if
\begin{equation*} 
  \SqBrs{1+(\rlmb-s)A_\lambda(s)} \twovector{\hat \phi}{0} = \twovector{0}{\ldots},
\end{equation*}
where the dots denote a non-zero vector. 
By (\ref{F: 2x2 repr-n of 1-r A lamb(r)}), this equality is equivalent to
\begin{equation*} 
  \SqBrs{1  + (s-\rlmb)^2 R_\lambda(\hat H_s)v \euD_\lambda(s) v^*} \euF_\lambda(s) \hat \phi = 0.
\end{equation*}
It follows that $\hat \phi$ is a vector of order two if and only if
\begin{equation*} 
  \begin{split}
    \hat \phi & = - (s-\rlmb)^2 \euF^{-1}_\lambda(s) R_\lambda(\hat H_s)v \euD_\lambda(s) v^* \euF_\lambda(s) \hat \phi
      \\ & = - (s-\rlmb)^2 R_\lambda(\hat H_{\rlmb}) v \euD_\lambda(s) v^* \euF_\lambda(s) \hat \phi.
  \end{split}
\end{equation*} 
Thus, for a vector $\hat \phi$ from $\hat \hilb$ the following equivalence holds:
\begin{equation} \label{F: criterion for hat phi(2)}
  \hat \phi \in \Upsilon_\lambda^2 \ \ \iff \ \  
         \hat \phi = - (s-\rlmb)^2 R_\lambda(\hat H_{\rlmb}) v \euD_\lambda(s) v^* \euF_\lambda(s) \hat \phi.
\end{equation} 
In particular, any resonance vector $\hat \phi$ of order two from $\hat \hilb$ belongs 
to the image of the operator~$R_\lambda(\hat H_{\rlmb}) v.$
The following theorem generalises this statement for vectors of arbitrary order. 
\begin{thm} \label{T: thm 2.4} Let $k\geq 2.$ 
A resonance vector $\hat \phi$ of order $k$ from $\hat \hilb$ 
belongs to the linear span of the images of operators 
$$
  R_\lambda(\hat H_{\rlmb})v, \ \ \hat A_\lambda(\rlmb)R_\lambda(\hat H_{\rlmb})v, 
    \ \ \ldots, \ \ \hat A^{k-2}_\lambda(\rlmb)R_\lambda(\hat H_{\rlmb})v.
$$
\end{thm}
\begin{proof} The induction base with $k=2$ follows from (\ref{F: criterion for hat phi(2)}).
Assume that the claim holds for a vector of order less than~$k$ and let $\hat \phi_k \in \hilb$
be a vector of order~$k.$ Since the operator $1+(\rlmb-s)A_\lambda(s)$ decreases order of a resonance vector by one,
the vector $\hat \phi_{k-1}(s) = \brs{1+(\rlmb-s)A_\lambda(s)} \hat \phi_k$ has order $k-1.$ 
Using $2 \times 2$ representation (\ref{F: 2x2 repr-n of 1-r A lamb(r)}) of the operator $1+(\rlmb-s)A_\lambda(s),$ we have
$$
  \SqBrs{1  + (s-\rlmb)^2 R_\lambda(\hat H_s)v \euD_\lambda(s) v^*} \euF_\lambda(s) \hat \phi_k = \hat \phi_{k-1}(s)
$$
This equality can be rewritten as 
$$
  \hat \phi_k + (s-\rlmb)^2 \euF^{-1}_\lambda(s) R_\lambda(\hat H_s)v \euD_\lambda(s) v^* \euF_\lambda(s)\hat \phi_k
         = \euF^{-1}_\lambda(s)\hat \phi_{k-1}(s),
$$
and by (\ref{F: euF(-1)=...}) this equality is equivalent to 
$$
  \hat \phi_k = - (s-\rlmb)^2 R_\lambda(\hat H_{\rlmb})v \euD_\lambda(s) v^* \euF_\lambda(s)\hat \phi_k + (1+(s-\rlmb)\hat A_\lambda(\rlmb)) \hat \phi_{k-1}(s).
$$
Using induction assumption, this equality proves the claim. 
\end{proof}

\subsection{The operator~$S_\lambda$}

In what follows the product of operators $R_\lambda(\hat H_{\rlmb})$ and~$V$ will be encountered very often;
for this reason, we shall introduce notation\label{Page: S(lambda)}
\begin{equation} \label{F: S(lambda)}
  \begin{split}
     S_\lambda & = R_\lambda(\hat H_{\rlmb})V
     \\ & 
       = \left(
         \begin{matrix}
           \hat A_\lambda(\rlmb) & R_\lambda(\hat H_{\rlmb}) v \\
           0 & 0 
         \end{matrix}
       \right).
  \end{split}
\end{equation}
Restriction of the power $S^k_\lambda$ to $\hat \hilb$ coincides with $\hat A^k_\lambda(\rlmb),$
restriction of this power to $\clV_\lambda$ coincides with $\hat A^{k-1}_\lambda(\rlmb)R_\lambda(\hat H_{\rlmb})v.$

\begin{lemma}  \label{L: new lemma}
For any $\hat f \in \hat \hilb$ 
$$
  (1+(\rlmb-s)A_\lambda(s)) S_\lambda \hat f = A_\lambda(s) \hat f + 
       \left(
         \begin{matrix}
          (s-\rlmb) R_\lambda(\hat H_s)v \euD_\lambda(s) v^* \hat f\\
           - \euD_\lambda(s)  v^* \hat f
         \end{matrix}
       \right)
$$
\end{lemma}
\begin{proof} Since $\hat f \in \hat \hilb,$ we have~$S_\lambda \hat f = \hat A_\lambda(\rlmb) \hat f.$
Hence, 
$$
  (1+(\rlmb-s)A_\lambda(s)) S_\lambda \hat f = (1+(\rlmb-s)A_\lambda(s)) \hat A_\lambda(\rlmb) \hat f.
$$
From (\ref{F: euF(-1)=...}) we have 
$$
  \hat A_\lambda(\rlmb) = (s-\rlmb)^{-1} (\euF^{-1}_\lambda(r)-1).
$$
Hence, combining the last two equalities 
and using $2 \times 2$ representation~(\ref{F: 2x2 repr-n of 1-r A lamb(r)}) of the operator
$1+(\rlmb-s)A_\lambda(s),$ we get 
\begin{equation*}
  \begin{split}
      (1+(\rlmb-s)A_\lambda(s)) S_\lambda \hat f & = (1+(\rlmb-s)A_\lambda(s)) \hat A_\lambda(\rlmb) \hat f
      \\ & = (s-\rlmb)^{-1} (1+(\rlmb-s)A_\lambda(s)) (\euF^{-1}_\lambda(r)-1) \hat f
       \\ & = 
       - (s-\rlmb)^{-1}(1+(\rlmb-s)A_\lambda(s)) \hat f 
       \\ & \hskip 1.2 cm 
        + (s-\rlmb)^{-1} \left(
         \begin{matrix}
           \SqBrs{1 + (s-\rlmb)^2 R_\lambda(\hat H_s)v \euD_\lambda(s) v^*} \hat f \\
           (\rlmb-s)\euD_\lambda(s)  v^* \hat f
         \end{matrix}
       \right)
       \\ & = 
       A_\lambda(s) \hat f + 
       \left(
         \begin{matrix}
          (s-\rlmb) R_\lambda(\hat H_s)v \euD_\lambda(s) v^* \hat f \\
           - \euD_\lambda(s)  v^* \hat f
         \end{matrix}
       \right).
   \end{split}
\end{equation*}
\end{proof}

\begin{thm} \label{T: 1st thm of section 3}
If a vector~$\chi \in \hilb$ is such that the vector 
$V\chi$ is orthogonal to the eigenspace~$\clV_\lambda,$
then
\begin{equation}  \label{F: fcukingly nice equality}
  (1+(\rlmb-s)A_\lambda(s))S_\lambda \chi = A_\lambda(s) \chi.
\end{equation}
\end{thm}
\begin{proof} 
(A) The premise $V\chi \perp \clV_\lambda$ means that the second component
of $V\chi$ in the direct sum $\hilb = \hat \clH \oplus \clV_\lambda$ is zero, that is, 
\begin{equation} \label{F: v* chi+aaa chi=0}
  v^* \chi + \aaa \chi = 0.
\end{equation}

(B) If $V\chi \perp \clV_\lambda,$ then 
\begin{equation} \label{F: V chi perp part (B)}
    (1+(\rlmb-s)A_\lambda(s)) R_\lambda(\hat H_{\rlmb}) v \chi 
       =  
       \left(
         \begin{matrix}
            (\rlmb-s) R_\lambda(\hat H_s)v \euD_\lambda(s)v^* \chi \\
             (s-\rlmb)^{-1} \chi + \euD_\lambda(s)v^* \chi
         \end{matrix}
       \right).
\end{equation}

Proof of (\ref{F: V chi perp part (B)}). 

Using (\ref{F: 2x2 repr-n of 1-r A lamb(r)}), we have 
\begin{equation} \label{F: bloody equality}
  \begin{split}
    (E) := (1+(\rlmb-s)A_\lambda(s)) R_\lambda(\hat H_{\rlmb}) v \chi & = 
       \left(
         \begin{matrix}
           \SqBrs{1 + (s-\rlmb)^2 R_\lambda(\hat H_s)v \euD_\lambda(s) v^*}\euF_\lambda(s) R_\lambda(\hat H_{\rlmb}) v \chi\\
           (\rlmb-s)\euD_\lambda(s)  v^* \euF_\lambda(s) R_\lambda(\hat H_{\rlmb}) v \chi
         \end{matrix}
       \right)
       \\ & = 
       \left(
         \begin{matrix}
           \SqBrs{1 + (s-\rlmb)^2 R_\lambda(\hat H_s)v \euD_\lambda(s) v^*} R_\lambda(\hat H_s)v\chi \\
           (\rlmb-s)\euD_\lambda(s)  v^* R_\lambda(\hat H_s)v\chi
         \end{matrix}
       \right),
  \end{split}
\end{equation}
where in the second equality the second resolvent identity (\ref{F: euF(r)Az(0)=Az(r)}) is used. 
The second component of this vector can be transformed as follows:
\begin{equation*}
  \begin{split}
  (\rlmb-s)\euD_\lambda(s)  v^* R_\lambda(\hat H_s)v\chi 
         & = \euD_\lambda(s)\SqBrs{\aaa \chi + (\rlmb-s)v^* R_\lambda(\hat H_s)v\chi} - \euD_\lambda(s)\aaa \chi
      \\ & = (s-\rlmb)^{-1} \hat P^\perp \chi - \euD_\lambda(s)\aaa \chi,
  \end{split}
\end{equation*}
where the second equality follows from the definition (\ref{F: euD lambda(r)}) of~$\euD_\lambda(s).$
Combining this with (\ref{F: v* chi+aaa chi=0}) gives 
\begin{equation*}
  \begin{split}
  (\rlmb-s)\euD_\lambda(s)  v^* R_\lambda(\hat H_s)v\chi 
        = (s-\rlmb)^{-1} \hat P^\perp \chi + \euD_\lambda(s)v^* \chi.
  \end{split}
\end{equation*}
Substituting the right hand side into (\ref{F: bloody equality})
yields 
\begin{equation} \label{F: third234}
  \begin{split}
    (E) & = 
       \left(
         \begin{matrix}
           R_\lambda(\hat H_s)v \SqBrs{\chi + (s-\rlmb)^2 \euD_\lambda(s) v^* R_\lambda(\hat H_s)v\chi} \\
             (s-\rlmb)^{-1} \chi + \euD_\lambda(s)v^* \chi
         \end{matrix}
       \right)
      \\ & =  
       \left(
         \begin{matrix}
           R_\lambda(\hat H_s)v \SqBrs{\chi +(\rlmb-s) \brs{(s-\rlmb)^{-1} \hat P^\perp \chi + \euD_\lambda(s)v^* \chi}} \\
             (s-\rlmb)^{-1} \chi + \euD_\lambda(s)v^* \chi
         \end{matrix}
       \right)
      \\ & =  
       \left(
         \begin{matrix}
           (\rlmb-s) R_\lambda(\hat H_s)v \euD_\lambda(s)v^* \chi \\
             (s-\rlmb)^{-1} \chi + \euD_\lambda(s)v^* \chi
         \end{matrix}
       \right),
  \end{split}
\end{equation}
where in the last equality we used $v \hat P^\perp = v.$ 

\smallskip

(C)
We have 
\begin{equation} \label{F: first234}
  \begin{split}
     (1+(\rlmb-s ) A_\lambda(s))S_\lambda \chi & = (1+(\rlmb-s)A_\lambda(s)) R_\lambda(\hat H_{\rlmb}) V \chi
     \\ & = (1+(\rlmb-s)A_\lambda(s)) R_\lambda(\hat H_{\rlmb}) \hat V \chi 
            + (1+(\rlmb-s)A_\lambda(s)) R_\lambda(\hat H_{\rlmb}) v \chi.
  \end{split}
\end{equation}
Applying Lemma~\ref{L: new lemma} to the vector $\hat P \chi = \hat \chi,$ we transform the first summand as follows
\begin{equation} \label{F: second234}
  (1+(\rlmb-s)A_\lambda(s))R_\lambda(\hat H_{\rlmb}) \hat V \chi 
    = A_\lambda(s) \hat \chi + 
       \left(
         \begin{matrix}
          (s-\rlmb) R_\lambda(\hat H_s)v \euD_\lambda(s) v^* \hat \chi\\
           - \euD_\lambda(s)  v^* \hat \chi
         \end{matrix}
       \right).
\end{equation}

Since $v^* \chi = v^* \hat \chi,$ combining (\ref{F: first234}), (\ref{F: second234}) 
and (\ref{F: V chi perp part (B)}) yields the equality

\begin{equation*}
    (1+(\rlmb-s)A_\lambda(s))S_\lambda \chi = A_\lambda(s) \hat \chi + (s-\rlmb)^{-1} \hat P^\perp \chi.
\end{equation*}
Since $\hat P^\perp \chi$ is an eigenvector of $H_{\rlmb},$ we have 
$A_\lambda(s) \hat P^\perp \chi = (s-\rlmb)^{-1} \hat P^\perp \chi.$
Hence, 
\begin{equation*}
    (1+(\rlmb-s)A_\lambda(s))S_\lambda \chi 
       = A_\lambda(s) \hat \chi +  A_\lambda(s) \hat P^\perp \chi
       = A_\lambda(s) \chi.
\end{equation*}
\end{proof}

The argument of this proof shows that in general (without the assumption $V\chi \perp \clV_\lambda$)
we have 
\begin{equation}  \label{F: fcukingly nice equality(2)}
  (1+(\rlmb-s)A_\lambda(s))S_\lambda \chi = A_\lambda(s) \chi + 
       \left(
         \begin{matrix}
           (s-\rlmb) R_\lambda(\hat H_s)v \euD_\lambda(s)(v^*+\aaa) \chi \\
             - \euD_\lambda(s)(v^*+\aaa) \chi
         \end{matrix}
       \right).
\end{equation}

\subsection{Depth of resonance vectors and depth criteria}

Recall that a resonance vector~$\chi$ has \emph{depth} \label{Page: depth of vector(2)}
at least~$k$ if there exists a resonance vector~$\chi_1$ such that 
$$ 
  \chi = (1+(\rlmb-s)R_\lambda(H_s)V)^k\chi_1.
$$ 
Equivalently, a resonance vector~$\chi$ has depth 
at least~$k$ if there exists a resonance vector~$\chi_1$ such that 
$ 
  \chi = \bfA_\lambda^k\chi_1.
$ 
A resonance vector has depth~$k$ if it has depth at least~$k$ but not at least~$k+1.$

\smallskip
The following theorem provides a criterion for a resonance vector to have depth $\geq 1.$

\begin{thm} \label{T: TFAE depth 1 criterion} 
Let~$\chi$ be a resonance vector. 
The following assertions are equivalent.
\begin{itemize}
  \item[(i)] The vector $V\chi$ is orthogonal to $\clV_\lambda.$
  \item[(ii)] The depth of the vector~$\chi$ is at least~$1.$
  \item[(iii)] The equality~$\bfA_\lambda(r_\lambda) S_\lambda \chi = - \chi$ holds. 
\end{itemize}
\end{thm}
\begin{proof} (i) $\then$ (iii).
By (\ref{F: Az (3.3.16)}),
the meromorphic operator-function $A_\lambda(s)$ has 
the following Laurent expansion at $s = \rlmb:$ 
\begin{equation} \label{F: Laurent for A lambda(r)}
  A_\lambda(s) = \tilde A_\lambda(s) + \sum_{j=0}^{d-1} (s-\rlmb)^{-j-1} \bfA^{j}_\lambda,
\end{equation}
where $\bfA^0_\lambda = P_\lambda.$ 
From this we obtain
$$
  1+(\rlmb-s)A_\lambda(s) 
         = (\rlmb-s)\tilde A_\lambda(s) + (1 - P_\lambda) 
         - \sum_{j=1}^{d-1} (s-\rlmb)^{-j} \bfA^{j}_\lambda.
$$
Applying the latter series to the vector~$S_\lambda \chi$ and retaining only the term $(s-\rlmb)^{-1}$  
gives 
$$
  (1+(\rlmb-s)A_\lambda(s)) S_\lambda \chi = \ldots - (s-\rlmb)^{-1} \bfA_\lambda S_\lambda \chi + \ldots.
$$
Since~$\chi$ is a resonance vector, the former series applied to the vector~$\chi$
also gives
$$
  A_\lambda(s) \chi = \ldots + (s-\rlmb)^{-1} \chi
      + \ldots.
$$
Since $V\chi \perp \clV_\lambda,$ according to Theorem~\ref{T: 1st thm of section 3}
the last two Laurent expansions are equal and therefore
$
  \bfA_\lambda S_\lambda \chi = - \chi.
$

\smallskip
(iii) $\then$ (ii). This is obvious.

\smallskip
(ii) $\then$ (i). 
Let~$\chi$ be a resonance vector of depth at least one. By definition of depth,
there exists a resonance vector~$\chi'$ such that 
$
 \bfA_\lambda\chi' = \chi.
$
Hence, for any eigenvector $\phi \in \clV_\lambda$ we have 
\begin{equation*}
    \scal{V\chi}{\phi}
         = \scal{V\bfA_\lambda\chi'}{\phi} \\
         = \scal{V\chi'}{\bfA_\lambda\phi} \\
         = 0,
\end{equation*}
where the last equality holds since~$\bfA_\lambda(r_\lambda)$ eliminates any eigenvector. 

\end{proof}


{\bf Question.} Are the items of this theorem also equivalent to this one: (iv)~$S_\lambda \chi$ is a resonance vector?
The implication (iii) $\then$ (iv) is obvious, so the question is whether (iv) $\then$ (iii)?



\subsection{$2 \times 2$ representations of $P_\lambda$ and~$\bfA_\lambda(r_\lambda)$}

The meromorphic operator-valued function $\euD_\lambda(s)$ 
defined in (\ref{F: euD lambda(r)})
acts on the finite-dimensional Hilbert space~$\clV_\lambda.$
Since the function 
$$
  \euD^{-1}_\lambda(s) = (s-\rlmb)\aaa - (s-\rlmb)^2 v^* R_\lambda(\hat H_s)v
$$ 
is holomorphic at $s = r_\lambda,$ by the analytic Fredholm alternative, at the pole $r_\lambda$ 
the Laurent series of the function $\euD_\lambda(s)$ has finitely many terms with negative powers.

\begin{lemma} The meromorphic function $\euD_\lambda(s)$ 
has a pole of order $d$ at $s = \rlmb.$ 
\end{lemma}
\begin{proof}
Since by (\ref{F: Laurent for A lambda(r)}) the operator~$A_\lambda(s)$ has 
a pole of order $d$ at $s = \rlmb,$ it follows from $2 \times 2$ representation 
(\ref{F: 2x2 repr-n of A lamb(r)}) of this operator that the function~$\euD_\lambda(s)$ 
has a pole of order at least~$d$ at~$s = \rlmb,$ since other factors in this $2 \times 2$ representation 
are holomorphic at~$s=r_\lambda.$
That the order of this pole is not greater than~$d$
can be observed from the second resolvent identity
$$
  R_\lambda(H_s) = R_\lambda(H_{s_0}) - (s-s_0)R_\lambda(H_s)V R_\lambda(H_{s_0})
$$
and the fact that $\euD_\lambda(s)$ is the $(2,2)$-entry of $R_\lambda(H_s),$
since the right hand side of the resolvent identity above has a pole of order at most~$d.$
\end{proof}
This lemma shows that the Laurent expansion of the function $\euD_\lambda(s)$ 
can be written in the following form:\label{Page: Laurent for euD}
\begin{equation} \label{F: Laurent for euD}
  \euD_\lambda(s) = \sum_{j=-d+1}^\infty D_{-j} (s-\rlmb)^{j-1}. 
\end{equation}
The additional factor $(s-\rlmb)^{-1}$ is introduced here for convenience; also,
since we shall be working mainly with coefficients of negative powers of~$(s-\rlmb),$
we choose to denote the coefficient of $(s-\rlmb)^{j-1}$ by $D_{-j}.$ 
Since the meromorphic function $\euD_\lambda(s)$ depends only on the triple $(\lambda; H_{\rlmb},V),$
the operators $D_1, \ldots, D_{d-1}$ are also invariants of this triple. 
Further, since $\euD_\lambda(s)$ is self-adjoint for real values of~$s,$ the operators $D_j$ are self-adjoint. 

\begin{lemma} \label{L: long unused lemma}
We have
\begin{equation*}
  \begin{split}
     D_{d-1} \aaa &= 0, \\
     D_{d-2} \aaa &= D_{d-1}v^*R_\lambda(\hat H_{\rlmb}) v, \\
     D_{d-3} \aaa &= D_{d-2}v^*R_\lambda(\hat H_{\rlmb}) v - D_{d-1}v^* \hat A_\lambda(\rlmb)R_\lambda(\hat H_{\rlmb}) v, \\
        \ldots \\
     D_{1} \aaa &= D_{2}v^*R_\lambda(\hat H_{\rlmb}) v - D_{3}v^* \hat A_\lambda(\rlmb)R_\lambda(\hat H_{\rlmb}) v 
       + \ldots + (-1)^{d-3} D_{d-1} v^* \hat A^{d-3}_\lambda(\rlmb) R_\lambda(\hat H_{\rlmb}) v, \\
     D_{0} \aaa &= D_{1}v^*R_\lambda(\hat H_{\rlmb}) v - D_{2}v^* \hat A_\lambda(\rlmb)R_\lambda(\hat H_{\rlmb}) v 
       + \ldots + (-1)^{d-2} D_{d-1} v^* \hat A^{d-2}_\lambda(\rlmb) R_\lambda(\hat H_{\rlmb}) v + 1, \\
     D_{-1} \aaa &= D_{0}v^*R_\lambda(\hat H_{\rlmb}) v - D_{1}v^* \hat A_\lambda(\rlmb)R_\lambda(\hat H_{\rlmb}) v 
       + \ldots + (-1)^{d-1} D_{d-1} v^* \hat A^{d-1}_\lambda(\rlmb) R_\lambda(\hat H_{\rlmb}) v. \\
  \end{split}
\end{equation*}
\end{lemma}
\begin{proof}
  By definition of $\euD_\lambda(s),$ we have 
  $
    1 = (s-\rlmb) \euD_\lambda(s)(\aaa +(\rlmb-s)v^* R_\lambda(\hat H_s) v),
  $
  that is,
  \begin{equation*}
    \begin{split}
      1 & = \sum_{j=-d+1}^\infty D_{-j} (s-\rlmb)^j
        \brs{\aaa +\sum_{j=0}^\infty (-1)^{j+1}(s-\rlmb)^{j+1} v^*\hat A^j_\lambda(\rlmb)R_\lambda(\hat H_{\rlmb}) v}.
    \end{split}
  \end{equation*}
  Comparing powers of $s-\rlmb$ on both sides gives the required equalities. 
\end{proof}

\smallskip
The operators $P_\lambda$ and~$\bfA_\lambda^k$ are coefficients 
of the Laurent series of the meromorphic function~$A_\lambda(s).$
Hence, using (\ref{F: 2x2 repr-n of A lamb(r)}), (\ref{F: Laurent for euD}) 
and the Neumann series 
$$
  \euF_\lambda(s) = \sum_{k=0}^\infty (-1)^k (s-\rlmb)^k A^k_\lambda(\rlmb),
$$
one can calculate $2 \times 2$ representations of these operators. We omit the straightforward 
calculations and present only the resulting formulas; to simplify them, we use notations\label{Page: Y(j)}
$$
  Y_j := R_\lambda(\hat H_{\rlmb}) v D_j v^*
$$
and 
$$
  \set{\hat A^l_\lambda(\rlmb),Y_j} = \hat A^l_\lambda(\rlmb)Y_j + \hat A^{l-1}_\lambda(\rlmb) Y_j \hat A_\lambda(\rlmb) 
    + \hat A^{l-2}_\lambda(\rlmb) Y_j \hat A^2_\lambda(\rlmb) + \ldots + Y_j \hat A^l_\lambda(\rlmb).
$$

\begin{thm} $2 \times 2$ representations of operators~$\bfA_\lambda^j(\rlmb),$ $j=1,\ldots,d-1,$ $P_\lambda(\rlmb)$ 
and $\tilde A_\lambda(\rlmb)$ are given by formulas
\begin{equation} \label{F: A(j)=2x2}
   \bfA^j_\lambda =
    \left(
       \begin{matrix}
    -Y_{j+1} + \set{\hat A_\lambda(\rlmb),Y_{j+2}} - \set{\hat A^2_\lambda(\rlmb),Y_{j+3}} + \ldots + (-1)^{d-j-1}\set{\hat A^{d-j-2}_\lambda(\rlmb),Y_{d-1}}  &  0  \\        
         D_j v^* - D_{j+1} v^* \hat A_\lambda(\rlmb) + D_{j+2} v^* \hat A^2_\lambda(\rlmb) - \ldots +(-1)^{d-j-1} D_{d-1}v^* \hat A^{d-j-1}_\lambda(\rlmb) &  0
       \end{matrix}
    \right),
\end{equation}
\begin{equation} \label{F: A(0)=2x2}
  P_\lambda =
    \left(
       \begin{matrix}
         -Y_1 + \set{\hat A_\lambda(\rlmb),Y_2} - \set{\hat A^2_\lambda(\rlmb),Y_3} + \ldots + (-1)^{d-1}\set{\hat A^{d-2}_\lambda(\rlmb),Y_{d-1}}  &  0  \\
         D_0 v^* - D_1 v^* \hat A_\lambda(\rlmb) + D_2 v^* \hat A^2_\lambda(\rlmb) - \ldots +(-1)^{d-1} D_{d-1}v^* \hat A^{d-1}_\lambda(\rlmb) &  1
       \end{matrix}
    \right),
\end{equation}
and 
\begin{equation} \label{F: tilde A(0)=2x2}
  \tilde A_\lambda(\rlmb) =
    \left(
       \begin{matrix}
         \hat A_\lambda(\rlmb) - Y_0 + \set{\hat A_\lambda(\rlmb),Y_1} - \set{\hat A^2_\lambda(\rlmb),Y_2} + \ldots + (-1)^{d-1}\set{\hat A^{d-1}_\lambda(\rlmb),Y_{d-1}}  &  0  \\
         D_{-1}v^* - D_0 v^* \hat A_\lambda(\rlmb) + D_1 v^* \hat A^2_\lambda(\rlmb) - \ldots +(-1)^{d} D_{d-1}v^* \hat A^{d}_\lambda(\rlmb) &  0
       \end{matrix}
    \right).
\end{equation}
In particular, 
\begin{equation*} 
   \bfA^{d-1}_\lambda =
    \left(
       \begin{matrix}
          0     &  0  \\
         D_{d-1}v^*   &  0
       \end{matrix}
    \right),
    \qquad 
   \bfA^{d-2}_\lambda =
    \left(
       \begin{matrix}
          - Y_{d-1}   &  0  \\
         D_{d-2}v^* - D_{d-1} v^* \hat A_\lambda(\rlmb) &  0
       \end{matrix}
    \right).
\end{equation*}
\end{thm}

\smallskip 
\noindent These formulas generalise those given in \cite[\S 14.4]{Az9} considering that in this setting 
the eigenvalue~$\lambda$ is degenerate. Note that the $(2,2)$-entry of $P_\lambda$ is~$1.$

\begin{thm} \label{T: -S A(j)=A(j-1)} In the~$2 \times 2$ representation, for all $j=1,2,\ldots,d,$
the only non-zero entry of the operator $-S_\lambda \bfA_\lambda^j$
is the $(1,1)$-entry which is equal to the $(1,1)$-entry of $\bfA^{j-1}_\lambda.$ 
In other words,
\begin{equation} \label{F: -S A(j)=A(j-1)}
  - S_\lambda \bfA_\lambda^j = \hat P \bfA^{j-1}_\lambda. 
\end{equation}
\end{thm}
\begin{proof}
Proof is a direct calculation based on $2 \times 2$ representation
(\ref{F: A(j)=2x2}) of the operator~$\bfA_\lambda^j.$
For $j \geq 1,$ we have 

\begin{equation*}
  \begin{split}
    S_\lambda \bfA_\lambda^j(r_\lambda) & = 
       \left(
         \begin{matrix}
            \hat A_\lambda(\rlmb) & R_\lambda(\hat H_{\rlmb}) v \\
            0 & 0  
         \end{matrix}
       \right)
    \left(
       \begin{matrix}
       \sum_{k=0}^{d-j-2} (-1)^{k+1} \set{\hat A^k_\lambda(\rlmb),Y_{k+j+1}} & 0 \\
       \sum_{k=0}^{d-j-1} (-1)^k D_{j+k} v^* \hat A^k_\lambda(\rlmb) & 0 
       \end{matrix}
    \right).
   \end{split}
\end{equation*}
The only non-zero entry of this product is $(1,1)$-entry,
which is equal to the following expression
\begin{equation*}
  \begin{split}
    & -\hat A_\lambda(\rlmb)Y_{j+1} + \hat A_\lambda(\rlmb)\set{\hat A_\lambda(\rlmb),Y_{j+2}} - \ldots + (-1)^{d-j-1}\hat A_\lambda(\rlmb)\set{\hat A^{d-j-2}_\lambda(\rlmb),Y_{d-1}} 
    \\ & \qquad + R_\lambda(\hat H_{\rlmb}) v D_j v^* - R_\lambda(\hat H_{\rlmb}) v D_{j+1} v^* \hat A_\lambda(\rlmb) + \ldots +(-1)^{d-j-1} R_\lambda(\hat H_{\rlmb}) v D_{d-1}v^* \hat A^{d-j-1}_\lambda(\rlmb) 
    \\ & = -\hat A_\lambda(\rlmb)Y_{j+1} + \hat A_\lambda(\rlmb)\set{\hat A_\lambda(\rlmb),Y_{j+2}} - \ldots + (-1)^{d-j-1}\hat A_\lambda(\rlmb)\set{\hat A^{d-j-2}_\lambda(\rlmb),Y_{d-1}} 
    \\ & \qquad + Y_j - Y_{j+1} \hat A_\lambda(\rlmb) + \ldots +(-1)^{d-j-1} Y_{d-1} \hat A^{d-j-1}_\lambda(\rlmb). 
  \end{split}
\end{equation*}
Re-arranging the summands in this expression we see that it is equal to 
$$
  Y_j - \set{\hat A_\lambda(\rlmb), Y_{j+1}} + \set{\hat A^2_\lambda(\rlmb), Y_{j+2}} - \ldots + (-1)^{d-j-1}\set{\hat A^{d-j-1}_\lambda(\rlmb),Y_{d-1}}
$$
Since this is the negative of the $(1,1)$-entry of the operator $\bfA^{j-1}_\lambda,$ 
we are done. 
\end{proof}

Subtracting~$\bfA_\lambda^{j-1}$ from both sides of (\ref{F: -S A(j)=A(j-1)}) 
and replacing $j-1$ by $j$ in the resulting formula gives for each $j=0,1,2,\ldots,d-1$ 
$$ 
  (1 + S_\lambda \bfA_\lambda)\bfA_\lambda^j = \hat P^\perp \bfA_\lambda^j.
$$
Since $\hat P^\perp$ is the orthogonal projection onto the eigenspace $\clV_\lambda,$
combining this equality with (\ref{F: A(j)=2x2}) gives for $j=1,2,\ldots$
\begin{equation} \label{F: S A(j+1) + A(j) = Djv*+...}
  (1 + S_\lambda \bfA_\lambda) \bfA^{j}_\lambda 
        = \sum_{k=0}^{d-1-j} (-1)^k D_{j+k} v^* \hat A_\lambda^k(\rlmb).
\end{equation}
If $j=0,$ then we have to use the formula (\ref{F: A(0)=2x2}) which has non-zero $(2,2)$-entry equal to~$1.$
This entry results in the additional summand $\hat P^\perp.$ 
Hence, we have 
\begin{equation} \label{F: S A(1) + A(0) = D0v*+...+1}
  (1 + S_\lambda \bfA_\lambda) P_\lambda 
        = \sum_{k=0}^{d-1} (-1)^k D_{k} v^* \hat A_\lambda^k(\rlmb) + \hat P^\perp.
\end{equation}

\begin{thm} For any $j = 1,2,\ldots,d-1$ 
\begin{equation*}
    D_{j}v^* = (1 + S_\lambda \bfA_\lambda) \bfA_\lambda^{j} (1 + \bfA_\lambda \hat A_\lambda(\rlmb)).
\end{equation*}
Also, 
\begin{equation*}
    D_{0}v^* = - \hat P^\perp + (1 + S_\lambda \bfA_\lambda) P_\lambda (1 + \bfA_\lambda \hat A_\lambda(\rlmb)).
\end{equation*}
\end{thm}
\begin{proof}
Using the equalities (\ref{F: S A(j+1) + A(j) = Djv*+...}), we obtain formulas
\begin{equation*}
  \begin{split}
     D_{d-1}v^* & = \bfA_\lambda^{d-1}, \\
     & \\
     D_{d-2}v^* & = \bfA_\lambda^{d-2} + S_\lambda \bfA_\lambda^{d-1} + D_{d-1}v^*\hat A_\lambda(\rlmb) \\
                & = \bfA_\lambda^{d-2} + S_\lambda \bfA_\lambda^{d-1} + \bfA_\lambda^{d-1}\hat A_\lambda(\rlmb)\\
                & = (1 + S_\lambda \bfA_\lambda) \bfA_\lambda^{d-2} (1 + \bfA_\lambda \hat A_\lambda(\rlmb)), \\
     & \\           
      D_{d-3}v^*
        & = \bfA_\lambda^{d-3} + S_\lambda \bfA_\lambda^{d-2} + D_{d-2}v^* \hat A_\lambda(\rlmb) - D_{d-1}v^* \hat A_\lambda^2(\rlmb) \\
        & = \bfA_\lambda^{d-3} + S_\lambda \bfA_\lambda^{d-2}
                + \brs{\bfA_\lambda^{d-2} + S_\lambda \bfA_\lambda^{d-1}
                + \bfA_\lambda^{d-1}\hat A_\lambda(\rlmb)}\hat A_\lambda(\rlmb) - \bfA_\lambda^{d-1} \hat A_\lambda^2(\rlmb) \\
        & = \bfA_\lambda^{d-3} + S_\lambda \bfA_\lambda^{d-2}
                + \bfA_\lambda^{d-2}\hat A_\lambda(\rlmb) + S_\lambda \bfA_\lambda^{d-1}\hat A_\lambda(\rlmb)  \\
        & = (1 + S_\lambda \bfA_\lambda) \bfA_\lambda^{d-3} (1 + \bfA_\lambda \hat A_\lambda(\rlmb)).
  \end{split}
\end{equation*}
We proceed by induction. 
Moving all summands of the right side of (\ref{F: S A(j+1) + A(j) = Djv*+...}), 
except the first one, to the left side and then using the induction assumption we obtain 
\begin{equation*}
  \begin{split}
      D_{j}v^* 
       & = (1 + S_\lambda \bfA_\lambda) \bfA_\lambda^{j} 
              + \sum_{k=1}^\infty (-1)^{k+1} D_{j+k} v^* \hat A_\lambda^k(\rlmb)
    \\ & = (1 + S_\lambda \bfA_\lambda) \bfA_\lambda^{j} 
              + \sum_{k=1}^\infty (-1)^{k+1} (1 + S_\lambda \bfA_\lambda) \bfA_\lambda^{j+k} (1 + \bfA_\lambda \hat A_\lambda(\rlmb)) \hat A_\lambda^k(\rlmb)
    \\ & = (1 + S_\lambda \bfA_\lambda) \bfA_\lambda^{j} \brs{1
              + \sum_{k=1}^\infty (-1)^{k+1} \bfA_\lambda^{k} (1 + \bfA_\lambda \hat A_\lambda(\rlmb)) \hat A_\lambda^k(\rlmb)}
    \\ & = (1 + S_\lambda \bfA_\lambda) \bfA_\lambda^{j} \brs{1
              + \sum_{k=1}^\infty (-1)^{k+1} \bfA_\lambda^{k} \hat A_\lambda^k(\rlmb)
              + \sum_{k=1}^\infty (-1)^{k+1} \bfA_\lambda^{k+1} \hat A_\lambda^{k+1}(\rlmb)}
    \\ & =  (1 + S_\lambda \bfA_\lambda) \bfA_\lambda^{j} \brs{1 + \bfA_\lambda \hat A_\lambda(\rlmb)}.
  \end{split}
\end{equation*}
The appearance of the additional summand $-\hat P^\perp$ in case of $j=0$ 
was explained before the statement of this theorem. We also remark that the infinite sums above are in fact finite,
but for simplicity the upper summation indexes are replaced by infinity. 
\end{proof}

Since $\aaa S_\lambda = 0,$
this theorem implies the equality 
\begin{equation} \label{F: alpha Dj v* = alpha bfA(j)(1+bfA hat A(0))}
  \aaa D_{j}v^* = \aaa \bfA_\lambda^{j} (1 + \bfA_\lambda \hat A_\lambda(\rlmb)), \ j = 1, 2, \ldots, d-1.
\end{equation}

Using $2 \times 2$ representations of operators $P_\lambda, \bfA_\lambda, \tilde A_\lambda(\rlmb)$ 
and Lemma~\ref{L: long unused lemma}, straightforward but somewhat lengthy calculations 
prove the following relations:
$$
  P_\lambda + \bfA_\lambda S_\lambda = \sum_{l=0}^{d-1} (-1)^{l} S^l_\lambda D_l (v^*+\aaa),
$$
and
$$ 
  \tilde A_\lambda(\rlmb) + P_\lambda S_\lambda = S_\lambda + \sum_{l=-1}^{d-1} (-1)^{l+1} S_\lambda^{l+1} D_l (v^*+\aaa),
$$
where $S^0_\lambda$ is the identity operator on the whole Hilbert space~$\hilb.$
Since these formulas are not used further, their proofs are omitted. 
We note that, since a vector~$\chi$ has the property $V\chi \perp \clV_\lambda$ if and only if 
$(v^*+\aaa)\chi = 0,$ Theorem~\ref{T: TFAE depth 1 criterion} immediately 
follows from the first of these formulas. 

\subsection{Resonance points with property $B$}

\begin{thm} \label{T: TFAE prop B}
The following assertions are equivalent:
\begin{enumerate}
  \item \label{prop B 1} $\aaa D_j = 0$ for all $j=1,2,\ldots,d-1.$
  \item \label{prop B 2} $\aaa D_j v^* = 0$ for all $j=1,2,\ldots,d-1.$
  \item \label{prop B 3} $\aaa \bfA_\lambda = 0.$
  \item \label{prop B 4} $v^* \bfA_\lambda = 0.$
  \item \label{prop B 5} The function $(s-\rlmb) \euD_\lambda(s) \aaa$ is holomorphic at $s = r_\lambda.$
  \item \label{prop B 6} The function $(s-\rlmb) v \euD_\lambda(s) \aaa$ is holomorphic at $s = r_\lambda.$
  \item \label{prop B 7}
$
  \im\brs{\hat P^\perp \bfA_\lambda} \subset \im\brs{\bfA_\lambda}.
$
\end{enumerate}
\end{thm}
\begin{proof} We prove the following equivalences:
(\ref{prop B 1}) $\iff$ (\ref{prop B 2}), (\ref{prop B 2}) $\iff$ (\ref{prop B 3}),
(\ref{prop B 3}) $\iff$ (\ref{prop B 4}),
(\ref{prop B 1}) $\iff$ (\ref{prop B 5}),
(\ref{prop B 2}) $\iff$ (\ref{prop B 6}),
(\ref{prop B 3}) $\iff$ (\ref{prop B 7}).

(\ref{prop B 1}) $\then$ (\ref{prop B 2}) is obvious.

(\ref{prop B 2}) $\then$ (\ref{prop B 1}).
We shall prove that $v D_j \aaa = 0$ implies $D_j \aaa = 0.$

The equality $D_{d-1} \aaa=0$ follows from Lemma~\ref{L: long unused lemma}.
The second equality in Lemma~\ref{L: long unused lemma} also implies that $\aaa D_{d-2} \aaa=0.$
So, if it were that $D_{d-2} \aaa\neq 0$ then by Corollary 
\ref{C: criterion of regularity for V} we would have $v D_{d-2} \aaa\neq 0$ 
which contradicts the premise. Hence, $D_{d-2} \aaa=0.$

The equalities $D_{d-1} \aaa=0,$ $D_{d-2} \aaa=0$
and the third equality in Lemma~\ref{L: long unused lemma} imply that $\aaa D_{d-3} \aaa=0.$
So, if it were that $D_{d-3} \aaa\neq 0$ then by 
Corollary~\ref{C: criterion of regularity for V} we would have $v D_{d-3} \aaa\neq 0$ 
which contradicts the premise. Hence, $D_{d-3} \aaa=0.$ And so on.

(\ref{prop B 3}) $\then$ (\ref{prop B 2}) follows immediately 
from formula (\ref{F: alpha Dj v* = alpha bfA(j)(1+bfA hat A(0))}).

(\ref{prop B 2}) $\then$ (\ref{prop B 3}). Using 
(\ref{F: alpha Dj v* = alpha bfA(j)(1+bfA hat A(0))}) with $j=d-1$ we infer that 
$
  0 = \aaa D_{d-1} v^* = \aaa \bfA_\lambda ^{d-1}.
$
Hence, using (\ref{F: alpha Dj v* = alpha bfA(j)(1+bfA hat A(0))}) with $j=d-2$ we infer that 
$
  0 = \aaa D_{d-2} v^* = \aaa \bfA_\lambda ^{d-2}.
$
And so on: $0 = \aaa D_{1} v^* = \aaa \bfA_\lambda ^{1}.$

(\ref{prop B 3}) $\iff$ (\ref{prop B 4}). This equivalence follows from formula $(v^*+\aaa)\bfA_\lambda = 0$
which holds for any resonance point (outside essential spectrum), by Theorem~\ref{T: TFAE depth 1 criterion}.

The equivalences (\ref{prop B 1}) $\iff$ (\ref{prop B 5}),
(\ref{prop B 2}) $\iff$ (\ref{prop B 6}) obviously follow from 
the Laurent expansion (\ref{F: Laurent for euD}) of~$\euD_\lambda(s).$ 

(\ref{prop B 7}) $\then$ (\ref{prop B 3}). Since $\aaa = \aaa \hat P^\perp,$ we have 
$
  \aaa \im \brs{\bfA_\lambda} = \aaa \im\brs{\hat P^\perp \bfA_\lambda}.
$
By the premise, we have $\im\brs{\hat P^\perp \bfA_\lambda}\subset \im\brs{\bfA_\lambda}.$ 
Combining this with the obvious inclusion $\im\brs{\hat P^\perp \bfA_\lambda}\subset \clV_\lambda$
gives 
$$
  \aaa \im \brs{\bfA_\lambda} = 
  \aaa \im\brs{\hat P^\perp \bfA_\lambda} \subset \aaa \brs{\im\brs{\bfA_\lambda} \cap \clV_\lambda} = \set{0},
$$
where the last equality follows from the fact that by Theorem~\ref{T: TFAE depth 1 criterion}
eigenvectors~$\chi$ of depth at least one are~$V$-orthogonal to $\clV_\lambda,$
and therefore for such eigenvectors $\aaa \chi = 0.$

(\ref{prop B 3}) $\then$ (\ref{prop B 7}). 
Since $\aaa = \aaa \hat P^\perp,$ we have $\aaa \hat P^\perp \bfA_\lambda = 0.$
Since the kernel of the crossing operator $\aaa$ consists of vectors of depth at least one, we are done. 
\end{proof}

We say that a resonance point~$r_\lambda$ has \emph{property $B$}
if one of the equivalent conditions of Theorem~\ref{T: TFAE prop B} hold. 

\subsection{Resonance points with property~$A$}

Theorem~\ref{T: TFAE depth 1 criterion} and formula (\ref{F: -S A(j)=A(j-1)}) indicate 
that the operators~$\bfA_\lambda$ and $-S_\lambda$ 
restricted to the vector space $\Upsilon_\lambda(r_\lambda)$ behave 
to a certain extent as inverses of each other.  
In particular, while the operator~$\bfA_\lambda$ decreases order of a resonance vector~$\chi$ by~$1,$ the operator 
$S_\lambda$ increases order of~$\chi$ by~$1,$ provided 
there is some room for increasing the order. Another property of~$\bfA_\lambda$ is that it increases depth of a resonance
vector by~$1$ (this is, in fact, definition of the depth). It is therefore reasonable to ask whether the operator
$S_\lambda$ decreases depth of a resonance vector by~$1.$ We say that a resonance point has property $A,$
if it possesses this property. We conjecture that all resonance points have the property~$A.$
In this subsection we give two conditions which are equivalent to property~$A.$

\begin{thm} \label{T: TFAE prop A} 
For a resonance point~$r_\lambda=0$ which does not belong to the essential spectrum,
the following assertions are equivalent: 
\begin{enumerate}
  \item \label{prop A 7} For all $j=1,2,\ldots,d-1$ \
$
  \im\brs{D_j v^*} \subset \im{\bfA_\lambda^j}.
$  
  \item \label{prop A 8} For all $j=1,2,\ldots,d-1$ \
$
  \im\brs{S_\lambda \bfA_\lambda^{j}} \subset \im\brs{\bfA_\lambda^{j-1}}. 
$
  \item \label{prop A 9} For all $j=1,2,\ldots,d-1$ \
$
  \im\brs{\hat P \bfA_\lambda^{j-1}} \subset \im\brs{\bfA_\lambda^{j-1}}.
$
  \item \label{prop A 10} For all $j=1,2,\ldots,d-1$ \
$
  \im\brs{\hat P^\perp \bfA_\lambda^{j-1}} \subset \im\brs{\bfA_\lambda^{j-1}}.
$
\end{enumerate}
\end{thm}

\begin{proof} 
The equivalence (\ref{prop A 7}) $\iff$ (\ref{prop A 8}) follows from the equality 
(\ref{F: S A(j+1) + A(j) = Djv*+...}) applied for $j=d-1, d-2, \ldots$
The equivalence (\ref{prop A 8}) $\iff$ (\ref{prop A 9}) follows immediately from formula (\ref{F: -S A(j)=A(j-1)}).
The equivalence (\ref{prop A 9}) $\iff$ (\ref{prop A 10}) is obvious. 
\end{proof}

We say that a resonance point~$r_\lambda$ has \emph{property~$A$}
if there holds one of the equivalent conditions of Theorem~\ref{T: TFAE prop A}.
Property~$A$ implies property $B,$ since the condition (\ref{prop B 7}) of Theorem~\ref{T: TFAE prop B}
is a special case of the condition (\ref{prop A 10}) of Theorem~\ref{T: TFAE prop A}.
Property~$A$ holds trivially in two cases: if $m=1$ or if $d\leq 2.$ 
Also, if $d \leq 3,$ then property~$B$ trivially implies property~$A.$ 

\begin{conj} (i) Property $B$ implies property~$A.$
(ii) Property $B$ holds.
(iii) Property~$A$ holds. 
\end{conj}

\section{On eigenpaths $\phi_\nu(s)$ of $H_0+sV$}
\label{S: On phi[nu]}

\subsection{Order of an eigenpath}

Let~$H_{\rlmb}$ be a~$\lambda$-resonant operator. 
Eigenvectors of~$H_{\rlmb}$ corresponding to the eigenvalue~$\lambda$
form a vector space $\clV_\lambda,$ the dimension of which we denote by~$m.$
Given a regular direction~$V,$ among eigenvectors of~$H_{\rlmb}$ 
one can distinguish vectors
which are also eigenvectors of the crossing operator $\aaa = \hat P^\perp V \hat P^\perp.$

We denote by $\lambda_\nu(s),$ $\nu=1,\ldots,m,$ the eigenvalue functions of~$H_s,$
and by $\phi_\nu(s)$ the corresponding eigenvector functions of~$H_s.$ 
\begin{prop} 
If $\phi_\nu(s)$ is an eigenpath of $H_s$ then the vector $\phi_\nu(\rlmb)$
is also an eigenvector of the crossing operator $\aaa = \hat P^\perp V \hat P^\perp.$
Moreover, the corresponding eigenvalue is~$\lambda'_\nu(\rlmb).$ 
\end{prop}
\begin{proof}
Differentiating $H_s \phi_\nu(s) = \lambda_\nu(s) \phi_\nu(s)$ 
and letting $s = \rlmb$ we obtain the equality
$$
  V \phi_\nu(\rlmb) + H_{\rlmb} \phi'_\nu(\rlmb) 
        = \lambda'_\nu(\rlmb) \phi_\nu(\rlmb) + \lambda_\nu(\rlmb) \phi'_\nu(\rlmb).
$$
Applying to both sides of this equality the operator $\hat P^\perp$ gives
$
  \hat P^\perp V \phi_\nu(\rlmb) = \lambda'_\nu(\rlmb) \phi_\nu(\rlmb),
$
which is what is required.
\end{proof}

\begin{defn}
Let $\phi(s)$ be an analytic path of eigenvectors of~$H_s$ 
and let $k$ be a positive integer. 
We say that the eigenpath $\phi(s)$ has \emph{order at least $k$},\label{Page: property U(k)} if 
the vectors 
\begin{equation} \label{F: property U(k)}
  V\phi(\rlmb), \ V\phi'(\rlmb), \ \ldots, \ V\phi^{(k-2)}(\rlmb) 
\end{equation}
are orthogonal to the eigenspace~$\clV_\lambda.$ 
\end{defn}
We also say that a path $\phi(s)$ has \emph{order $k$},\label{Page: strict U(k)} if 
in addition $V\phi^{(k-1)}(\rlmb)$ is not orthogonal to $\clV_\lambda.$

Since for $k=1$ the set of vectors (\ref{F: property U(k)}) is empty,
order of every eigenpath is at least~$1.$ As we shall see later, 
the largest of orders of eigenpaths is equal to the order 
of the direction~$V.$ 

\begin{lemma} The definition of the order of an eigenpath is correct in the sense that it does 
not depend on normalisation of the eigenpath $\phi(s).$
That is, if $a(s)$ is an analytic function such that $a(\rlmb) \neq 0,$
then for the eigenpath $\psi(s) = a(s)\phi(s)$ the vectors 
$$
  V\psi(\rlmb), \ V\psi'(\rlmb), \ \ldots, \ V\psi^{(k-2)}(\rlmb) 
$$
are orthogonal to~$\clV_\lambda$ if and only if so are the vectors (\ref{F: property U(k)}).
\end{lemma}
\begin{proof} This immediately follows from the Leibniz rule 
\begin{equation} \label{F: psi(k)(s)=sum a phi}
  \psi^{(k)}(s) = \sum_{j=0}^k {k \choose j} a^{(j)}(s) \phi^{(k-j)}(s).
\end{equation}
\end{proof}

\smallskip

If $(H_0+sV) \phi(s) = \lambda(s)\phi(s),$ then 
$
  \lambda'(\rlmb) = \scal{\phi(\rlmb)}{V\phi(\rlmb)},
$
which is a well-known fact in perturbation theory (see e.g. \cite[\S 38]{LL3}).
In particular, if~$\lambda'(\rlmb) = 0,$ then the vector $V\phi(\rlmb)$ is orthogonal to the vector $\phi(\rlmb).$
The following lemma is a generalisation of this statement. 

\begin{lemma}  \label{L: property (Uk)}
Let $k\geq 2$ and let $\phi(s)$ be an analytic path of eigenvectors of the path 
$H_s = H_{\rlmb} + s V.$ The following assertions are equivalent:
\begin{enumerate}
  \item[(i)] the path $\phi(s)$ has order at least~$k,$
  \item[(ii)] the vectors 
$
  V\phi(\rlmb), \ V\phi'(\rlmb), \ \ldots, \ V\phi^{(k-2)}(\rlmb) 
$
are orthogonal to the vector $\phi(\rlmb),$ 
  \item[(iii)] the equalities 
$
  \lambda'(\rlmb) = 0, \ \ldots, \ \lambda^{(k-1)}(\rlmb) = 0
$
hold, where~$\lambda(s)$ is an analytic path of eigenvalues of $H_s$ which corresponds to $\phi(s),$
  \item[(iv)] for all $j=1,2,\ldots,k-1$ \ 
  $
    (H_{\rlmb}-\lambda)\phi^{(j)}(\rlmb) = - j V \phi^{(j-1)}(\rlmb).
  $
\end{enumerate}
\end{lemma}
\begin{proof} Since $\phi(\rlmb)$ is an eigenvector, (i) plainly implies (ii). 
The implication (iv) $\then$ (i) is also obvious. 
We shall prove that (ii) implies (iii) and that (iii) implies (iv).

\smallskip
(ii) $\then$ (iii). Differentiating $k-1$ times the eigenvalue equation $H_s\phi(s) = \lambda(s)\phi(s)$ 
gives the equality 
\begin{equation} \label{F: Differentiated eigenvalue equation(0)}
  (k-1) V \phi^{(k-2)}(s) + H_s \phi^{(k-1)}(s) = \sum_{j=0}^{k-1} {k-1 \choose j} \lambda^{(j)}(s) \phi^{(k-1-j)}(s).
\end{equation}
Here we let $s = \rlmb$ and take the scalar product of both sides with the vector 
$\phi(\rlmb).$ This leads to cancellation of the second summand of the left hand side with the first summand of the 
right hand side. Hence, we obtain the equality
\begin{equation} \label{F: Differentiated eigenvalue equation}
  (k-1)\scal{\phi(\rlmb)}{V \phi^{(k-2)}(\rlmb)} = \sum_{j=1}^{k-1} {k-1 \choose j} \lambda^{(j)}(\rlmb) \scal{\phi(\rlmb)}{\phi^{(k-1-j)}(\rlmb)}.
\end{equation}
If $k=2$ then 
$
  \scal{\phi(\rlmb)}{V \phi(\rlmb)} = \lambda'(\rlmb) \scal{\phi(\rlmb)}{\phi(\rlmb)}.
$
This equality implies the assertion for $k=2.$ 
Assume that the claim holds for $k < n.$ Then from 
(\ref{F: Differentiated eigenvalue equation}) with $k = n,$ using the induction assumption, we get
$$
  (n-1)\scal{\phi(\rlmb)}{V \phi^{(n-2)}(\rlmb)} = \lambda^{(n-1)}(\rlmb) \scal{\phi(\rlmb)}{\phi(\rlmb)}.
$$
Since by the premise $\scal{\phi(\rlmb)}{V \phi^{(n-2)}(\rlmb)}=0,$ this gives~$\lambda^{(n-1)}(\rlmb) = 0.$ 

\smallskip
(iii) $\then$ (iv). Letting $s=\rlmb$ in (\ref{F: Differentiated eigenvalue equation(0)}) gives the equality 
$$
  (k-1) V \phi^{(k-2)}(\rlmb) + H_{\rlmb} \phi^{(k-1)}(\rlmb) 
         = \sum_{j=0}^{k-1} {k-1 \choose j} \lambda^{(j)}(\rlmb) \phi^{(k-1-j)}(\rlmb).
$$
By the premise, the right hand side simplifies to~$\lambda(\rlmb) \phi^{(k-1)}(\rlmb) = \lambda \phi^{(k-1)}(\rlmb).$
Hence,
$$
  (H_{\rlmb}-\lambda) \phi^{(k-1)}(\rlmb) = - (k-1) V \phi^{(k-2)}(\rlmb).
$$
\end{proof}

This proof also shows that if an eigenpath $\phi(s)$ has order $k,$ then
\begin{equation} \label{F: (phi,V phi(k-1))=lambda(k)(0)(phi,phi)}
  \scal{\phi(\rlmb)}{V\phi^{(k-1)}(\rlmb)} = \frac 1{k} \lambda^{(k)}(\rlmb) \scal{\phi(\rlmb)}{\phi(\rlmb)}.
\end{equation}
In particular, in this case the number $\scal{\phi(\rlmb)}{V\phi^{(k-1)}(\rlmb)}$
is non-zero and real. 

\begin{lemma} \label{L: nice one} 
Let $k\geq 2$ and let $\phi(s)$ be an analytic path of eigenvectors of the path $H_s = H_0+sV.$ 
If $\phi(s)$ has order at least~$k,$ then 
\begin{enumerate}
  \item $\phi^{(j)}(\rlmb)$ is a resonance vector of order $j+1$ for all $j=0,1,2,\ldots,k-1,$ and 
  \item for all $j=1,2,\ldots,k-1$ 
  \begin{equation} \label{F: bfA phi(j)=phi(j-1)}
    \bfA_\lambda(\rlmb) \phi^{(j)}(\rlmb) = j \phi^{(j-1)}(\rlmb).
  \end{equation}  
In particular, $\phi(\rlmb)$ is an eigenvector of depth at least~$k-1.$
\end{enumerate}
\end{lemma}
\begin{proof}
We prove this using induction on $k.$

(A) Let $k=2.$
Differentiation of the eigenvalue equation gives
$$
  V \phi(s) + H_s \phi'(s) = \lambda'(s) \phi(s) + \lambda(s) \phi'(s).
$$
Since $k = 2,$ by the premise and item (iii) of Lemma~\ref{L: property (Uk)}, 
the first summand in the right hand side of the last equality vanishes at $s = \rlmb.$ Hence,
$
  V \phi(\rlmb) + H_{\rlmb} \phi'(\rlmb) = \lambda(\rlmb) \phi'(\rlmb),
$
which can be rewritten as 
$
  (H_{\rlmb} - \lambda) \phi'(\rlmb) = - V \phi(\rlmb).
$
This implies 
$$
  \phi'(\rlmb) = - S_\lambda \phi(\rlmb) + \text{order 1 vector}.
$$
Combining this with the premise $V\phi(\rlmb) \perp \clV_\lambda$
and Theorem~\ref{T: TFAE depth 1 criterion} implies 
that the vector $\phi'(\rlmb)$ is a resonance vector of order two. 
Theorem~\ref{T: TFAE depth 1 criterion} also implies the equality (\ref{F: bfA phi(j)=phi(j-1)}) for $j=2.$

(B) Assume that the claim holds for values of~$k$ less than~$n$
and let~$\phi(s)$ be an eigenpath of order~$\geq n.$ 
By item (iv) of Lemma~\ref{L: property (Uk)},
we have 
$$
  (H_{\rlmb}-\lambda)\phi^{(n-1)}(\rlmb) = - (n-1) V \phi^{(n-2)}(\rlmb).
$$
Applying the sliced resolvent $R_\lambda(\hat H_{\rlmb})$ to both sides of this equality gives 
$$
  \frac 1 {n-1} \phi^{(n-1)}(\rlmb) = - S_\lambda \phi^{(n-2)}(\rlmb) + \text{order 1 vector}. 
$$
By the induction assumption the order of the vector $\phi^{(n-2)}$ is $n-1$ 
and, since $\phi(s)$ has order $\geq n,$ the relation $V\phi^{(n-2)} \perp \clV_\lambda$
holds. Hence, Theorem~\ref{T: TFAE depth 1 criterion} 
implies that $\phi^{(n-1)}$ is a vector of order $n$ and that (\ref{F: bfA phi(j)=phi(j-1)}) holds. 
\end{proof}

\begin{lemma} \label{L: third one}
Let $k\geq 2.$ If $\phi(s)$ is an analytic path of eigenvectors of the path $H_s = H_0+sV$ 
such that $\phi(\rlmb)$ has depth at least $k-1,$ then the eigenpath $\phi(s)$ has order at least~$k.$
\end{lemma}
\begin{proof}
If $k=2$ then the assertion follows from the equivalence for $\phi(\rlmb)$ to have depth at least one
and to be~$V$-orthogonal to~$\clV_\lambda,$ see Theorem~\ref{T: TFAE depth 1 criterion}.
Assume that the claim holds for $k<n$ and let $\phi(\rlmb)$ be of depth at least~$n-1.$
For the eigenvalue function~$\lambda(s)$ which corresponds to $\phi(s),$ 
by the induction assumption and item (iii) of Lemma~\ref{L: property (Uk)} we have 
$
  \lambda'(\rlmb) = \ldots = \lambda^{(n-2)}(\rlmb) = 0.
$
Combining this with (\ref{F: Differentiated eigenvalue equation}), we obtain
$$
  (n-1) \scal{\phi(\rlmb)}{V \phi^{(n-2)}(\rlmb)} = \lambda^{(n-1)}(\rlmb) \scal{\phi(\rlmb)}{\phi(\rlmb)}.
$$
Since $\phi(\rlmb)$ has depth at least $n-1,$ there exists a vector~$f$ such that~$\bfA_\lambda^{n-1} f = \phi(\rlmb).$
Hence, 
\begin{equation*}
 \begin{split}
  \lambda^{(n-1)}(\rlmb) \scal{\phi(\rlmb)}{\phi(\rlmb)} & = (n-1) \scal{\bfA_\lambda^{n-1} f}{V \phi^{(n-2)}(\rlmb)} 
  \\ & = (n-1) \scal{f}{V \bfA_\lambda^{n-1} \phi^{(n-2)}(\rlmb)} 
  \\ & = 0,
 \end{split}
\end{equation*}
where the last equality follows from the induction assumption, according to which 
and Lemma~\ref{L: nice one} the vector $\phi^{(n-2)}(\rlmb)$
has order~$n-1$ and therefore~$\bfA_\lambda^{n-1} \phi^{(n-2)}(\rlmb) = 0.$
This gives~$\lambda^{(n-1)}(\rlmb) = 0.$ Hence, by Lemma~\ref{L: property (Uk)} 
the proof is complete. 
\end{proof}

Since the eigenvalue~$\lambda$ has geometric multiplicity~$m,$ there are~$m$ eigenpaths 
$\phi_\nu(s),$ $\nu=1,2,\ldots,m,$
and their orders we denote by $\tilde d_\nu.$\label{Page: tilde d(nu)}

We summarise Lemmas~\ref{L: property (Uk)},~\ref{L: nice one} 
and~\ref{L: third one} in the following theorem.

\begin{thm} \label{T: TFAE for (Uk)}
For each $\nu=1,\ldots,m$ the following assertions are equivalent:
\begin{enumerate}
  \item[(i)] The eigenpath $\phi_\nu(s)$ has order $\tilde d_\nu.$
  \item[(ii)] The vectors 
$
  V\phi_\nu(\rlmb), \ V\phi_\nu'(\rlmb), \ \ldots, \ V\phi_\nu^{(\tilde d_\nu-2)}(\rlmb) 
$
are orthogonal to the vector $\phi_\nu(\rlmb),$ and the vector $V\phi_\nu^{(\tilde d_\nu-1)}(\rlmb)$
is not. 
  \item[(iii)] The equalities 
$
  \lambda_\nu'(\rlmb) = 0, \ \ldots, \ \lambda_\nu^{(\tilde d_\nu-1)}(\rlmb) = 0,
$
and the inequality $\lambda_\nu^{(\tilde d_\nu)}(\rlmb) \neq 0$ hold, 
where~$\lambda_\nu(s)$ is an analytic path of eigenvalues of $H_s$ which corresponds to $\phi_\nu(s).$
  \item[(iv)] For all $j=1,2,\ldots,\tilde d_\nu-1$ the equalities 
  $
    (H_{\rlmb}-\lambda)\phi_\nu^{(j)}(\rlmb) = - j V \phi_\nu^{(j-1)}(\rlmb)
  $
  hold, but it fails for $j=\tilde d_\nu.$
  \item[(v)] For all $j=1,2,\ldots,\tilde d_\nu-1$ 
  the equalities $\bfA_\lambda \phi_\nu^{(j)}(\rlmb) = j \phi_\nu^{(j-1)}(\rlmb)$
  hold, but it fails for $j=\tilde d_\nu.$
  \item[(vi)] $\phi_\nu(\rlmb)$ is an eigenvector of depth~$\tilde d_\nu-1.$
\end{enumerate}
\end{thm}

We shall refer to the positive integer $\tilde d_\nu$ as the order of the eigenvalue function $\lambda_\nu(s)$ too.
Thus, the order of $\lambda_\nu(s)$ is the smallest of positive integers $\tilde d_\nu,$ such that 
$\lambda^{(\tilde d_\nu)}_\nu(s) \neq 0,$ or 
\begin{equation} \label{F: lambda(s)=lambda +eps(s-l)...}
  \lambda_\nu(s) = \lambda + \eps_\nu(s-r_\lambda)^{\tilde d_\nu} + O((s-r_\lambda)^{\tilde d_\nu+1}), \ s \to \rlmb,
\end{equation}
where $\eps_\nu \neq 0.$

\begin{rems} \rm The formula (\ref{F: bfA phi(j)=phi(j-1)}) is independent from 
a choice of normalisation of the eigenpath~$\phi(s).$ That is, if $\psi(s) = a(s)\phi(s)$
is another eigenpath, where $a(s)$ is an analytic function with $a(\rlmb) \neq 0,$ then 
$\bfA_\lambda \psi^{(j)}(\rlmb) = j \psi^{(j-1)}(\rlmb).$
This equality also follows from (\ref{F: bfA phi(j)=phi(j-1)}) and (\ref{F: psi(k)(s)=sum a phi}).
\end{rems}

\begin{rems} \rm In this paper we use notation $\phi^{(k-1)}$ to denote a vector of order~$k.$
This is consistent with the usage of the bracketed superscript index $(k)$ for the $k$-th order 
derivative. In \cite{Az9} the notation $\phi^{(k)}$ was used to denote a vector of order~$k.$
\end{rems}

\smallskip

Assume that~$H_{\rlmb}$ is a~$\lambda$-resonance point of multiplicity~$m.$
Let~$V$ be a regular direction and let $H_s = H_{\rlmb} + (s-\rlmb)V.$
Let~$\lambda_\nu(s), \ \nu=1,\ldots,m,$ be eigenvalue functions of $H_s,$ 
which are listed counting multiplicities, and 
let $\phi_\nu(s), \ 1,\ldots,m,$ be corresponding eigenvector functions.
If the eigenvalue~$\lambda$ of~$H_{\rlmb}$ is splitting, that is, if the functions 
$\lambda_1(s), \ldots, \lambda_m(s)$ are distinct, then for $\mu \neq \nu$ the vectors  
$\phi_\nu(s)$ and $\phi_\mu(s)$ are orthogonal as eigenvectors of a self-adjoint operator $H_s$
corresponding to different eigenvalues~$\lambda_\nu(s)$ and~$\lambda_\mu(s).$
Even if the eigenvalue~$\lambda$ of~$H_{\rlmb}$ is not splitting, it is always possible to choose 
the eigenvector functions $\phi_\nu(s)$ to be pairwise orthogonal. 
Thus, we can and do assume that for any~$s$
\begin{equation} \label{F: (phi(i),phi(j)) = 0}
  \scal{\phi_\nu(s)}{\phi_\mu(s)} = 0.
\end{equation}
In particular,
$
  \scal{\phi_\nu(\rlmb)}{\phi_\mu(\rlmb)} = 0.
$
Hence, a regular direction~$V$ induces a natural orthogonal decomposition
of the eigenspace. This orthogonal decomposition of the eigenspace has an additional
property given by the following lemma.

\begin{lemma} \label{L: (V phi(i),phi(j)) = 0}
In the setting given above, if $\mu \neq \nu,$ then
$
  \scal{V\phi_\mu(\rlmb)}{\phi_\nu(\rlmb)} = 0.
$
\end{lemma}
\begin{proof}
Taking derivative of (\ref{F: (phi(i),phi(j)) = 0}) we get
\begin{equation} \label{F: (phi'i(s),phi j(s)) + ... =0}
  \scal{\phi'_\nu(s)}{\phi_\mu(s)} + \scal{\phi_\nu(s)}{\phi'_\mu(s)}= 0.
\end{equation}
Further, for any~$s$ from some neighbourhood of $\rlmb$ we have 
\begin{equation*}
     0 = \scal{\lambda_\nu(s) \phi_\nu(s)}{\phi_\mu(s)} = \scal{H_s \phi_\nu(s)}{\phi_\mu(s)}.
\end{equation*}
Hence,
\begin{equation*}
  \begin{split}
     0 & = \frac d{ds} \scal{H_s \phi_\nu(s)}{\phi_\mu(s)}
     \\ & = \scal{H'_s \phi_\nu(s)}{\phi_\mu(s)} + \scal{H_s \phi'_\nu(s)}{\phi_\mu(s)} + \scal{H_s \phi_\nu(s)}{\phi'_\mu(s)}.
  \end{split}
\end{equation*}
With $s = \rlmb$ this gives 
\begin{equation*}
  \begin{split}
     0 & = \scal{V \phi_\nu(\rlmb)}{\phi_\mu(\rlmb)} + \scal{H_{\rlmb} \phi'_\nu(\rlmb)}{\phi_\mu(\rlmb)} + \scal{H_{\rlmb} \phi_\nu(\rlmb)}{\phi'_\mu(\rlmb)}
     \\ & = \scal{V \phi_\nu(\rlmb)}{\phi_\mu(\rlmb)} + \lambda \scal{ \phi'_\nu(\rlmb)}{\phi_\mu(\rlmb)} + \lambda \scal{\phi_\nu(\rlmb)}{\phi'_\mu(\rlmb)}
     \\ & = \scal{V \phi_\nu(\rlmb)}{\phi_\mu(\rlmb)},
  \end{split}
\end{equation*}
where the last equality follows from (\ref{F: (phi'i(s),phi j(s)) + ... =0}).
\end{proof}

\begin{lemma} \label{L: previous lemma}
Let $\phi_\mu(s)$ and $\phi_\nu(s)$ be two (orthogonal) eigenpaths of $H_0+sV$ with corresponding 
eigenvalue functions~$\lambda_\mu(s)$ and~$\lambda_\nu(s).$ If the eigenpath $\phi_\mu(s)$
has order at least $k,$ then for all $j=0,1,\ldots,k-1,$ \ 
$
  \scal{V \phi^{(j)}_\mu(\rlmb)}{\phi_\nu(\rlmb)} = 0.
$
\end{lemma}
\begin{proof} 
It is sufficient to prove that $\scal{V \phi^{(k-1)}_\mu(\rlmb)}{\phi_\nu(\rlmb)} = 0.$ 
Since by (\ref{F: (phi(i),phi(j)) = 0}) the function 
$
  s \mapsto \scal{H_s \phi_\mu(s)}{\phi_\nu(s)}
$
is zero, 
for any non-negative integer $k$ the Leibniz rule implies 
\begin{equation} \label{F: 0=D^s(H(s) phi,phi)=...}
  \begin{split}
     0 & = \frac{d^{k}}{d s^{k}} \scal{H_s \phi_\mu}{\phi_\nu} \\
       & = \scal{H_s \phi^{(k)}_\mu}{\phi_\nu} 
            + \sum_{j=0}^{k-1} {k \choose j} \scal{H_s \phi^{(j)}_\mu}{\phi^{(k-j)}_\nu}
            + \sum_{j=0}^{k-1} k{k-1 \choose j} \scal{V \phi^{(j)}_\mu}{\phi^{(k-1-j)}_\nu}.
  \end{split}
\end{equation}
Here we have explicitly separated the first summand $\scal{H_s \phi^{(k)}_\mu}{\phi_\nu}$
since it will be treated differently. After this special treatment, this summand will be returned 
into the sum. 

In this equality we substitute~$s$ by $\rlmb,$ but nevertheless, we write $\phi_\mu$ instead of
$\phi_\mu(\rlmb),$ etc. 
We have $H_{\rlmb} \phi_\mu = \lambda \phi_\mu.$ 
By Lemma~\ref{L: property (Uk)},
the premise asserts that the path $\phi_\mu(s)$ has order at least~$k.$
Hence, the item (iv) of Lemma~\ref{L: property (Uk)}
implies that for all $j = 0,1,\ldots,k-1$ 
$$
  H_{\rlmb} \phi^{(j)}_\mu = \lambda \phi^{(j)}_\mu - j V \phi^{(j-1)}_\mu.
$$
Using these equalities, we get from (\ref{F: 0=D^s(H(s) phi,phi)=...}) with $s = \rlmb$ 
\begin{equation*}
  \begin{split}
     0  & = \scal{H_{\rlmb} \phi^{(k)}_\mu}{\phi_\nu} 
               + \sum_{j=0}^{k-1} {k \choose j} \scal{H_{\rlmb} \phi^{(j)}_\mu}{\phi^{(k-j)}_\nu}
               + \sum_{j=0}^{k-1} k{k-1 \choose j} \scal{V \phi^{(j)}_\mu}{\phi^{(k-1-j)}_\nu} \\
        & = \lambda \scal{\phi^{(k)}_\mu}{\phi_\nu} 
              + \lambda \sum_{j=0}^{k-1} {k \choose j} \scal{\phi^{(j)}_\mu}{\phi^{(k-j)}_\nu} \\
     \\ & \hskip 2.6 cm - \sum_{j=0}^{k-1} j {k \choose j} \scal{V\phi^{(j-1)}_\mu}{\phi^{(k-j)}_\nu}
                      + \sum_{j=0}^{k-1} k{k-1 \choose j} \scal{V \phi^{(j)}_\mu}{\phi^{(k-1-j)}_\nu} \\ 
        & = \lambda \sum_{j=0}^{k} {k \choose j} \scal{\phi^{(j)}_\mu}{\phi^{(k-j)}_\nu} 
              - \sum_{j=1}^{k-1} \frac{k!}{(j-1)!(k-j)!} \scal{V\phi^{(j-1)}_\mu}{\phi^{(k-j)}_\nu}
     \\ & \hskip 4.5 cm + \sum_{j=0}^{k-1} \frac{k!}{j!(k-1-j)!} \scal{V \phi^{(j)}_\mu}{\phi^{(k-1-j)}_\nu}. 
  \end{split}
\end{equation*}
The first summand of the last expression is equal to 
$
  \lambda \frac{d^{k}}{d s^{k}} \scal{\phi_\mu}{\phi_\nu}\big|_{s=\rlmb},
$
which by (\ref{F: (phi(i),phi(j)) = 0}) is zero. 
Hence, 
\begin{equation*}
  \begin{split}
        0 & = - \sum_{j=0}^{k-2} \frac{k!}{j!(k-1-j)!} \scal{V\phi^{(j)}_\mu}{\phi^{(k-1-j)}_\nu}
             + \sum_{j=0}^{k-1} \frac{k!}{j!(k-1-j)!} \scal{V \phi^{(j)}_\mu}{\phi^{(k-1-j)}_\nu} \\
          & = k \scal{V \phi^{(k-1)}_\mu}{\phi_\nu}. 
  \end{split}
\end{equation*}
\end{proof}

\begin{thm} \label{T: phi(j,mu)(0) is self-dual} 
If $\phi_\mu(s)$ and $\phi_\nu(s)$ are two eigenpaths with 
orders $\tilde d_\mu$ and  $\tilde d_\nu$ respectively,
then for any $j=0,1,2,\ldots,\tilde d_\mu-1$ and $k=0,1,2,\ldots,\tilde d_\nu-1$
$$
  \scal{V \phi^{(j)}_\mu(\rlmb)}{\phi^{(k)}_\nu(\rlmb)} = 0.
$$
\end{thm}
\begin{proof}
As in Lemma~\ref{L: previous lemma}, it is sufficient 
to prove $\scal{V\phi^{(\tilde d_\mu-1)}_\mu}{\phi^{(\tilde d_\nu-1)}_\nu} = 0.$

Let $l = \tilde d_\mu+\tilde d_\nu-2.$ 
Since eigenpaths $\phi_\mu(s)$ and $\phi_\nu(s)$ are distinct we have 
$$\scal{H_s \phi_\mu(s)}{\phi_\nu(s)} = \lambda_\mu(s)\scal{\phi_\mu(s)}{\phi_\nu(s)} = 0,$$
where~$\lambda_\mu(s)$ is the eigenvalue function for $\phi_\mu(s).$ 
Leibniz rule gives the equality 
\begin{equation} \label{F: blwjb equality}
  \begin{split}
     0 & = \frac{d^{l+1}}{d s^{l+1}} \scal{H_s \phi_\mu}{\phi_\nu} \\
        & = \sum_{j=0}^{l+1} {l+1 \choose j} \scal{H_s \phi^{(j)}_\mu}{\phi^{(l+1-j)}_\nu}
         + \sum_{j=0}^{l} (l+1){l \choose j} \scal{V \phi^{(j)}_\mu}{\phi^{(l-j)}_\nu}.
  \end{split}
\end{equation}
We transform the first summand as follows:
\begin{equation*}
  \begin{split}
    \sum_{j=0}^{l+1} {l+1 \choose j} & \scal{H_s \phi^{(j)}_\mu}{\phi^{(l+1-j)}_\nu}  \\
         & = \sum_{j=\tilde d_\mu}^{l+1} {l+1 \choose j} \scal{\phi^{(j)}_\mu}{H_s \phi^{(l+1-j)}_\nu} 
             + \sum_{j=0}^{\tilde d_\mu-1} {l+1 \choose j} \scal{H_s \phi^{(j)}_\mu}{\phi^{(l+1-j)}_\nu}.
   \end{split}
\end{equation*}
Here we let $s = \rlmb$ and apply the equalities 
$$
  H_{\rlmb} \phi^{(l+1-j)}_\nu = \lambda \phi^{(l+1-j)}_\nu - (l+1-j) V\phi^{(l-j)}_\nu, \quad 
  H_{\rlmb} \phi^{(j)}_\mu = \lambda \phi^{(j)}_\mu - j V \phi^{(j-1)}_\mu,
$$
which follow from the premise that the eigenpaths $\phi_\nu(s)$ and $\phi_\mu(s)$
have orders $\tilde d_\nu$ and $\tilde d_\mu$ respectively, according to Theorem~\ref{T: TFAE for (Uk)}.
This gives 
\begin{equation*}
  \begin{split}
    (E) & := \sum_{j=0}^{l+1} {l+1 \choose j} \scal{H_{\rlmb} \phi^{(j)}_\mu}{\phi^{(l+1-j)}_\nu}  \\
       & = \sum_{j=\tilde d_\mu}^{l+1} {l+1 \choose j} \lambda \scal{\phi^{(j)}_\mu}{\phi^{(l+1-j)}_\nu} 
             + \sum_{j=0}^{\tilde d_\mu-1} {l+1 \choose j} \lambda \scal{\phi^{(j)}_\mu}{\phi^{(l+1-j)}_\nu} \\
       & \qquad 
        - \sum_{j=\tilde d_\mu}^{l+1} (l+1-j){l+1 \choose j} \scal{\phi^{(j)}_\mu}{V\phi^{(l-j)}_\nu} 
        - \sum_{j=0}^{\tilde d_\mu-1} j {l+1 \choose j} \scal{V\phi^{(j-1)}_\mu}{\phi^{(l+1-j)}_\nu}.
   \end{split}
\end{equation*}
This expression we can rewrite as follows:
\begin{equation*}
  \begin{split}
    (E) & = \lambda \ \sum_{j=0}^{l+1} {l+1 \choose j} \scal{\phi^{(j)}_\mu}{\phi^{(l+1-j)}_\nu} \\
       & \qquad - \sum_{j=\tilde d_\mu}^{l+1} \frac{(l+1)!}{j!(l-j)!} \scal{\phi^{(j)}_\mu}{V\phi^{(l-j)}_\nu} 
        - \sum_{j=1}^{\tilde d_\mu-1} \frac{(l+1)!}{(j-1)!(l+1-j)!} \scal{V\phi^{(j-1)}_\mu}{\phi^{(l+1-j)}_\nu} \\
       & = 0 - \sum_{j=\tilde d_\mu}^{l+1} \frac{(l+1)!}{j!(l-j)!} \scal{\phi^{(j)}_\mu}{V\phi^{(l-j)}_\nu} 
        - \sum_{j=0}^{\tilde d_\mu-2} \frac{(l+1)!}{j!(l-j)!} \scal{V\phi^{(j)}_\mu}{\phi^{(l-j)}_\nu}.
   \end{split}
\end{equation*}
Here the first summand is zero for the same reason as in Lemma~\ref{L: previous lemma}. 
Combining this equality with (\ref{F: blwjb equality}) gives 
$
  \scal{V\phi^{(\tilde d_\mu-1)}_\mu}{\phi^{(\tilde d_\nu-1)}_\nu} = 0.
$
\end{proof}

In subsection~\ref{SS: relation of Upsilon} we give another proof of Theorem~\ref{T: phi(j,mu)(0) is self-dual}.

\subsection{On ground eigenvalue}
If $H_{\rlmb}$ is bounded below then its smallest eigenvalue, if it exists, is called the \emph{ground eigenvalue}.

\begin{prop} \label{P: fcuking proposition} Let~$\lambda$ be a simple eigenvalue of $H_{\rlmb}.$
The order of the triple $(\lambda; H_{\rlmb},V)$ is equal to~$d$ if and only if 
$$
  \aaa = 0, \ \scal{R_\lambda(\hat H_{\rlmb})v}{v} = 0, 
  \ \ldots, \ \scal{\brs{R_\lambda(\hat H_{\rlmb})V}^{d-3}R_\lambda(\hat H_{\rlmb}) v}{v} = 0
$$
and 
$$
  \scal{\brs{R_\lambda(\hat H_{\rlmb})V}^{d-2}R_\lambda(\hat H_{\rlmb}) v}{v} \neq 0.
$$
\end{prop}
\begin{proof} Since $\clV_\lambda$ is one-dimensional, the matrix~$v$ is equal to a vector~$V\chi,$
where~$\chi$ is an eigenvector. Further, since~$\lambda$ is a simple eigenvalue, 
by Theorem~\ref{T: TFAE depth 1 criterion} the order is equal to~$d$ if and only if
$$
  \chi \perp V\chi, \ S_\lambda \chi \perp V\chi, \ \ldots, S_\lambda^{d-2} \chi \perp V\chi
$$
and 
$
  S_\lambda^{d-1} \chi 
$
is not orthogonal to~$V\chi.$
It is left to note that 
$$
  \scal{\brs{R_\lambda(\hat H_{\rlmb})V}^{j}R_\lambda(\hat H_{\rlmb}) v}{v} = \scal{S^{j+1}_\lambda \chi}{V\chi}.
$$
\end{proof}

\begin{thm} If~$\lambda$ is a non-degenerate ground eigenvalue, then order 
of the corresponding resonance point is at most two.
\end{thm}
\begin{proof} Since~$\lambda$ is the smallest eigenvalue, the sliced resolvent 
$R_\lambda(\hat H_{\rlmb})$ is a strictly positive operator. 
Hence, if the order $d$ is greater than two, then Proposition~\ref{P: fcuking proposition} gives 
$\scal{R_\lambda(\hat H_{\rlmb})v}{v} = 0$ 
and combining this with $R_\lambda(\hat H_{\rlmb})>0$ gives $R_\lambda(\hat H_{\rlmb})v = 0,$ which is a contradiction. 
\end{proof}

If the order of a resonance point is even then the corresponding eigenvalue makes a U-turn.
One may ask whether this is a left or a right U-turn. At the ground eigenvalue,
the following theorem answers this question. 

\begin{thm} If~$\lambda$ is a non-degenerate ground eigenvalue
and the order of a direction~$V$ is equal to~$2,$ then the corresponding U-turn is a left U-turn.
\end{thm}
\begin{proof} Since $d = 2,$ using Theorem~\ref{T: TFAE for (Uk)} one infers that 
$$
  0 \neq \scal{\phi(\rlmb)}{V\phi'(\rlmb)} = - \scal{\phi(\rlmb)}{V S_\lambda \phi(\rlmb)} = - \scal{\chi}{V R_\lambda(\hat H_{\rlmb})V\chi}.
$$
Since~$\lambda$ is a ground eigenvalue, the operator $R_\lambda(\hat H_{\rlmb})$ is positive.
Combining this with the previous equality and inequality gives 
$
  \scal{\phi(\rlmb)}{V\phi'(\rlmb)} < 0.
$
It follows from this and formula (\ref{F: (phi,V phi(k-1))=lambda(k)(0)(phi,phi)}) 
that~$\lambda''(\rlmb) < 0.$ Hence, the U-turn is the left one.
\end{proof}

\section{Characterisation of order~$k$ directions}
\label{S: char-n order k dir-ns}
\label{Page: H(s) 3}
Recall that a resonant path $H(s)$ is an analytic path of self-adjoint operators 
in an affine space~$\clA$ such that~$\lambda$
is an eigenvalue of all $H(s).$ For any resonant path $H(s)$ there exists (see \cite{Kato}) an analytic path
of eigenvectors~$\chi(s),$ that is,
\begin{equation} \label{F: H(s)chi(s)=lambda chi(s)}
  H(s) \chi(s) = \lambda \chi(s).
\end{equation}
Such an analytic path of eigenvectors~$\chi(s)$ we shall call a \emph{resonant eigenpath}.
We say that a resonant path $H(s)$ is \emph{simple},
if at least one point on that path is simple. If a resonant path is simple,
then all its points except a discrete set are simple. Further,
a simple resonant path has a unique resonant eigenpath up to a multiplicative constant, see e.g. \cite{Kato}.

\subsection{Tangent directions have high orders}

We say that a direction~$V$ is \emph{tangent} \label{Page: tangent to order k direction}
to the resonance set~$\Rset$ at~$H_{\rlmb}$ \emph{to order at least~$k,$}
if there exists a resonant path $H(s) \subset \euR(\lambda)$
such that for some (necessarily real) numbers $c_2,c_3,\ldots,c_{k-1}$
\begin{equation} \label{F: H(s)=H0+sV+c2 s2 V+...}
  H(s) = H_{\rlmb} + (s-\rlmb)V + \sum_{j=2}^{k-1} c_j (s-\rlmb)^j V + O((s-\rlmb)^k), \ \ s \to \rlmb. 
\end{equation}
In this case we also say that the path $H(s)$ is tangent to~$V$ at $H(\rlmb)=H_{\rlmb}$ to order at least~$k.$
If~$k$ is the greatest positive integer with this property, then we say that~$V$ is tangent at~$H_{\rlmb}$ to order~$k.$

We say that a direction~$V$ is \emph{tangent} \label{Page: tangent direction}
at~$H_{\rlmb}$ if~$V$ is tangent to order at least~$2.$
If a direction~$V$ is tangent only to order 1 at $H_{\rlmb} \in \Rset,$ then we say that~$V$
is \emph{transversal} at~$H_{\rlmb}.$ \label{Page: transversal direction}

\begin{lemma} \label{L: tangent V}
Assume Assumption~\ref{A: Main Assumption}.
Let~$H_{\rlmb}$ be a resonant point and let~$V$ be a regular direction.
If the direction~$V$ is tangent to the resonance set $\euR(\lambda)$ at~$H_{\rlmb}$
then order of~$V$ is at least~$2.$
\end{lemma}
\begin{proof} As usual, we let $H_s = H_0 + sV.$ Since~$V$ is tangent to $\euR(\lambda)$ at $H_{\rlmb},$
there exists a resonant path $\set{H(s)}$ such that $H(\rlmb) = H_{\rlmb}$ and
$
  H'(\rlmb) = V.
$
Let~$\chi(s)$ be a corresponding resonant eigenpath.
Differentiating the eigenvalue equation (\ref{F: H(s)chi(s)=lambda chi(s)}) we get
$$
  H'(s) \chi(s) + H(s) \chi'(s) = \lambda \chi'(s).
$$
Letting here $s = \rlmb$ gives 
$
  V \chi(\rlmb) + H_{\rlmb} \chi'(\rlmb) = \lambda \chi'(\rlmb).
$
This equality can be rewritten as follows: 
$$
  (1 + (\rlmb-s) R_\lambda(H_s)V) \chi'(\rlmb) = - R_\lambda(H_s)V \chi(\rlmb),
$$
where~$s$ is any real number such that the operator $H_s-\lambda$ has bounded inverse
(such real numbers~$s$ exist since by the premise~$V$ is a regular direction).
The eigenvalue equation $H_{\rlmb} \chi(\rlmb) = \lambda \chi(\rlmb)$
is equivalent to
$
  (1 + (\rlmb-s) R_\lambda(H_s)V) \chi(\rlmb) = 0.
$
Combining this equality with previous one gives 
$$
  (1 + (\rlmb-s) R_\lambda(H_s)V) \chi'(\rlmb) = (\rlmb-s)^{-1}\chi(\rlmb).
$$
Therefore,
$
  (1 + (\rlmb-s) R_\lambda(H_s)V)^2 \chi'(\rlmb) = 0.
$
Hence, if~$V$ is tangent to $\euR(\lambda),$ then~$\chi'(\rlmb)$ is a vector of order~$2$
and therefore the direction~$V$ has order at least~$2.$
\end{proof}

\begin{thm} \label{T: tangent V to order k}
Assume Assumption~\ref{A: Main Assumption}. Let~$k \geq 1,$ let
$H_{\rlmb}$ be a resonance point and~$V$ be a regular direction at~$H_{\rlmb}.$
If~$H(s)$ is a resonant path tangent to~$V$ at~$H_{\rlmb}$ to order at least~$k$
and if~$\chi(s)$ is a corresponding analytic resonant eigenpath,
then 
\begin{enumerate}
  \item[(i)] vectors 
$
  \chi(\rlmb),\chi'(\rlmb),\ldots,\chi^{(k-1)}(\rlmb)
$
have orders respectively $1,2,\ldots, k,$
  \item[(ii)] the direction~$V$ has order at least~$k,$
  \item[(iii)] for any $j=1,2,\ldots,k$ 
$$
  \bfA_\lambda(r_\lambda) \chi^{(k-1)}(\rlmb) 
        = (k-1)\chi^{(k-2)}(\rlmb) 
                  + \sum_{j=2}^{k-1} j!{k-1 \choose j} c_j \chi^{(k-1-j)}(\rlmb),
$$
where the numbers $c_2, \ldots, c_k$ are as in (\ref{F: H(s)=H0+sV+c2 s2 V+...}), and 
  \item[(iv)] the eigenvector~$\chi(\rlmb)$ has depth at least $k-1.$
\end{enumerate}
\end{thm}
\begin{proof}
(i) \ We prove this item using induction on~$k.$ 

That the vector~$\chi(\rlmb)$ has order 1 is trivial. 
That in case of $k \geq 2$ the vector~$\chi'(\rlmb)$ has
order 2 was proved in Lemma~\ref{L: tangent V}. 
Now, assuming that the assertion holds for $k=n-1,$ we prove it for
$k=n;$ still, we write $k$ instead of~$n.$

Since a path $H(s)$ which is tangent to~$V$ at $H_{\rlmb}$ to order at least~$k$ 
is also tangent to~$V$ to order at least~$k-1,$ it
follows from the induction assumption that the vectors
$$
  \chi(\rlmb), \ \chi'(\rlmb), \ \chi''(\rlmb), \ \ldots, \ \chi^{(k-2)}(\rlmb)
$$
have orders $1,2,\ldots,k-1$ respectively.
Differentiating~$k-1$ times the eigenvalue equation (\ref{F: H(s)chi(s)=lambda chi(s)}) gives 
\begin{equation} \label{F: sum C H(j)chi(k-j)=lambda chi(k)}
  \sum_{j=0}^{k-1} {k-1 \choose j} H^{(j)}(s) \chi^{(k-1-j)}(s) = \lambda \chi^{({k-1})}(s).
\end{equation}
Since~$H(s)$ is tangent to~$V$ at $H_{\rlmb}$ to order~$k,$ the operators $H'(\rlmb),
\ldots, H^{(k-1)}(\rlmb)$ are co-linear to~$V.$ Therefore, it follows
from the previous equality, taken with $s = \rlmb,$ that 
\begin{equation} \label{F: H0 chi(k)(0)+c2,c3,...=lambda chi(k)(0)}
  H_{\rlmb} \chi^{(k-1)}(\rlmb) + (k-1) V \chi^{(k-2)}(\rlmb) 
      + \tilde c_2 V \chi^{(k-3)}(\rlmb) + \ldots + \tilde c_{k-1} V\chi(\rlmb) 
      = \lambda \chi^{(k-1)}(\rlmb),
\end{equation}
where
$$
  \tilde c_j = {k-1 \choose j} j! c_j
$$
and the numbers $c_j$ are from (\ref{F: H(s)=H0+sV+c2 s2 V+...}).
Adding to both sides of (\ref{F: H0 chi(k)(0)+c2,c3,...=lambda chi(k)(0)})
the vector $(s-\rlmb)V\chi^{(k-1)}(\rlmb)$ gives 
$$
  (H_s - \lambda) \chi^{(k-1)}(\rlmb) + (k-1) V \chi^{(k-2)}(\rlmb) + \tilde c_2 V \chi^{(k-3)}(\rlmb) 
          + \ldots + \tilde c_{k-1} V\chi(\rlmb) = (s-\rlmb) V \chi^{(k-1)}(\rlmb).
$$
Since~$V$ is regular, there exists $s \in \mbR$ such that 
the inverse $R_\lambda(H_s)=(H_s-\lambda)^{-1}$ exists. 
Multiplying both sides of the last equality by $R_\lambda(H_s),$ 
we obtain the equality
\begin{equation}  \label{F: 11}
  \begin{split}
     \chi^{(k-1)}(\rlmb) + R_\lambda(H_s)V & \brs{(k-1) \chi^{(k-2)}(\rlmb) 
                  + \tilde c_2 \chi^{(k-3)}(\rlmb) + \ldots + \tilde c_{k-1} \chi(\rlmb)} 
     \\ & = (s-\rlmb) R_\lambda(H_s)V \chi^{(k-1)}(\rlmb),
  \end{split}         
\end{equation}
which can be rewritten as
\begin{equation*}
  \brs{1+(\rlmb-s) R_\lambda(H_s)V} \chi^{(k-1)}(\rlmb) = - R_\lambda(H_s)V \brs{(k-1) \chi^{(k-2)}(\rlmb) 
         + \tilde c_2 \chi^{(k-3)}(\rlmb) + \ldots + \tilde c_{k-1} \chi(\rlmb)}.
\end{equation*}
By the induction assumption, the vector in the last pair of brackets has order~$k-1.$
Since by the definition (\ref{F: res eq-n of order k}) of vectors of order~$k$
the operator $R_\lambda(H_s)V$ preserves order of vectors,
the vector in the right hand side has order~$k-1$ too.
Since by the same definition (\ref{F: res eq-n of order k}) the operator $1-sR_\lambda(H_s)V$ decreases
order of vectors by 1, it follows that the vector~$\chi^{(k-1)}(\rlmb)$ has order~$k.$

(ii) \ This item follows from (i).

(iii) \ Taking contour integrals of both sides of (\ref{F: 11}) over a contour enclosing the resonance point $s = \rlmb$
and using the formulas (\ref{F: Pz=oint A(s)}) and (\ref{F: A(lamb)=oint s A(s)})
we obtain
$$
  \bfA_\lambda(r_\lambda) \chi^{(k-1)}(\rlmb) = (k-1)\chi^{(k-2)}(\rlmb)
           + \tilde c_2 \chi^{(k-3)}(\rlmb) + \ldots + \tilde c_{k-1} \chi(\rlmb).
$$

(iv) This item follows immediately from (iii). 
\end{proof}

\smallskip
If a resonant path $H(s)$ is tangent to~$V$ to order~$k$ at a resonant point $H_{\rlmb},$
then changing the parameter~$s$ if necessary we can always make
the operators $H''(\rlmb),\ldots,H^{(k-1)}(\rlmb)$ equal to zero, so that the path $H(s)$ takes the form
\begin{equation} \label{F: standard path}
  H(s) = H_{\rlmb} + (s-\rlmb)V + O((s-\rlmb)^k), \ s \to \rlmb.
\end{equation}
For example, assuming that $\rlmb = 0,$ 
to eliminate the coefficient $c_2$ one can replace~$s$ by $t -c_2 t^2.$
Once $c_2$ is eliminated, the change of variables $s = t-c_3 t^3$ eliminates $c_3,$ and so on. 

\begin{defn}
A path of the form (\ref{F: standard path}) will be called \emph{standard}.\label{Page: standard path}
\end{defn}

According to Theorem~\ref{T: tangent V to order k},
with a resonant path tangent to order~$k$ we can associate a set of resonance vectors~$\chi_0, \ldots, \chi_{k-1}$
of respective orders $1,\ldots,k;$ namely, the first~$k$ coefficients
\begin{equation} \label{F: chi(j)}
  \chi_j = \frac{1}{j!} \chi^{(j)}(\rlmb), \ j = 0, 1, 2, \ldots
\end{equation}
of the Taylor expansion of a resonant eigenpath~$\chi(s).$

\begin{prop} \label{P: amazing proposition} 
Assume Assumption~\ref{A: Main Assumption}. 
Let~$V$ be a regular direction at~$H_{\rlmb}$
and let $H(s)$ be a resonant path tangent to~$V$ at~$H_{\rlmb}$ to
order~$k.$ The path $H(s)$ is standard if and only if
for all~$j=1,2,\ldots,k-1$ there holds the equality 
\begin{equation} \label{F: bfA chi(k)=k chi(k-1)}
  \bfA_\lambda(r_\lambda) \chi_j = \chi_{j-1},
\end{equation}
where~$\chi(s)$ is a corresponding analytic path of eigenvectors.
\end{prop}
\begin{proof}
This follows from item (iii) of Theorem~\ref{T: tangent V to order k}. 
\end{proof}

\subsection{A resonance curve associated with an eigenvector}

Assume that~$V$ is a regular direction at a resonance point~$H_{\rlmb}.$
We choose another direction~$W$ and consider the real affine
plane $\alpha = H_{\rlmb} + \mbR V + \mbR W$ in the affine space~$\clA$
determined by the point~$H_{\rlmb}$ and the directions~$V$ and~$W.$
It is possible that the intersection of $\alpha$ and the resonance set $\euR(\lambda)$ 
in a neighbourhood of~$H_{\rlmb}$ consists of only one point~$H_{\rlmb}.$
To avoid this one can choose~$W$ to be transversal to the resonance set. Since, as we shall see later, 
the resonance set has co-dimension 1, this will ensure that the intersection 
of the plane with the resonance set is a curve.

\begin{thm} \label{T: intersection is a curve}
Assume Assumption~\ref{A: Main Assumption}.
Further, let~$H_{\rlmb}$ be a resonance point and let~$V$ be a regular direction.
Let~$\chi$ be an eigenvector of~$H_{\rlmb}$ and let~$W = \scal{\chi}{\cdot}\chi.$
Then the intersection of the real affine plane 
$$
  \alpha := H_{\rlmb} + \mbR V + \mbR W
$$
with a sufficiently small neighbourhood of the point~$H_{\rlmb}$ in $\Rset \setminus\set{H_{\rlmb}+\mbR W}$ 
consists of one and only one one-dimensional analytic curve~$\gamma.$ 
Moreover, this curve is simple. 
\end{thm}
\begin{proof}
First we show that any neighbourhood of~$H_{\rlmb}$ has a resonance point 
in the affine plane~$\alpha$ which is not in~$H_{\rlmb}+\mbR W.$
Assume the contrary. Then there exists a convex neighbourhood~$O$ of~$H_{\rlmb}$
which does not have a resonance point outside the line~$H_{\rlmb}+\mbR W.$ 
Fix small enough $s_0$ so that $H_{\rlmb} \pm s_0 W$ are in the neighbourhood~$O.$ 
The operators $H_{\rlmb} + s_0 W$ and $H_{\rlmb} - s_0 W$ have the eigenvalue~$\lambda$ of multiplicity $m-1$
and non-degenerate eigenvalues~$\lambda+s_0$ and~$\lambda-s_0$ respectively. 

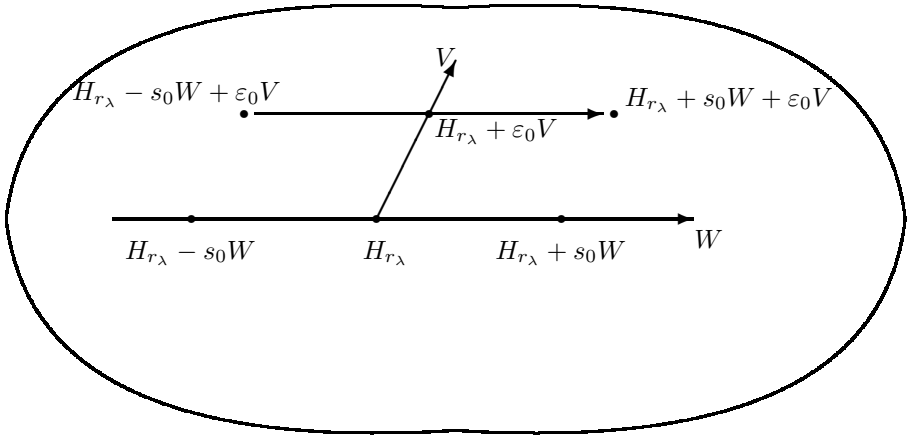
\begin{picture}(350,200)
\thicklines
\qbezier(10,100)(20,190)(180,180)
\qbezier(180,180)(340,190)(350,100)

\qbezier(10,100)(20,10)(180,20)
\qbezier(180,20)(340,10)(350,100)

\put(150,100){\vector(1,0){120}}
\put(267,89){\small~$W$}

\put(150,100){\line(-1,0){100}}
\put(150,100){\circle*{3}}
\put(145,85){\small $H_{\rlmb}$}

\put(220,100){\circle*{3}}
\put(195,85){\small $H_{\rlmb}+s_0 W$}

\put(80,100){\circle*{3}}
\put(55,85){\small $H_{\rlmb}-s_0 W$}

\put(150,100){\vector(1,2){30}}
\put(169,158){\small~$V$}

\put(170,140){\circle*{3}}
\put(172,131){\small $H_{\rlmb}+\eps_0 V$}

\put(240,140){\circle*{3}}
\put(244,143){\small $H_{\rlmb}+s_0 W+\eps_0 V$}

\put(100,140){\circle*{3}}
\put(35,145){\small $H_{\rlmb}-s_0 W+\eps_0 V$}

\put(104,140){\vector(1,0){132}}

\end{picture}

\noindent
Since~$V$ is regular, for all small enough non-zero $\eps$ the operators $H_{\rlmb}+\eps V$ 
are non-resonant. 
We choose a small enough $\eps_0>0$ so that
all operators $H_{\rlmb} + s W + \eps V,$ $(s,\eps) \in [-s_0,s_0] \times [-\eps_0,\eps_0],$
are in the neighbourhood~$O.$
We also choose $\eps_0$ small enough so that the perturbed eigenvalues 
$\lambda+s_0+\ldots$ and~$\lambda-s_0+\ldots$ of the perturbed operators
$H_{\rlmb} \pm s_0 W+\eps_0 V$ are on the same side of~$\lambda$ 
as the original non-perturbed eigenvalues~$\lambda+s_0$ and~$\lambda-s_0$ respectively. 
The other $m-1$ eigenvalues
of the operators 
$$
  H_{\rlmb} \pm s_0 W + \eps V \big|_{\eps=0}
$$
will move away from~$\lambda$ as $\eps$ becomes non-zero, since by assumption 
there are no resonance points in the neighbourhood $O$ except the line $H_{\rlmb}+\mbR W.$
Thus, when we deform 
$H_{\rlmb} - s_0 W+\eps_0 V$ to $H_{\rlmb} + s_0 W+\eps_0 V$ by changing $-s_0$ to $s_0,$ the number of eigenvalues 
on each side of~$\lambda$ changes. Hence, for some~$s$ between $-s_0$ and $s_0$ the
operator $H_{\rlmb} + s W+\eps_0 V$ must have~$\lambda$ as an eigenvalue. This operator is therefore resonant
and belongs to the neighbourhood. This is a contradiction with our assumption. 

The assertion that there can be only one simple analytic resonance curve in a neighbourhood of $H_{\rlmb}$ 
in the plane $\alpha$ follows from the fact that the operator~$W$ is non-negative and hence 
an eigenvalue of $H_{\rlmb} + s W+\eps_0 V$ can only move in the positive direction as~$s$ increases.
\end{proof}

\begin{defn} The curve $\gamma$ which exists and is uniquely determined by Theorem~\ref{T: intersection is a curve}
we denote by\label{Page: gamma chi}
$
  \gamma_\chi,
$
or by $\gamma_{\chi}(\lambda,H_{\rlmb},V),$ if necessary. 
\end{defn}

Since the resonant curve $\gamma_{\chi}$ is simple, there exists only one (up to scaling)
analytic resonant eigenpath~$\chi(s)$ corresponding to $\gamma_{\chi}.$
Therefore, to an eigenvector~$\chi$ we can assign another eigenvector~$\chi(\rlmb).$

\begin{thm} \label{T: chi(0) is exactly that}
Assume Assumption~\ref{A: Main Assumption} and let $H_{\rlmb},$~$V$ and~$\chi$ be 
as in Theorem~\ref{T: intersection is a curve}. If~$\chi(s)$ is a resonant eigenpath 
corresponding to the curve of operators $\gamma_{\chi}$ then the vectors~$\chi(\rlmb)$ 
and~$\chi$ are co-linear. 
\end{thm}
\begin{proof} If for~$H_{\rlmb}$ the eigenvalue~$\lambda$ is simple, then the assertion is trivial.
We shall reduce the general case to the case of simple eigenvalue~$\lambda,$
by slightly perturbing~$H_{\rlmb}.$ 

Let $G_t = H_{\rlmb} + t W,$ where~$W$ as in Theorem~\ref{T: intersection is a curve}.
The operator $G_t$ has an eigenvalue~$\lambda$
of multiplicity $m-1$ and, for all small enough $t,$ it has 
an eigenvalue~$\lambda + t$ of multiplicity $1$ 
corresponding to the eigenvector~$\chi.$ Since~$\lambda+t$ is a simple eigenvalue of $G_t,$
the curve 
$$
  \gamma_{\chi}(\lambda+t,G_t,V)
$$ 
has a resonant eigenpath~$\chi_t(s)$ which starts at~$\chi.$
When~$t$ is deformed to zero, the curve $\gamma_{\chi}(\lambda+t,G_t,V)$ is analytically 
deformed to~$\gamma_{\chi}(\lambda,H_{\rlmb},V).$
Thus, the resonant eigenpath~$\chi_t(s)$ gets deformed to~$\chi_0(s)$ in such a way that 
the base of this resonant eigenpath stays co-linear to~$\chi$ for all~$t.$ Hence, in the final position
$\chi_0(s)$ of this deformation the base of~$\chi_0(s)$ will still be co-linear to~$\chi.$
Finally, it is left to note that any two eigenpaths corresponding to $\gamma_{\chi}$ are co-linear,
since $\gamma_\chi$ is a simple curve. 
Hence,~$\chi(s)$ starts at~$\chi$ too. 
\end{proof}

\subsection{High order directions are tangent}

Now we are going to prove the reverse to Theorem~\ref{T: tangent V to order k}:
if~$V$ is a direction of order~$k$ at~$H_{\rlmb}$ then~$V$ is tangent to~$\Rset$ at~$H_{\rlmb}$ to order~$k.$

\begin{thm} \label{T: high order then tangent}
Assume Assumption~\ref{A: Main Assumption} and 
let~$k$ be an integer greater than~$1.$
If~$\chi_0$ is an eigenvector of depth at least $k-1$ for the triple $(\lambda; H_{\rlmb},V),$
then 
\begin{enumerate}
  \item the direction~$V$ is tangent to the resonance curve $\gamma_{\chi_0}$ to order at least $k,$
  \item in the Taylor expansion 
\begin{equation} \label{F: chi(s)=chi0+s*chi1+...}
  \chi(s) = \sum_{j=0}^\infty (s-\rlmb)^j \chi_j 
\end{equation}
of a resonant eigenpath $\chi(s)$ corresponding to $\gamma_{\chi_0},$ 
the vectors 
$
  \chi_0, \chi_1, \ldots, \chi_{k-1}
$
have orders respectively $1,2,\ldots,k,$
  \item for any resonant eigenpath (\ref{F: chi(s)=chi0+s*chi1+...}) corresponding to 
$\gamma_{\chi_0}$ and for all $j=1,2,\ldots,k-1$ the vector
$\bfA_\lambda \chi_j$ is a linear combination of vectors
$\chi_0, \chi_1, \ldots, \chi_{j-1}.$ Moreover, if the parametrisation 
of the curve $\gamma_{\chi_0}$ is standard then 
$
  \bfA_\lambda \chi_j = \chi_{j-1}.
$
  \item the vectors 
$
  \chi_0, \chi_1, \ldots, \chi_{k-2},
$
have depth at least one and are~$V$-orthogonal to $\clV_\lambda.$
\end{enumerate} 
\end{thm}
\begin{proof}
Let $H(s)$ be a parametrisation of the resonance curve $\gamma_{\chi_0}.$ 
Since the curve $\gamma_{\chi_0}$ is the intersection of the resonance set by the plane $H_{\rlmb}+\mbR V + \mbR W,$
where $W = \scal{\chi_0}{\cdot}\chi_0,$ 
the Taylor expansion of the path $H(s)$ has the form
\begin{equation} \label{F: H(s)=H0+s*(alpha1*V+beta1*W)+...}
  H(s) = H_{\rlmb} + \sum_{j=1}^\infty (s-\rlmb)^j(\alpha_j V + \beta_j W).
\end{equation}
By Theorem~\ref{T: chi(0) is exactly that}, 
a resonant eigenpath~$\chi(s)$ corresponding to this path has Taylor series
$$
  \chi(s) = \chi_0 + (s-\rlmb)\chi_1 + (s-\rlmb)^2\chi_2 + \ldots
$$
which starts at the vector~$\chi_0,$ that is,~$\chi(\rlmb) = \chi_0.$ 
Comparing the coefficients of~$s-\rlmb$ on both sides of the eigenvalue equation
(\ref{F: H(s)chi(s)=lambda chi(s)}) gives 
$$
  (H_{\rlmb}-\lambda)\chi_1 = -(\alpha_1 V + \beta_1 W)\chi_0.
$$
The vector $(H_{\rlmb}-\lambda)\chi_1$ is orthogonal to the eigenspace~$\clV_\lambda$ and in particular it is
orthogonal to~$\chi_0.$ Hence,
$$
  \scal{\chi_0}{(\alpha_1 V + \beta_1 W)\chi_0} = 0.
$$
Since~$\chi_0$ has depth at least one, it follows from Theorem~\ref{T: TFAE depth 1 criterion}
that $\scal{\chi_0}{V\chi_0} = 0.$
Combining this equality with previous one implies that 
$
  \beta_1 = 0,
$
and therefore, $H(s)$ is tangent to~$V$ at~$H_{\rlmb}$ (to order at least 2).
So, we have 
$
  (H_{\rlmb}-\lambda)\chi_1 = -\alpha_1 V \chi_0.
$
This equality implies that 
$$
  \chi_1 = -\alpha_1 S_\lambda \chi_0 + \text{order 1 vector}.
$$
Since~$\chi_0$ has depth at least 1, the vector $V\chi_0$ is orthogonal to $\clV_\lambda.$ 
Hence, by Theorem~\ref{T: TFAE depth 1 criterion} the previous equality implies 
$
  \bfA_\lambda \chi_1 = \alpha_1 \chi_0.
$
In particular,~$\chi_1$ is a vector of order~$2.$ Further, 
if the parametrisation of $\gamma_{\chi_0}(s)$ is standard, then $\alpha_1=1.$ 

We have proved the theorem in case of $k=2.$ We proceed by induction on~$k.$
So, assume that the claim holds for values of $k$ not greater than~$n$ 
and let~$\chi_0$ be an eigenvector of depth~$\geq n.$
Differentiating~$n$ times the eigenvalue equation 
$H(s) \chi(s) = \lambda \chi(s)$ 
we obtain
$$
  \sum_{j=0}^n {n \choose j} H^{(j)}(s) \chi^{(n-j)}(s) = \lambda \chi^{(n)}(s).
$$
Letting $s = \rlmb$ and replacing~$\chi^{(j)}(\rlmb)/j!$ by~$\chi_j$ gives 
$$
  H_{\rlmb} \chi_n + \sum_{j=1}^n (\alpha_j V + \beta_j W) \chi_{n-j} = \lambda \chi_{n}.
$$
By the induction assumption, we have 
\begin{equation} \label{F: beta1=...=beta(n-1)=0}
  \beta_1 = \ldots = \beta_{n-1} = 0.
\end{equation}
Hence,
\begin{equation} \label{F: (H0-l)chi(n)+...=0}
  (H_{\rlmb} - \lambda)\chi_n + \sum_{j=1}^{n-1} \alpha_j V \chi_{n-j} + (\alpha_n V + \beta_n W) \chi_0 = 0.
\end{equation}
Since~$\chi_0$ has depth at least $n,$ for some vector~$g$ we have~$\chi_0 = \bfA_\lambda ^n g.$
Since, by the induction assumption,~$\chi_j,$ $j=0,1,\ldots,n-1,$ is a vector of order $j+1,$ 
it follows that for all $j = 0,1,\ldots,n-1$
$$
  \scal{\chi_0}{V\chi_{j}} = \scal{\bfA_\lambda ^n g}{V\chi_{j}} = \scal{g}{V\bfA_\lambda ^n \chi_{j}} = 0.
$$
Hence, it follows from (\ref{F: (H0-l)chi(n)+...=0})
by taking the scalar product of the left hand side and~$\chi_0$ that 
\begin{equation} \label{F: beta(n)=0}
  \beta_n = 0
\end{equation}
and so
$$
  (H_{\rlmb} - \lambda)\chi_n + \sum_{j=1}^{n} \alpha_j V \chi_{n-j} = 0.
$$
It follows from (\ref{F: H(s)=H0+s*(alpha1*V+beta1*W)+...}), (\ref{F: beta1=...=beta(n-1)=0}) and 
(\ref{F: beta(n)=0}) that~$V$ is tangent to $\gamma_{\chi_0}$ to order at least~$n+1.$
Further, the vector $(H_{\rlmb}-\lambda)\chi_n$ is orthogonal to $\clV_\lambda$
and the vectors 
$
  V\chi_0, \ \ldots, V\chi_{n-2} 
$
are also orthogonal to $\clV_\lambda$ by the induction assumption 
(since the vectors~$\chi_0,\ldots,\chi_{n-2}$ have depth at least one). 
Hence, according to the last equality, so is the vector $V \chi_{n-1}.$ 
According to the last equality, we also have
$$
  \chi_n + \sum_{j=1}^{n} \alpha_j S_\lambda \chi_{n-j} = \text{order 1 vector}.
$$
Hence, by Theorem~\ref{T: TFAE depth 1 criterion}, it follows that 
$$
  \bfA_\lambda \chi_n = \sum_{j=1}^{n} \alpha_j \chi_{n-j}.
$$
Since, by the induction assumption, the vectors~$\chi_0, \ldots, \chi_{n-2}$ 
have depth at least one, the last equality implies that the vector~$\chi_{n-1}$
also has depth at least one. 

Further, if the parametrisation of $\gamma_{\chi_0}(s)$ is standard, then $\alpha_2 = \ldots = \alpha_{n} = 0$
and $\alpha_1 = 1$ and therefore
$
  \bfA_\lambda \chi_n = \chi_{n-1}.
$
\end{proof}

Theorem~\ref{T: high order then tangent} combined with Theorem~\ref{T: tangent V to order k}
proves the following theorem.
\begin{thm} \label{T: V is k-tangent iff order >=k}
Assume Assumption~\ref{A: Main Assumption}. A regular direction at any resonance point~$H_{\rlmb}$
is tangent to order at least~$k$ if and only if the order of the direction is at least~$k.$
\end{thm}

Theorem~\ref{T: V is k-tangent iff order >=k} has the following corollary.
\begin{thm} \label{T: V simple iff transversal}
Assume Assumption~\ref{A: Main Assumption}.
A regular direction at a resonance point is simple if and only if it is transversal.
\end{thm}
The last two theorems give geometric interpretation of order of a regular direction
in the case where~$\lambda$ is outside the essential spectrum.

\section{Resonance points as functions of the spectral parameter}
\label{S: res points as f-ns of s}

In this section we study resonance points~$r_z$ as functions of the spectral parameter~$z.$
To stress on this, we will often write $r(z)$ instead of~$r_z.$ 

\subsection{Order of a resonance function~$r_z=r(z)$}
In this subsection we consider the following question. Every resonance point~$r_z$
corresponding to a complex number~$z$ from the complement of the essential spectrum
has an order, which is a positive number. A natural question is how the order of~$r_z$ depends on~$z.$
In the following theorem we show that the order of~$r_z$ is equal to 1 for all values of~$z$ except
a discrete set, provided that~$r_z$ admits analytic continuation to a gap in the essential spectrum
in the real axis and that the analytic continuation has a real value at least at one point.
We conjecture that this property holds in general without this assumption, but since we are interested
in analytic continuation of~$r_\lambda,$ this hypothesis automatically holds in our case.

\begin{thm} \label{T: r(z) has order 1} 
Let~$r_\lambda$ be a real resonance point of the triple $(\lambda; H_{0},V).$
Analytic continuation of~$r_\lambda$ as a function of the complex variable~$\lambda$ 
has order 1 except a discrete set.
\end{thm}
\begin{proof}
Assume the contrary: there exists a non-empty open subset $G$ of the domain of holomorphy of~$r_z$ 
such that for all $z \in G$ the resonance point~$r_z$ has order at least two.
Then by \cite[Corollary 3.4.7]{Az9} there exist holomorphic vector-functions 
$\chi_1(z)$ and~$\chi_2(z)$ such that 
\begin{equation} \label{F: (Hrz-z)chi1=0}
  (H_{0}+r(z)V - z)\chi_1(z) = 0
\end{equation}
and 
\begin{equation} \label{F: (Hrz-z)chi2=-V chi1}
  (H_{0}+r(z)V - z)\chi_2(z) = - V\chi_1(z).
\end{equation}
Differentiation of the equality (\ref{F: (Hrz-z)chi1=0})
with respect to~$z$ gives
\begin{equation} \label{F: (r'zV-1)chiz=...}
  (r'(z)V-1)\chi_1(z) + (H_{0}+r(z)V-z) \chi'_1(z) = 0.
\end{equation}
Let $\phi(\bar z)$ be an anti-resonance 
vector-function of order 1, that is,
\begin{equation} \label{F: bar version of (Hrz-z)chi1=0}
  (H_{0} + \bar r(z)V - \bar z)\phi(\bar z) = 0.
\end{equation}
Since 
\begin{equation*} 
  \begin{split}
  \scal{\phi(\bar z)}{(H_{0}+r(z)V - z) \chi'_1(z)} 
         & = \scal{(H_{0} + \bar r(z)V - \bar z)\phi(\bar z)}{ \chi'_1(z)} 
      \\ & = 0,
  \end{split}
\end{equation*}
taking the scalar product of $\phi(\bar z)$ with both sides of 
the equality (\ref{F: (r'zV-1)chiz=...}) gives 
$$
  \scal{\phi(\bar z)}{(r'(z)V-1)\chi_1(z)} = 0.
$$
Further, the equality (\ref{F: (Hrz-z)chi2=-V chi1}) implies that 
\begin{equation} \label{F: (chi1,V chi1)=0}
  \begin{split}
     - \scal{\phi(\bar z)}{V\chi_1(z)} & = \scal{\phi(\bar z)}{(H_{0}+r(z)V - z)\chi_2(z)} 
     \\ & = \scal{(H_{0} + \bar r(z)V - \bar z)\phi(\bar z)}{\chi_2(z)} 
     \\ & = 0.
  \end{split}
\end{equation}
Combining this with the previous equality implies
$$
  \scal{\phi(\bar z)}{\chi_1(z)} = 0.
$$
Since $r(z)$ takes real values in some interval $I$ of the real axis,
we can take~$\chi_1(z)$ to be holomorphic extension of a first order vector-function in $I,$
and we can take $\phi(\bar z)$ to be 
anti-holomorphic extension of the same function. Since the scalar product is anti-linear 
in the first argument, the scalar product of this pair of holomorphic and anti-holomorphic vector-functions
will be holomorphic and it would vanish on the interval $I$ of the real axis. This implies that both these 
functions are zero in the gap $I$ and therefore everywhere. Since~$\chi_1(z)$ is an eigenvector, this gives
a contradiction. 
\end{proof}

\subsection{Cycles of resonance points}

Let~$r_\lambda$ be a resonance point of geometric multiplicity~$m$
and algebraic multiplicity~$N.$ 
When~$\lambda$ is shifted to $z=\lambda+iy$ with small $y>0,$ the resonance point~$r_\lambda$
splits into~$N=\dim \Upsilon_\lambda(\rlmb)$ resonance points~$r_z^{(j)},$ counting algebraic multiplicities. 
The resonance points~$r_z^{(j)}$ are holomorphic functions of~$z.$
When~$z$ makes one round around~$\lambda,$ these~$N$ holomorphic functions undergo a permutation.
We shall show in this subsection that this permutation consists of~$m$ disjoint cycles of lengths $d_1, \ldots, d_m,$
where~$m$ is the number of Jordan cells of the compact operator~$A_\lambda(s)$ corresponding to the eigenvalue
$(s-r_\lambda)^{-1},$ and~$d_\nu$ is the size of~$\nu$-th cell. 

For convenience, in this subsection we shall often indicate dependence 
of a resonance point~$r_z$ on the spectral parameter~$z$ in the usual way as 
$r(z)$ instead of using subindex notation. If there is no danger of confusion, we may choose to drop 
the variable~$z$ from the notation altogether.

\smallskip In section~\ref{S: On phi[nu]} we studied eigenvalues~$\lambda_\nu(r)$ 
of the operator $H_r = H_0 + r V$ 
as functions of the coupling constant~$r.$ The coupling constant~$r$ was treated as a real variable. 
In this section we consider the coupling constant~$r$ as a function of the spectral parameter~$\lambda,$
but unlike section~\ref{S: On phi[nu]}, we shall treat both variables~$r$ and~$\lambda$ as complex variables. 
Since the spectral
variable treated as a complex variable is denoted by~$z,$ the functions under study are~$r_\nu(z),$
which are inverses of~$\lambda_\nu(r).$ According to Theorem~\ref{T: TFAE for (Uk)}, 
eigenvalue functions~$\lambda_\nu(r)$ of order $\tilde d_\nu>1$ satisfy~$\lambda'_\nu(r_\lambda) = 0.$
Hence, in general the corresponding inverse function~$r_\nu(z)$ is a branching multi-valued 
holomorphic function in a neighbourhood of $z = \lambda.$ 

\begin{thm} \label{T: about res cycles}
(a) For each cycle $r_\nu^{(\cdot)}(z)$ of resonance points there exists $\eps>0$\label{Page: r(nu)(.)(z)}
and a resonance point $r_\nu^{(0)}(z)$ of this cycle which takes real values for all~$z$ from an interval~$I,$ where~$I$ is either $[\lambda,\lambda+\eps)$ 
or $(\lambda-\eps,\lambda].$ \\

(b) The number of real resonance points in a cycle for $z \in I$ is either one or two. \\

(c) In the case there are two real resonance points $r'$ and $r''$ in a cycle for $z \in I,$ 
the numbers $r'-r_\lambda$ and $r''-r_\lambda$ have different signs. \\

(d) The numbers of non-real resonance points in each cycle for $z \in I$ are the same in~$\mbC_+$ and~$\mbC_-.$ 

(e) In the case there are two real resonance points $r'$ and $r''$ in a cycle for $z \in I,$ 
they shift to different half-planes~$\mbC_+$ and~$\mbC_-$ as $z \in I$ is shifted to $z+iy$ with small $y>0.$ 
\end{thm}
\begin{proof} Part (a) follows from the fact that an isolated eigenvalue~$\lambda$ of $H_0+r_\lambda V$
is stable, that is, an eigenvalue~$\lambda_\nu(s)$ of $H_s$ depends on~$s$ continuously. 

Parts (b) and (c) follow from the formula (see (\ref{F: lambda(s)=lambda +eps(s-l)...}))
$$
  \lambda_\nu(s) = \lambda + \eps_\nu (s-\rlmb)^{\tilde d_\nu} + O((s-\rlmb)^{\tilde d_\nu+1}), \ s \to \rlmb,
$$
where $\tilde d_\nu$ is order of $\lambda_\nu(s).$ 
Namely, with $\eps>0$ sufficiently small, 
if the order $\tilde d_\nu$ is odd, then there is one and only one real resonance point for all $z \in I,$
where $I = [\lambda,\lambda+\eps)$ or $I = (\lambda-\eps,\lambda].$
In this case, the interval~$I$ can be chosen to be either of the intervals $[\lambda,\lambda+\eps)$
or $(\lambda-\eps,\lambda].$
If the order $\tilde d_\nu$ is even, then there are exactly two real resonance points for all $z \in I$ and they are
located on different sides of~$r_\lambda$ in the real axis of the coupling constant. 
In this case, if $\eps_\nu>0,$ then the interval~$I$ is $[\lambda,\lambda+\eps)$
and if $\eps_\nu<0,$ then the interval~$I$ is $[\lambda-\eps,\lambda).$

Proof of part (d). By Lemma~\ref{L: if r then so is bar r}, if~$r_z$ is a resonance point 
corresponding to~$z,$ then~$\bar r_z$ is a resonance point corresponding to~$\bar z.$
Hence, the set of resonance points corresponding to a real value of~$\lambda$
is symmetric with respect to the real axis. We still need to show that if a resonance point $r^{(j)}_\nu(z)$
belongs to a cycle~$\nu,$ then its conjugate also belongs to the same cycle, but this readily follows from 
the Schwarz reflection principle. 

Proof of part (e). By part (d), for real values of~$z$ from $I$
the set of resonance points in a cycle is symmetric with respect to the real axis. 
Combining this with the fact that for non-real values of~$z$ there can be no real resonance points, 
one can infer the claim. 
\end{proof}

The following figure demonstrates this theorem. In this figure 
there are two cycles of lengths~$d_1=5$ (black dots) and~$d_2=4$ (white dots).

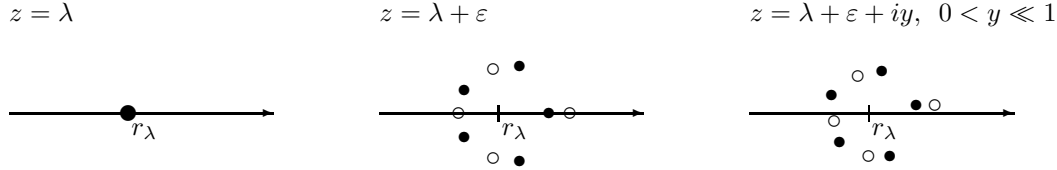
\begin{picture}(125,90)
\put(10,75){\small $z=\lambda$}
\put(10,40){\vector(1,0){100}}
\put(55,40){\circle*{6}}    
\put(53,32){{\small~$r_\lambda$}}
\end{picture}
\hskip 0.4 cm 
\begin{picture}(125,90)
\put(10,75){\small $z=\lambda+\eps$}
\put(10,40){\vector(1,0){100}}
\put(55,37){\line(0,1){6}}    
\put(53,32){{\small~$r_\lambda$}}
\put(74,40){\circle*{4}}   
\put(63,58){\circle*{4}}   
\put(63,22){\circle*{4}}   
\put(42,49){\circle*{4}}   
\put(42,31){\circle*{4}}   

\put(82,40){\circle{4}}   
\put(40,40){\circle{4}}   
\put(53,57){\circle{4}}   
\put(53,23){\circle{4}}   
\end{picture}
\hskip 0.4 cm 
\begin{picture}(125,90)
\put(10,75){\small $z=\lambda+\eps+iy,$ \ $0 < y <\!\!\!<1$ }
\put(10,40){\vector(1,0){100}}
\put(55,37){\line(0,1){6}}    
\put(53,32){{\small~$r_\lambda$}}

\put(73,43){\circle*{4}}   

\put(60,56){\circle*{4}}   
\put(63,24){\circle*{4}}   

\put(41,47){\circle*{4}}   
\put(44,29){\circle*{4}}   

\put(80,43){\circle{4}}   
\put(42,37){\circle{4}}   
\put(51,54){\circle{4}}   
\put(55,24){\circle{4}}   
\end{picture}

\begin{prop} \label{P: number of cycles = m} 
The number of cycles of the permutation of the~$N$ resonance points~$r_z^{(j)}$
is equal to geometric multiplicity~$m$ of the resonance point~$r_\lambda.$
More precisely, there is a natural one-to-one correspondence between
cycles of resonance points $r_\nu^{(j)}(z)$ and the eigenvalue functions~$\lambda_\nu(s)$
of the operator~$H_s,$ given by the following diagram:
\begin{equation} \label{F: phi[mu] one-to-one r[nu]}
  \phi_\nu(s) \quad \leftrightarrow \quad \lambda_\nu(s) \quad \leftrightarrow \quad r_\nu^{(\cdot)}(z).
\end{equation}
That is, with an eigenpath $\phi_\nu(s)$ we associate an eigenvalue function~$\lambda_\nu(s)$
and the inverse of this eigenvalue function is the multi-valued holomorphic function~$r_\nu^{(\cdot)}(z).$
\end{prop}
\begin{proof} This immediately follows from parts (a), (b) and (c) of Theorem~\ref{T: about res cycles}. 
\end{proof}


\begin{cor} 
Order $\tilde d_\nu$ of eigenpath $\phi_\nu(s)$ is equal 
to the size~$\hat d_\nu$ of a cycle corresponding to the eigenpath~$\phi_\nu(s).$
\end{cor}
\begin{proof} 
By items (i) and (iii) of Theorem~\ref{T: TFAE for (Uk)}, the order $\tilde d_\nu$
of $\phi_\nu(s)$ is a number determined by the equality~(\ref{F: lambda(s)=lambda +eps(s-l)...}).
This implies that the inverse of the function $\lambda_\nu(s)$ in a neighbourhood of $s = r_\lambda$
is a multivalued function $r^{(\cdot)}_\nu(z)$ with $\tilde d_\nu$ branches in a neighbourhood of $z = \lambda,$ 
and these branches form a single cycle. Hence, $\tilde d_\nu = \hat d_\nu.$ 
\end{proof}

Since the sum of cycle sizes $\hat d_\nu$ of resonance points is equal to $N,$ this gives the following corollary.
\begin{cor} \label{C: sum tilde d = N} The sum of orders of eigenpaths $\phi_\nu(s)$ is equal to the algebraic 
multiplicity~$N$ of the resonance point~$r_\lambda.$
\end{cor}

\begin{cor} The~$m$ sets of vectors 
$$
  \frac 1{j!} \phi_\nu^{(j)}(\rlmb), \quad \nu=1,\ldots,m, \ j=0,1,\ldots,\tilde d_\nu-1
$$
form a Jordan basis for the nilpotent operator $\bfA_\lambda(r_\lambda).$ 
\end{cor}
\begin{proof} By item (v) of Theorem \ref{T: TFAE for (Uk)}, for each $\nu=1,\ldots,m$ and for each $j=0,1,\ldots,\tilde d_\nu-1$
we have $\bfA_\lambda(r_\lambda) \phi_\nu^{(j)}(\rlmb) = \phi_\nu^{(j-1)}(\rlmb).$
Since by Corollary \ref{C: sum tilde d = N} the sum of numbers $\tilde d_\nu$ is equal to the 
algebraic multiplicity~$N,$ we are done. 
\end{proof}

We collect these assertions in the following theorem.
\begin{thm} \label{T: depth of phi[nu] is d[nu]-1}
For each $\nu=1,\ldots,m,$ the following numbers are equal (assuming that they are arranged in decreasing order):
\begin{enumerate}
  \item the order of eigenpath $\phi_\nu(s),$
  \item the size of the cycle $\nu$ of resonance points of the group of $r_\lambda,$
  \item the size of the $\nu$th Jordan block of the nilpotent operator $\bfA_\lambda(r_\lambda).$ 
\end{enumerate}
Moreover, the vectors $\frac 1{j!} \phi_\nu^{(j)}(\rlmb), \ \nu=1,\ldots,m, \ j=0,1,\ldots,\tilde d_\nu-1,$ form
a Jordan basis of the nilpotent operator $\bfA_\lambda(r_\lambda).$ 
\end{thm}


\subsection{Decomposition of~$P_\lambda(r_\lambda)$}
Let 
$
 r_\nu^{(0)}, \ldots, r_\nu^{(d_\nu-1)}
$
be the cycle~$\nu.$ 
The function
$$
  P_z^{[\nu]} := \sum_{j=0}^{d_\nu-1} P_z(r_\nu^{(j)}(z))
$$
is single-valued in a neighbourhood of~$\lambda.$
\begin{prop} \label{P: P[nu](z) is holomorphic}
The function $P_z^{[\nu]}$ of~$z$ is holomorphic in a neighbourhood of~$\lambda.$
\end{prop}
\begin{proof} By \cite[Theorem II.1.8]{Kato}, the Laurent expansion of the function $P_z^{[\nu]}$
can have only finitely many terms with negative powers of $(z-\lambda).$
The sum
$
  \sum_{\nu=1}^m P_z^{[\nu]} 
$
converges to~$P_\lambda(r_\lambda)$ as $z \to \lambda$ and therefore is bounded 
in a neighbourhood of~$\lambda.$ Since 
$
 P_z^{[\nu]}P_z^{[\mu]} = \delta_{\nu\mu}P_z^{[\nu]},
$
each of the~$m$ functions $P_z^{[\nu]}$ is also bounded in a neighbourhood of~$\lambda.$ 
\end{proof}

Therefore, this operator has a limit as $z \to \lambda,$ which we denote by\label{Page: P(lambda)[nu]:=}
\begin{equation} \label{F: P(lambda)[nu]:=}
  P_\lambda^{[\nu]}(r_\lambda) := \lim_{z \to \lambda} P_z^{[\nu]}.
\end{equation}
We have the equalities
\begin{equation} \label{F: P[nu]P[mu]=0}
  P_\lambda^{[\nu]}P_\lambda^{[\mu]} = \delta_{\nu\mu}P_\lambda^{[\nu]}
\end{equation}
and 
$$
  P_\lambda(r_\lambda) = \sum_{\nu=1}^m P_\lambda^{[\nu]}(r_\lambda).
$$
Since $P_z(r_\nu^{(j)}(z))P_z(r_\nu^{(k)}(z))=0$ for different resonance points $r_\nu^{(j)}$ and $r_\nu^{(k)},$
the operator $P_z^{[\nu]}$ is an idempotent. 

\begin{lemma} \label{L: P[nu] and bfA commute} 
The operators $P_\lambda^{[\nu]}(r_\lambda)$ and~$\bfA_\lambda(r_\lambda)$ commute.
\end{lemma}
\begin{proof} Since the limits of $P_z^{[\nu]}$ and $A_z(s)$ as $z \to \lambda$ exist,
it is enough to take limits of both sides of the equality 
$
  P_z^{[\nu]} A_z(s) = A_z(s) P_z^{[\nu]}, 
$
and then use Laurent expansion of $A_\lambda(s).$ 
\end{proof}

In a similar way, we introduce operators $Q_z^{[\nu]}$ by formula 
$$
  Q_z^{[\nu]} := \sum_{j=0}^{d_\nu-1} Q_z(r_\nu^{(j)}(z)).
$$
Many properties of $Q_z^{[\nu]}$ are analogues to those of $P_z^{[\nu]}.$

It follows from (\ref{F: VP = QV}) that 
$
  V P_z^{[\nu]} = Q_z^{[\nu]}V.
$
Taking in this equality the limit $z \to \lambda,$ we obtain
\begin{equation} \label{F: VP[nu]=Q[nu]V}
  V P_\lambda^{[\nu]} = Q_\lambda^{[\nu]}V.
\end{equation}

Let \label{Page: Upsilon lambda nu}  
$
  \Upsilon_\lambda^{[\nu]} := \im P_\lambda^{[\nu]}.
$

\begin{cor} \label{C: A(lamb)(nu) is reduced} 
The nilpotent operator~$\bfA_\lambda(r_\lambda)$ is reduced 
by the vector space~$\Upsilon_\lambda^{[\nu]}.$
\end{cor}
We denote restriction of~$\bfA_\lambda(r_\lambda)$ to the vector space 
$\Upsilon_\lambda^{[\nu]}$ by \label{Page: bfA(nu)2}
$
  \bfA_\lambda^{[\nu]}(r_\lambda).
$
In previous sections we denoted by~$d_\nu$ the sizes of the Jordan cells 
of the nilpotent operator~$\bfA_\lambda(r_\lambda).$ 
The last lemma shows that numbers~$d_\nu$ and dimensions of the vectors 
spaces~$\Upsilon_\lambda^{[\nu]},$ and therefore lengths of the cycles, are equal. 

\begin{lemma} \label{L: dim(Ups[nu]cap clV)=1}
Dimension of the vector space 
$
 \Upsilon_\lambda^{[\nu]} \cap \clV_\lambda 
$
is equal to $1.$
\end{lemma}
\begin{proof} By Corollary~\ref{C: A(lamb)(nu) is reduced}, the dimension of 
$\Upsilon_\lambda^{[\nu]} \cap \clV_\lambda$ is at least~$1.$ 
Hence, this dimension is to be~$1,$ since otherwise we get 
a contradiction with Proposition~\ref{P: number of cycles = m}.
\end{proof}

\begin{lemma} \label{L: bfA[nu] is cyclic} 
The operator~$\bfA_\lambda^{[\nu]}$ is cyclic.
\end{lemma}
\begin{proof} This follows immediately from Corollary~\ref{C: A(lamb)(nu) is reduced}
and~\ref{L: dim(Ups[nu]cap clV)=1}. 
\end{proof}

\subsection{Two lemmas}

Since by Theorem~\ref{T: r(z) has order 1}
the functions $r_\nu^{(j)}(z)$ have order~$1,$ the operator valued function $A_z(s)$ 
near the (real) point~$r_\lambda$ has the Laurent expansion
\begin{equation} \label{F: Laurent of Az(s) at r(nu)(j)}
  A_z(s) = \tilde A_z(s) + \sum_{\nu=1}^m \sum_{j=0}^{d_\nu-1} \frac{P_z(r_\nu^{(j)}(z))}{s-r_\nu^{(j)}(z)},
\end{equation}
where $\tilde A_z(s)$ is a meromorphic function which has no poles 
in a neighbourhood of~$r_\lambda$ which includes all $r_\nu^{(j)}(z).$ 
The functions $\frac{P_z(r_\nu^{(j)}(z))}{s-r_\nu^{(j)}(z)}$ of~$z$ (with~$s$ fixed), 
taken individually, are not single-valued in a neighbourhood of~$\lambda,$ 
unless $d_\nu=1,$ but each of the~$m$ sums 
\begin{equation} \label{F: nu part of Laurent for Az(s)}
  \sum_{j=0}^{d_\nu-1} \frac{P_z(r_\nu^{(j)}(z))}{s-r_\nu^{(j)}(z)}
\end{equation}
is a single-valued function of~$z$ in a neighbourhood of~$\lambda.$
We also have the equality
\begin{equation} \label{F: P[nu]Az(s)=sum}
  P_z^{[\nu]} A_z(s) = \sum_{j=0}^{d_\nu-1} \frac{P_z(r_\nu^{(j)}(z))}{s-r_\nu^{(j)}(z)}.
\end{equation}

\begin{lemma} \label{L: a fcuking lemma}
As~$z \to \lambda,$ the fractional part  
$$
  P_z(r_\lambda)A_z(s) = \sum_{\nu=1}^m \sum_{j=0}^{d_\nu-1} \frac{P_z(r_\nu^{(j)}(z))}{s-r_\nu^{(j)}(z)}
$$
of the Laurent expansion~(\ref{F: Laurent of Az(s) at r(nu)(j)}) 
of the meromorphic function~$A_z(s)$ at~$N$ poles 
$$
  r_\nu^{(j)}(z), \ \nu=1,\ldots,m, \ j=0,\ldots,d_\nu-1,
$$ 
of the group of~$r_\lambda$
converges in the norm topology to the fractional part of the Laurent
expansion of $A_{\lambda}(s)$ at the pole~$r_{\lambda},$ which is
\begin{equation} \label{F: frac part at lambda}
  P_\lambda(r_\lambda)A_\lambda(s) = \frac{P_{\lambda}(r_{\lambda})}{s-r_{\lambda}} 
         + \frac{\bfA_{\lambda}(r_{\lambda})}{(s-r_{\lambda})^2} 
         + \ldots 
         + \frac{\bfA^{d-1}_{\lambda}(r_{\lambda})}{(s-r_{\lambda})^{d-1}}.
\end{equation}
The convergence is uniform with respect to~$s$ on compact subsets of a deleted neighbourhood of~$r_\lambda.$
\end{lemma}
\begin{proof} 
The meromorphic function $A_z(s)$ converges in norm to $A_\lambda(s)$ as $z \to \lambda$
uniformly on compact subsets of a deleted neighbourhood of the pole~$r_\lambda$ in~$\mbC.$
$P_z(r_\lambda)$ converges to~$P_\lambda(r_\lambda)$ as $z \to \lambda$ in norm (in fact, even in trace class
norm, see \cite[Lemma 5.2.2]{Az9}). The claim follows. 
\end{proof}

\begin{lemma} \label{L: a lemma}
For each $\nu=1,\ldots,m,$
the following equality holds:
\begin{equation} \label{F: frac part at lambda[nu]}
  \lim_{z \to \lambda} \sum_{j=0}^{d_\nu-1} \frac{P_z(r_\nu^{(j)}(z))}{s-r_\nu^{(j)}(z)}
  = 
  \frac{P^{[\nu]}_{\lambda}(r_{\lambda})}{s-r_{\lambda}} 
        + \frac{\bfA^{[\nu]}_{\lambda}(r_{\lambda})}{(s-r_{\lambda})^2} 
        + \ldots + \frac{\brs{\bfA^{[\nu]}_{\lambda}(r_{\lambda})}^{d_\nu-1}}{(s-r_{\lambda})^{d_\nu-1}},
\end{equation}
where~$d_\nu$ is the size of the cycle~$\nu.$
\end{lemma}
\begin{proof}
By Proposition~\ref{P: P[nu](z) is holomorphic}, 
$$
  \lim_{z \to \lambda} P_z^{[\nu]} A_z(s) = P_\lambda^{[\nu]} A_\lambda(s),
$$
where the convergence is in norm and is uniform on compact subsets of a deleted neighbourhood of~$r_\lambda.$
This equality is the same as (\ref{F: frac part at lambda[nu]}).
\end{proof}

\subsection{Puiseux series for $r_\nu^{(j)}(z)$}

For a positive integer $d$ let
$$
  \eps_d = e^{2\pi i /d}.
$$
For an integer $k$ and a positive integer $d$ let 
$$
  [d|k] = 
      \left\{  
         \begin{array}{cl}
             0, & \text{if $d$ does not divide $k,$} \\
             1, & \text{if $d$ divides $k.$} \\
         \end{array}
      \right. 
$$
For any integer~$k$ we have 
\begin{equation} \label{F: sum eps(jk)=[d|k]d}
  \sum_{j=0}^{d-1} \eps_d^{jk} = [d|k]\cdot d.
\end{equation}

For each $\nu=1,\ldots,m,$ the functions 
$$
  r_\nu^{(0)}(z), \ \ldots, \ r_\nu^{(d_\nu-1)}(z)
$$
are different branches of a
multi-valued holomorphic function with Puiseux series
\begin{equation} \label{F: Puiseux for rz}
  r_\nu^{(j)}(z) = \sum_{k=0}^\infty r_{k/d_\nu}\eps_{d_\nu}^{kj} (z-\lambda)^{k/d_\nu}, \ \ j = 0,\ldots,d_\nu-1.
\end{equation}
The Puiseux series for~$r_\nu^{(j)}(z)$ does not have the part with negative
powers of $(z-\lambda),$ since this function is continuous at~$z=\lambda.$ 

\begin{prop} The coefficients $r_{k/d_\nu}$ of the Puiseux series (\ref{F: Puiseux for rz}) are real.
\end{prop}
\begin{proof} More precisely, these numbers can be chosen to be real.
For $j=0$ we have 
$$
  r_\nu^{(0)}(z) = \sum_{k=0}^\infty r_{k/d_\nu}(z-\lambda)^{k/d_\nu}.
$$
Recall that~$\lambda$ is an isolated eigenvalue of the self-adjoint operator $H_{r_\lambda}.$
Due to stability of isolated eigenvalues,
one of the functions $r^{(j)}_\nu(z)$ is to take real values for real values of~$z$ close to~$\lambda$ on the left
or on the right. From this one can infer that the numbers $r_{k/d_\nu}$ are to be real. 
\end{proof}

\begin{prop} \label{P: first Puiseux coef is non-zero} 
The first coefficient $r_{1/d_\nu}$ of the Puiseux series (\ref{F: Puiseux for rz})
for $r_\nu^{(j)}(z)$ is non-zero.
\end{prop}
\begin{proof}
Let $r_{f_\nu/d_\nu}$ be the first non-zero coefficient of the Puiseux series (\ref{F: Puiseux for rz}).
Since~$r_\nu^{(\cdot)}(z)$ is the inverse of~$\lambda_\nu(s)$ (see diagram (\ref{F: phi[mu] one-to-one r[nu]})),
it follows from item (iii) of Theorem~\ref{T: TFAE for (Uk)} and (\ref{F: Puiseux for rz}) that 
$$
  \lambda_\nu(s) = \lambda + \eps_\nu (s-\rlmb)^{d_\nu/f_\nu} + O((s-\rlmb)^{d_\nu/f_\nu+1})
$$
with non-zero $\eps_\nu.$ Since~$\lambda_\nu(s)$ is analytic at $s = r_\lambda,$~$f_\nu$ 
divides~$d_\nu.$ If~$f_\nu>1,$ then this would imply
that the set of~$d_\nu$ numbers $\set{r_\nu^{(j)}(z)\colon j=0,\ldots,d_\nu-1}$ is not a cycle. 
Hence, the only possibility is that $f_\nu=1$ and we are done. 
\end{proof}

\subsection{Puiseux series for $P_z(r_\nu^{(j)}(z))$}

For each $\nu=1,\ldots,m,$ the operator-functions 
$$
  P_z(r^{(0)}_\nu(z)), \ldots, P_z(r^{(d_\nu-1)}_\nu(z))
$$ 
are different branches of a multi-valued holomorphic function with
Puiseux series
\begin{equation} \label{F: Puiseux for Pz}
  P_z(r_\nu^{(j)}(z)) = \tilde P^j_\nu(z) + \sum_{l=0}^p \eps_{d_\nu}^{-l j}(z-\lambda)^{-l/d_\nu}P_{-l/d_\nu}, 
        \ \ j = 0,\ldots,d_\nu-1
\end{equation}
where $\tilde P^j_\nu(z)$ is the part with positive powers of~$z-\lambda.$
Since by definition (\ref{F: Pz=oint A(s)}) 
the operator $P_z(r_\nu^{(j)}(z))$ is an eigenprojection of the compact
operator $A_z(s)$ corresponding to a non-zero eigenvalue
$$(s-r_\nu^{(j)}(z))^{-1},$$
by \cite[Theorem II.1.8]{Kato}, the part of the Puiseux 
expansion~(\ref{F: Puiseux for Pz}) with negative powers is finite. 


\smallskip
Our first aim is to prove that the upper limit~$p$ in the Puiseux series 
(\ref{F: Puiseux for Pz}) is equal to $d_\nu-1.$ For this we need two auxiliary lemmas. 

\begin{lemma} \label{L: comb frac part written as almost Laurent}
For each cycle of resonance points~$r_\nu^{(\cdot)}(z),$
\begin{equation} \label{F: comb frac part written as almost Laurent}
    \sum_{j=0}^{d_\nu-1} \frac {P_z(r_\nu^{(j)}(z))}{s-r_\nu^{(j)}(z)} 
        = \sum_{j=0}^{d_\nu-1} \frac{A_j^\nu(z)}{(s-r_\nu^{(0)}(z))\ldots(s-r_\nu^{(j)}(z))},
\end{equation}
where for $j=0,1,\ldots,d_\nu-1$ 
\begin{equation} \label{F: Aj=sum (rj-r1)...Pj}
  A_j^\nu(z) = \sum_{b=j}^{d_\nu-1} (r_\nu^{(b)}(z) - r_\nu^{(0)}(z))
               \ldots(r_\nu^{(b)}(z) - r_\nu^{(j-1)}(z))P_z(r^{(b)}_\nu(z)).
\end{equation}
\end{lemma}
\begin{proof} 
This lemma has general character as it holds for any set of numbers~$r_\nu^{(j)}(z)$
and any bounded operators~$P_z(r_\nu^{(j)}(z)).$ For this reason we write~$r_j$ 
for~$r_\nu^{(j)}(z),$ $P_j$ for~$P_z(r_\nu^{(j)}(z)),$ $d$ for~$d_\nu,$
$A_j$ for $A_j^\nu(z).$

We prove the equality (\ref{F: comb frac part written as almost Laurent}) by induction 
on~$d.$ For~$d=1$ the assertion is trivial. 
Assume that the assertion holds for smaller values of~$d.$ Since the operators~$A_j$ 
depend on~$d,$ we will indicate this dependence by writing~$A_j^{(d)}.$ 

By the induction assumption 
$$
  \sum_{j=0}^{d-1} \frac {P_j}{s-r_j} 
      = \sum_{j=0}^{d-2} \frac{A_j^{(d-1)}} {(s-r_0)\ldots(s-r_j)} 
          + \frac {P_{d-1}}{s-r_{d-1}}.
$$
We need to find operators $A_j^{(d)}$ such that 
$$
  \sum_{j=0}^{d-1} \frac{A_j^{(d)}}{(s-r_0)\ldots(s-r_j)} =  
         \sum_{j=0}^{d-2} \frac{A_j^{(d-1)}} {(s-r_0)\ldots(s-r_j)} 
              + \frac {P_{d-1}}{s-r_{d-1}}.
$$
Multiplying both sides of this equality by $\prod_{k=0}^{d-1}(s-r_{k})$ gives 
$$
  \sum_{j=0}^{d-1} A_j^{(d)} \prod_{k=j+1}^{d-1}(s-r_{k}) =  
         \sum_{j=0}^{d-2} A_j^{(d-1)} \prod_{k=j+1}^{d-1}(s-r_{k})
              + \prod_{k=0}^{d-2}(s-r_{k})P_{d-1}.
$$
Replacing~$s$ by~$r_{d-1}$ gives
$$
  A_{d-1}^{(d)} = 0 + \prod_{k=0}^{d-2}(r_{d-1} - r_{k})P_{d-1}.
$$
Replacing~$s$ by~$r_{d-2}$ gives
$$
  A_{d-1}^{(d)} + A_{d-2}^{(d)} \brs{r_{d-2} - r_{d-1}} = 
              A_{d-2}^{(d-1)}\brs{r_{d-2}-r_{d-1}} + 0.
$$
This gives 
\begin{equation}
  \begin{split}
      A_{d-2}^{(d)} & = \frac{A_{d-1}^{(d)}} {r_{d-1} - r_{d-2}} + A_{d-2}^{(d-1)} 
       \\ & = \prod_{k=0}^{d-3}(r_{d-1} - r_{k}) P_{d-1} 
               + \sum_{b=d-2}^{d-2} (r_{b} - r_{0})\ldots (r_{b} - r_{j-1}) P_{b}
       \\ & = \prod_{k=0}^{d-3}(r_{d-1} - r_{k}) P_{d-1} 
               + (r_{d-2} - r_{0}) \ldots (r_{d-2} - r_{d-3}) P_{d-2}
       \\ & = \sum_{b=d-2}^{d-1} (r_{b} - r_{0}) \ldots (r_{b} - r_{d-3}) P_{b}.
  \end{split}
\end{equation}
where the second equality follows from the previous one and the induction assumption. 

And so on. 
\end{proof}

\begin{lemma} \label{L: about convergence of mero-c f-ns}
Let~$r^* \in \mbC$ and let~$r_0^n,\ldots,r_{d-1}^n, \ n=1,2,\ldots,$
be~$d$ sequences of complex numbers such that~$r_j^n \to r^*$ 
as~$n \to \infty$ for all~$j=0,1,\ldots,d-1.$ Assume that $B_0,\ldots,B_{d-1}$ 
are bounded operators and that $A_0^n,\ldots,A_{d-1}^n, \ n=1,2,\ldots,$ 
are~$d$ sequences of bounded operators such that for all $s \in \mbC \setminus\set{r^*}$
\begin{multline*}
  \lim_{n\to \infty} \brs{\frac{A_0^n}{s-r_0^n} 
          + \frac{A_1^n}{(s-r_0^n)(s-r_1^n)} 
          + \ldots 
          + \frac{A_{d-1}^n}{(s-r_0^n)(s-r_1^n)\ldots(s-r_{d-1}^n)}}
       \\ = \frac{B_0}{s-r^*} + \frac{B_1}{(s-r^*)^2} + \ldots + \frac{B_{d-1}}{(s-r^*)^d},
\end{multline*}
where the limit is in norm and is uniform on compact subsets of~$\mbC \setminus\set{r^*}.$ 
Then for all $j=0,1,\ldots,d-1$ the sequence $A_j^n$ converges in norm to~$B_j$ as $n \to \infty.$ 
\end{lemma}
\begin{proof} 
Multiplying the equality in the lemma by
$$
  \lim_{n\to \infty} (s-r_0^n)\ldots(s-r_{d-1}^n) = (s-r^*)^d
$$
gives for all $s \neq r^*$ the equality
$$
  \lim_{n\to \infty} (A_0^n(s-r_1^n)\ldots(s-r_{d-1}^n)+\ldots + A_{d-2}^n(s-r_{d-1}^n) + A_{d-1}^n)
      = B_0(s-r^*)^{d-1} + \ldots + B_{d-1}.
$$
Since the convergence is uniform, we can differentiate the left hand side under the limit sign
(see e.g. \cite[\S IX.12]{Died}). 
Hence, applying the operator $\frac{d^d}{ds^d}$ to both sides of the last equality, we infer that $A_0^n \to B_0.$ 
Applying $\frac{d^{d-1}}{ds^{d-1}},$ we infer that $A_1^n \to B_1.$ And so on. 
\end{proof}

\begin{lemma} \label{L: Aj to bfA(j-1)} 
For each cycle~$\nu$ and for each $k=0,1,\ldots,d_\nu-1$ the limits 
$\lim_{z \to \lambda} A_k^\nu(z)$ exist and
$$
  \lim_{z \to \lambda} A_k^\nu(z) = \bfA_{\lambda}^k(r_{\lambda}) P_\lambda^{[\nu]}(r_\lambda),
$$
where the convergence is in norm.
\end{lemma}
\begin{proof} Combining Lemmas~\ref{L: a lemma} and~\ref{L: comb frac part written as almost Laurent}, 
we get 
$$
  \lim_{z \to \lambda} \sum_{j=0}^{d_\nu-1} \frac{A_j^\nu(z)}{(s-r_\nu^{(0)}(z))\ldots(s-r_\nu^{(j)}(z))}
   = 
     \frac{P^{[\nu]}_{\lambda}(r_{\lambda})}{s-r_{\lambda}} 
        + \frac{\bfA^{[\nu]}_{\lambda}(r_{\lambda})}{(s-r_{\lambda})^2} 
        + \ldots + \frac{\brs{\bfA^{[\nu]}_{\lambda}(r_{\lambda})}^{d_\nu-1}}{(s-r_{\lambda})^{d_\nu-1}},
$$
where the convergence is in norm uniformly in~$s$ on compact subsets of a deleted neighbourhood of~$r_\lambda.$
Hence, Lemma~\ref{L: about convergence of mero-c f-ns} completes the proof.
\end{proof}

\begin{prop} \label{P: p=d-1} The integer~$p$ from Puiseux series (\ref{F: Puiseux for Pz}) 
for $P_z(r_\nu^{(j)}(z))$ is equal to $d_\nu-1:$ 
  \begin{equation} \label{F: p=(d-1)}
    p = d_\nu-1.
  \end{equation}
\end{prop}
\begin{proof}
By Lemma~\ref{L: Aj to bfA(j-1)}, we have 
\begin{equation} \label{F: Ad to bfA(d-1)}
  \lim_{z \to \lambda} A_{d_\nu-1}^\nu(z) = \bfA_{\lambda}^{d_\nu-1}(r_\lambda) P_\lambda^{[\nu]}(r_\lambda) \neq 0.
\end{equation}
Here the last inequality ($\neq 0$) follows from the fact that 
the operator $\bfA_{\lambda}^{d_\nu-1}(r_\lambda)$ is reduced by the image of 
the idempotent $P_\lambda^{[\nu]}(r_\lambda)$ and that this reduction is cyclic, 
see Lemma~\ref{L: bfA[nu] is cyclic}.
By (\ref{F: Aj=sum (rj-r1)...Pj}), we have 
$$
  A^\nu_{d_\nu-1}(z) = (r_\nu^{(d_\nu-1)}(z) - r_\nu^{(0)}(z)) 
           \ldots (r_\nu^{(d_\nu-1)}(z) - r_\nu^{(d_\nu-2)}(z)) P_z(r^{(d_\nu-1)}_\nu(z)).
$$
Since by Proposition~\ref{P: first Puiseux coef is non-zero}, 
the first Puiseux series coefficient $r_{1/d_\nu}$ for $r_\nu^{(j)}(z)$ is non-zero, 
the first non-zero fractional term of the product
$$
  (r_\nu^{(d_\nu-1)}(z) - r_\nu^{(0)}(z))\ldots (r_\nu^{(d_\nu-1)}(z) - r_\nu^{(d_\nu-2)}(z))
$$ 
is 
$
  (z-\lambda)^{\frac{d_\nu-1}{d_\nu}}.
$ 
The term with the smallest power in Laurent expansion of $P_z(r_\nu^{(d_\nu-1)}(z))$ is
$
  (z-\lambda)^{-\frac p{d_\nu}}.
$ 
Hence, if $p < d_\nu-1,$ then the limit in
(\ref{F: Ad to bfA(d-1)}) would be zero, and if $p > d_\nu-1,$ then
the limit in (\ref{F: Ad to bfA(d-1)}) would diverge.
\end{proof}

\begin{prop} \label{P: what a croc}
Let $\nu = 1,\ldots,m.$
For any $k \geq 0$ the function 
$$
  \sum_{j=0}^{d_\nu-1} (r_\nu^{(j)}(z)-r_\lambda)^k P_z(r_\nu^{(j)}(z))
$$
is analytic at~$z = \lambda.$ 
Moreover, the limit of this sum as $z \to \lambda$ is equal to~$P_\lambda^{[\nu]}\bfA_\lambda^k(r_\lambda).$
\end{prop}
\begin{proof} The function is symmetric with respect to~$r_\nu^{(j)}(z), \ j=1,\ldots,d_\nu,$ 
and therefore it is single-valued in a neighbourhood of~$\lambda.$
By Proposition~\ref{P: p=d-1}, this function also cannot have whole negative powers of $z-\lambda$ 
in its power series expansion at~$\lambda.$ Hence, it is analytic at~$\lambda.$ 
In particular, if $k=0,$ then its limit as $z \to \lambda$ is equal to $P^{[\nu]}_\lambda(r_\lambda).$ 

Further, we have 
$$
  \sum_{j=0}^{d_\nu-1} (r_\nu^{(j)}(z)-r_\lambda) P_z(r_\nu^{(j)}(z)) 
     = \sum_{j=0}^{d_\nu-1} (r_\nu^{(j)}(z) - r_\nu^{(0)}(z)) P_z(r_\nu^{(j)}(z))
     + \sum_{j=0}^{d_\nu-1} (r_\nu^{(0)}(z) - r_\lambda) P_z(r_\nu^{(j)}(z)).
$$
The first summand converges to~$P_\lambda^{[\nu]}\bfA_\lambda(r_\lambda)$ 
by Lemma~\ref{L: Aj to bfA(j-1)} and (\ref{F: Aj=sum (rj-r1)...Pj})
(taken with $j=1$). The second summand converges to zero, since by Proposition~\ref{P: P[nu](z) is holomorphic}
$$
  \sum_{j=0}^{d_\nu-1} P_z(r_\nu^{(j)}(z))
$$ 
is analytic and~$r_\nu^{(0)}$ is continuous at $z = \lambda.$ 
Thus, the claim holds for $k=1.$ Using this, for any $k\geq 1$ we have 
$$
  \lim_{z\to\lambda} \sum_{j=0}^{d_\nu-1} (r_\nu^{(j)}(z) - r_\lambda)^k P_z(r_\nu^{(j)}(z)) 
    = \lim_{z\to\lambda} \brs{\sum_{j=0}^{d_\nu-1} (r_\nu^{(j)}(z) - r_\lambda) P_z(r_\nu^{(j)}(z))}^k
    = P_\lambda^{[\nu]}\bfA_\lambda^k(r_\lambda),
$$
where the first equality follows from (\ref{F: Pz(1)Pz(2)=0}).
\end{proof}

Since $p = d_\nu-1,$ we have 
$$
  \lim_{z \to \lambda} (z-\lambda)P_z(r_\nu^{(j)}(z)) = 0.
$$


\begin{prop} \label{P: Prop 6.20} For all $\nu=1,\ldots,m$ and for all $k=1,2,\ldots,d_\nu-1$ 
\begin{equation} \label{F: fcuk}
    P_\lambda^{[\nu]}(r_\lambda) \bfA_\lambda^k(r_\lambda) 
          = d_\nu \cdot \sum_{l=k}^{d_\nu-1} \brs{\sum_{m_1+\ldots+m_k=l} 
                r_{m_1/d_\nu}\ldots r_{m_k/d_\nu}} P_{-l/d_\nu},
\end{equation}
where in the sum $m_1, \ldots, m_k \geq 1.$
\end{prop}
\begin{proof}
We shall prove this only for $k=2.$ The general case is proved by the same calculation. 

We have, by Proposition~\ref{P: what a croc},
\begin{equation*}
  \begin{split}
    P_\lambda^{[\nu]}(r_\lambda)\bfA_\lambda^2(r_\lambda) & = \lim_{z \to \lambda} \sum_{j=0}^{d_\nu-1} (r_\nu^{(j)}(z)-r_\lambda)^2 P_z(r_\nu^{(j)}(z))
     \\ & = \lim_{z \to \lambda} \sum_{j=0}^{d_\nu-1} \brs{\sum_{k=1}^\infty r_{k/d_\nu}\eps_{d_\nu}^{kj}(z-\lambda)^{k/d_\nu}}
     \\ & \qquad\qquad\qquad
                    \brs{\sum_{m=1}^\infty r_{m/d_\nu}\eps_{d_\nu}^{mj} (z-\lambda)^{m/d_\nu}} 
                    \brs{\sum_{l=1}^{d_\nu-1} \eps_{d_\nu}^{-lj}(z-\lambda)^{-l/d_\nu}P_{-l/d_\nu}}
     \\ & = \lim_{z \to \lambda} \sum_{k=1}^\infty \sum_{m=1}^\infty \sum_{l=1}^{d_\nu-1} \sum_{j=0}^{d_\nu-1} 
                       \eps_{d_\nu}^{(k+m-l)j} (z-\lambda)^{(k+m-l)/d_\nu} r_{k/d_\nu} r_{m/d_\nu} P_{-l/d_\nu}.
  \end{split}
\end{equation*}
In this sum, there are finitely many terms with negative powers of $z-\lambda$ 
and the sum of positive powers of $z-\lambda$ converges absolutely in some neighbourhood of~$\lambda.$
Hence, all the interchanges of summations, which have been performed 
so far and which are about to follow, are justified. 

Let $x = k+m.$ Then
\begin{equation*}
  \begin{split}
    P_\lambda^{[\nu]}(r_\lambda)\bfA_\lambda^2(r_\lambda) 
        & = \lim_{z \to \lambda} \sum_{x=2}^\infty \sum_{m=1}^{x-1} \sum_{l=1}^{d_\nu-1} 
                               \sum_{j=0}^{d_\nu-1} \eps_{d_\nu}^{(x-l)j} (z-\lambda)^{(x-l)/d_\nu}\, r_{(x-m)/d_\nu} r_{m/d_\nu} P_{-l/d_\nu}
     \\ & = d_\nu \cdot \lim_{z \to \lambda} \sum_{x=2}^\infty \sum_{m=1}^{x-1} 
                               \sum_{l=1}^{d_\nu-1} [d_\nu|x-l] (z-\lambda)^{(x-l)/d_\nu}\, r_{(x-m)/d_\nu} r_{m/d_\nu} P_{-l/d_\nu},
  \end{split}
\end{equation*}
where in the last equality we have used (\ref{F: sum eps(jk)=[d|k]d}).
The terms with $x>l$
disappear after taking the limit~$z \to \lambda.$ 
Since $x\geq 2$ and $1\leq l \leq d_\nu-1,$ there are no non-zero terms
with $x<l$ and $d_\nu \mid x-l.$ 
Hence, the factor $[d_\nu|x-l]$ can be replaced by Kronecker's symbol $\delta_{xl}.$
This gives
\begin{equation*}
  \begin{split}
    P_\lambda^{[\nu]}(r_\lambda)\bfA_\lambda^2(r_\lambda) & = d_\nu \cdot \lim_{z \to \lambda} \sum_{x=2}^\infty \sum_{l=1}^{d_\nu-1} \delta_{x l}
                       (z-\lambda)^{(x-l)/d_\nu} \brs{\sum_{m=1}^{x-1} r_{(x-m)/d_\nu}\, r_{m/d_\nu}} P_{-l/d_\nu}
    \\ & = d_\nu \cdot \sum_{l=2}^{d_\nu-1} \brs{\sum_{m=1}^{l-1} r_{(l-m)/d_\nu} r_{m/d_\nu}} P_{-l/d_\nu}.
  \end{split}
\end{equation*}
\end{proof}

Two special cases of (\ref{F: fcuk}) are the formulas 
$$
  P_\lambda^{[\nu]}(r_\lambda)\bfA_{\lambda}(r_{\lambda}) = d_\nu \sum_{l=1}^{d_\nu-1} r_{l/d_\nu} P_{-l/d_\nu},
$$
and 
$$
  P_\lambda^{[\nu]}(r_\lambda)\bfA^{d_\nu-1}_{\lambda}(r_{\lambda}) = d_\nu \cdot r_{1/d_\nu}^{d_\nu-1} P_{\frac{1-d_\nu}{d_\nu}}.
$$

\begin{cor} 
For any $j=1,\ldots,d_\nu-1$ 
the operator~$P_\lambda^{[\nu]}(r_\lambda)\bfA_\lambda^j(r_\lambda)$ 
is a linear combination of operators 
$$
  P_{-\frac{j}{d_\nu}}, \ \ldots,  P_{-\frac{d_\nu-1}{d_\nu}}
$$
from the Puiseux expansion of $P_z(r_\nu^{(j)}(z))$ at $z = \lambda.$
\end{cor}

We summarise results of this section in the following theorem. 
\begin{thm}
Let~$r_\lambda$ be a real resonance point of the line $H_{\rlmb}+(s-\rlmb)V,$
corresponding to a point~$\lambda$ outside essential spectrum. 
Let~$N$ and~$m$ be respectively algebraic and geometric multiplicities 
of~$r_\lambda.$ 
Let $d_1, \ldots, d_m$ be the sizes of Jordan cells of the compact operator~$A_\lambda(s)$
corresponding to the eigenvalue $(s-r_\lambda)^{-1}.$ 
Let 
$
  r_z^1, \ldots, r_z^N
$
be resonance points of the group of~$r_\lambda$ corresponding to $z \approx \lambda \in \mbC.$ 
Then 
\begin{enumerate}
  \item As~$z$ makes one round about~$\lambda,$ the set of~$N$ resonance points undergoes a permutation
  which is the product of~$m$ disjoint cyclic permutations of resonance points  
  $$
    r_\nu^{(0)}(z), \ \ldots, \ r_\nu^{(d_\nu-1)}(z), \ \ \nu=1,\ldots,m,
  $$
  and, as the notation indicates, the sizes of these cyclic permutations are the same as the sizes 
  of the Jordan cells $d_1, \ldots, d_m.$
  \item For real values of~$z$ close to~$\lambda,$
  there is either one or two \emph{real} resonance points in each of these~$m$ cycles of resonance points. 
  In case there are two real resonance points in a cycle, one of them is greater than~$r_\lambda$
  and the other is smaller than~$r_\lambda;$ further, as~$z$ is shifted off the real axis, 
  these two real resonance points of a cycle shift off to different complex half-planes.
  \item The Puiseux series of the function $r_\nu^{(j)}(z)$ has the form 
  $$
    r_\nu^{(j)}(z) = \sum_{k=0}^\infty r_{k/d_\nu} e^{\frac{2\pi i kj}{d_\nu}} (z-\lambda)^{k/d_\nu}
  $$
  with $r_{1/d_\nu} \neq 0$ (where $r_0 = r_\lambda$) and all coefficients $r_{k/d_\nu}$ are real. 
  \item The Puiseux series of the idempotent $P_z(r_\nu^{(j)}(z))$ has the form 
  $$
     P_z(r_\nu^{(j)}(z))
             = \tilde P^{(j)}_\nu(z) 
                  + \sum_{l=0}^{d_\nu-1} e^{\frac{-2\pi i l j}{d_\nu}}(z-\lambda)^{-l/d_\nu}P_{-l/d_\nu}, 
  $$
  where $\tilde P^{(j)}_\nu(z)$ is continuous at $z = \lambda.$
  \item For each $\nu=1,\ldots,m$ and for all $k\geq 0,$ 
  $$
    \lim _{z \to \lambda} \sum_{j=0}^{d_\nu-1} (r_\nu^{(j)}(z)-r_\lambda)^k P_z(r_\nu^{(j)}(z)) 
          = P^{[\nu]}_\lambda(r_\lambda)\bfA_\lambda^k (r_\lambda),
  $$
  where $P^{[\nu]}_\lambda(r_\lambda)$ is an idempotent of rank~$d_\nu$ which commutes with 
 ~$\bfA_\lambda (r_\lambda)$ and which is given by 
  $$
    P^{[\nu]}_\lambda(r_\lambda) = \lim_{z \to \lambda} \sum_{j=0}^{d_\nu-1} P_z(r_\nu^{(j)}(z)).
  $$
\end{enumerate}
\end{thm}


\subsection{Relation of $\Upsilon_\lambda^{[\nu]}$ to $\phi_\nu(s)$}
\label{SS: relation of Upsilon}

The total algebraic multiplicity~$N$ of a real resonance point~$r_\lambda$ 
can be naturally split into the sum of~$m$ integers $d_1, \ldots, d_m,$ 
where~$m$ is the geometric multiplicity of $r_\lambda.$ Namely,~$d_\nu$ is the size 
of the~$\nu$th Jordan cell of the compact operator~$A_\lambda(s)$ corresponding to the eigenvalue 
$(s-r_\lambda)^{-1}.$ To a triple $(\lambda; H_{\rlmb},V)$ we can also associate~$m$ eigenpaths $\phi_\nu(s),$
$s \in \mbR,$ of $H_s = H_0+sV.$ With each of these eigenpaths $\phi_\nu(s)$ we can also associate
an integer, its order. 
Finally, there is a third way to split~$N$ into the sum $d_1+\ldots+d_m,$ where~$d_\nu$ is the length 
of the~$\nu$th cycle of resonance points. 
Theorem~\ref{T: depth of phi[nu] is d[nu]-1} shows that these three sets of integers are identical, 
so our usage of the same notation~$d_\nu$ in all these instances is justified. 
We now show that~$d_\nu$ is equal to $\dim \Upsilon_\lambda^{[\nu]}.$
One can expect that there should be a deeper connection between this vector space and the eigenpath $\phi_\nu(s).$
The following theorem demonstrates this connection. 

\begin{thm} \label{T: (i) and (ii) bfA phi(j)=phi(j-1)} 
If $\phi_\nu(s)$ is an eigenpath of order~$d_\nu,$ then

\smallskip
(i) the vectors 
$
  \phi_\nu(r_\lambda), \ \phi'_\nu(r_\lambda), \ \ldots, \ \phi^{(d_\nu-1)}_\nu(r_\lambda)
$
belong to the vector space~$\Upsilon_\lambda^{[\nu]}(r_\lambda),$ and 

\smallskip
(ii) for each $j=1,\ldots,d_\nu-1,$ \ 
$
  \bfA_\lambda(r_\lambda) \phi_\nu^{(j)}(r_\lambda) = j \phi_\nu^{(j-1)}(r_\lambda).
$
\end{thm}
\begin{proof}
This proof uses divided differences $\phi^{[j]}_\nu$ of the function~$\phi_\nu$ 
taken at resonance points of cycle~$\nu.$

(i)
The function $\phi_\nu(s)$ is a holomorphic function of $s$ in a neighbourhood of $s = \rlmb.$
For each $j=1,2,\ldots,d_\nu$ 
we consider the divided difference of order $j$ for this function evaluated at~$j$ resonance points
$r_\nu^{(0)}(z), \ldots, r_\nu^{(j-1)}(z):$ 
$$
  \phi^{[j]}_\nu(r_\nu^{(0)}(z), \ldots, r_\nu^{(j-1)}(z)).
$$
As $z \to \lambda,$ all resonance points $r_\nu^{(0)}(z), \ldots, r_\nu^{(j-1)}(z)$ 
approach $r_\lambda,$ and therefore, by a well-known property of 
high order divided differences, this function approaches the~$j$th derivative $\phi^{(j)}_\nu(r_\lambda)$
of $\phi_\nu(s)$ up to a constant factor: 
\begin{equation} \label{F: lim phi[j]=phi(j)/j!}
  \lim_{z \to \lambda} \phi^{[j]}_\nu(r_\nu^{(0)}(z), \ldots, r_\nu^{(j-1)}(z)) = \frac 1{j!}\phi_\nu^{(j)}(r_\lambda).
\end{equation}
Hence, using this equality, Proposition~\ref{P: P[nu](z) is holomorphic} and (\ref{F: P(lambda)[nu]:=}), we get
\begin{equation*}
  \begin{split}
     \lim_{z \to \lambda} 
        \sum_{k=0}^{d_\nu-1} P_z (r_\nu^{(k)}(z)) & \phi^{[j]}_\nu(r_\nu^{(0)}(z), \ldots, r_\nu^{(j-1)}(z))
     \\ & = \brs{\lim_{z \to \lambda} \sum_{k=0}^{d_\nu-1} P_z (r_\nu^{(k)}(z))} 
         \brs{\lim_{z \to \lambda}\phi^{[j]}_\nu(r_\nu^{(0)}(z), \ldots, r_\nu^{(j-1)}(z))}
     \\ & = \frac 1{j!} P_\lambda^{[\nu]}(r_\lambda) \phi^{(j)}_\nu(r_\lambda).
  \end{split}
\end{equation*}
On the other hand, since 
\begin{equation} \label{P(k)phi(l)=delta(kl)phi(l)}
  P_z (r_\nu^{(k)}(z)) \phi_\nu(r_\nu^{(l)}(z)) = \delta_{kl} \phi_\nu(r_\nu^{(l)}(z))
\end{equation}
and since $P_z (r_\nu^{(k)}(z))P_z (r_\nu^{(l)}(z)) = 0$ for $k\neq l,$ we also have 
\begin{equation*}
  \begin{split}
     \lim_{z \to \lambda} \sum_{k=0}^{d_\nu-1} P_z (r_\nu^{(k)}(z)) \phi^{[j]}_\nu(r_\nu^{(0)}(z), \ldots, r_\nu^{(j-1)}(z))
        & = \lim_{z \to \lambda} \phi^{[j]}_\nu(r_\nu^{(0)}(z), \ldots, r_\nu^{(j-1)}(z))
     \\ & = \frac 1{j!}\phi^{(j)}_\nu(r_\lambda).
  \end{split}
\end{equation*}
Hence, 
$
  P_\lambda^{[\nu]}(r_\lambda) \phi^{(j)}_\nu(r_\lambda) = \phi^{(j)}_\nu(r_\lambda)
$
and therefore $\phi^{(j)}_\nu(r_\lambda) \in \Upsilon_\lambda^{[\nu]}(r_\lambda)$.

\smallskip
(ii) This proof uses the same argument. We denote
by $M \phi$ the function $$(M \phi)(s) = (s-r_\lambda) \phi(s).$$
On one hand, using Proposition~\ref{P: what a croc} (with $k=1$) and (\ref{F: lim phi[j]=phi(j)/j!}), we have 
\begin{equation*}
  \begin{split}
   (E) := \lim_{z \to \lambda} 
        \sum_{k=0}^{d_\nu-1} (r_\nu^{(k)}(z) - r_\lambda ) 
               P_z (r_\nu^{(k)}(z)) \phi^{[j]}_\nu(r_\nu^{(0)}(z), \ldots, r_\nu^{(j-1)}(z))
     = \frac 1{j!}\bfA_\lambda^{[\nu]}(r_\lambda) \phi^{(j)}_\nu(r_\lambda ).
  \end{split}
\end{equation*}
On the other hand, using (\ref{P(k)phi(l)=delta(kl)phi(l)}) one can infer that 
\begin{equation*}
      \sum_{k=0}^{d_\nu-1} (r_\nu^{(k)}(z) - r_\lambda ) 
           P_z (r_\nu^{(k)}(z)) \phi^{[j]}_\nu(r_\nu^{(0)}(z), \ldots, r_\nu^{(j-1)}(z))
           = (M\phi_\nu)^{[j]}(r_\nu^{(0)}(z), \ldots, r_\nu^{(j-1)}(z)),
\end{equation*}
so that 
\begin{equation*}
  \begin{split}
    (E) & = \lim_{z \to \lambda} (M\phi_\nu)^{[j]}(r_\nu^{(0)}(z), \ldots, r_\nu^{(j-1)}(z)),
     \\ & = \frac 1{j!}\frac {d^j}{ds^j} (M\phi_\nu)(s)\Big|_{s=r_\lambda}
     \\ & = \frac 1{(j-1)!}\phi_\nu^{(j-1)}(r_\lambda),
  \end{split}
\end{equation*}
where the second equality follows from (\ref{F: lim phi[j]=phi(j)/j!}).
Hence, 
$
  \bfA_\lambda^{[\nu]}(r_\lambda) \phi^{(j)}_\nu(r_\lambda )
  = j \phi_\nu^{(j-1)}(r_\lambda).
$
\end{proof}

\begin{cor} \label{C: dim Upsilon[nu]=d(nu)} 
  Dimension of the vector space $\Upsilon_\lambda^{[\nu]}(r_\lambda)$ is equal to $d_\nu.$ 
\end{cor}
\begin{proof} The sum of dimensions of the vector spaces 
$\Upsilon_\lambda^{[\nu]}(r_\lambda),$ $\nu=1,\ldots,m,$ is equal to the algebraic multiplicity~$N$ of the 
resonance point $r_\lambda.$ The sum of numbers $d_\nu,$ $\nu=1,\ldots,m$ is also equal to~$N.$
Since in addition to this, by Theorem~\ref{T: (i) and (ii) bfA phi(j)=phi(j-1)}, 
we have $\dim \Upsilon_\lambda^{[\nu]}(r_\lambda) \geq d_\nu$
for all $\nu,$ it follows that $\dim \Upsilon_\lambda^{[\nu]}(r_\lambda) = d_\nu.$
\end{proof}

Now we are in position to give a second proof of Theorem~\ref{T: phi(j,mu)(0) is self-dual}. 
\begin{cor} \label{C: second proof} Let $\phi_\nu(s)$ and $\phi_\mu(s)$ be two different eigenpaths of~$H_s,$ $s \in \mbR.$
For all $j = 0,1,\ldots,d_\nu-1$ and all $k = 0,1,\ldots,d_\mu-1,$
$$
  \scal{\phi_\nu^{(j)}(\rlmb)}{V\phi_\mu^{(k)}(\rlmb)} = 0.
$$
\end{cor}
\begin{proof} By Theorem~\ref{T: (i) and (ii) bfA phi(j)=phi(j-1)}
and Corollary~\ref{C: dim Upsilon[nu]=d(nu)}, the vectors 
$
 \phi_\nu(r_\lambda), \ \phi'_\nu(r_\lambda), \ \ldots, \ \phi^{(d_\nu-1)}_\nu(r_\lambda)
$
span the vector space $\Upsilon_\lambda^{[\nu]} = \im P_\lambda^{[\nu]}.$ 
Hence, using formulas (\ref{F: VP[nu]=Q[nu]V}) and (\ref{F: P[nu]P[mu]=0}),
we have 
$$
  \scal{\phi_\nu^{(j)}(\rlmb)}{V\phi_\mu^{(k)}(\rlmb)} 
     = \scal{\phi_\nu^{(j)}(\rlmb)}{V P_\lambda^{[\mu]}\phi_\mu^{(k)}(\rlmb)}
     = \scal{P_\lambda^{[\mu]}\phi_\nu^{(j)}(\rlmb)}{V \phi_\mu^{(k)}(\rlmb)} = 0.
$$
\end{proof}

\subsection{Sign of a cycle}
Previous results show that the decomposition 
$$
  d_1 + d_2 + \ldots + d_m
$$
of algebraic multiplicity~$N$ admits a number of interpretations:
$d_\nu$ is the length of the~$\nu$th cycle, and the size of the~$\nu$th Jordan cell, 
and the largest of numbers such that $\lambda^{(d_\nu)}_\nu(r_\lambda) \neq 0.$
In this subsection we show that to each cycle one can assign a sign $\pm 1.$
This can be done in several equivalent ways.

\begin{thm} \label{T: sign of nu} For each $\nu = 1,\ldots,m$ and for all small enough $\eps>0$ and $y>0,$ signs 
of the following real numbers coincide:
\begin{enumerate}
  \item $\lambda_\nu(r_\lambda + \eps) - \lambda_\nu(r_\lambda),$
  \item $\scal{\phi_\nu(r_\lambda)}{V\phi_\nu^{(d_\nu-1)}(r_\lambda)},$
  \item $\Im r_\nu^{(0)}(z+iy),$ for all $0<y << 1,$ and for all $z \in I,$ where $I$ is one of the two 
intervals $(\lambda, \lambda+\eps)$ or $(\lambda-\eps, \lambda),$ on which the branch $r_\nu^{(0)}(z+iy)$ 
takes positive values (such an interval and such a branch exist and are unique).
\end{enumerate}
\end{thm}
\begin{proof} That signs of the first and second numbers coincide 
follows from~(\ref{F: (phi,V phi(k-1))=lambda(k)(0)(phi,phi)}).
Equality of the signs of the first and third numbers can be inferred by considering 
the four cases: $\lambda^{(d_\nu)}(r_\lambda)>0$ and $\lambda^{(d_\nu)}(r_\lambda)<0$ 
for even and odd $d_\nu.$ This comparison is straightforward and therefore is omitted. 
\end{proof}

This sign will be called the \emph{sign of a cycle}~$\nu.$

If an eigenvalue $\lambda_\nu(s)$ crosses the threshold value~$\lambda$ from one side to the other,
as~$s$ crosses~$r_\lambda$ in positive direction, then the sign of the corresponding cycle~$\nu$ 
is the contribution of the eigenvalue~$\lambda_\nu(s)$ to the spectral flow through~$\lambda.$
If an eigenvalue~$\lambda_\nu(s)$ makes a U-turn at the threshold value~$\lambda,$
then there is a dichotomous ambiguity in the way of assigning a sign to the cycle~$\nu.$
Theorem~\ref{T: sign of nu} provides one way of choosing the sign of an eigenvalue making a U-turn,
though, since such an eigenvalue does not contribute to the spectral flow, it is not essential which way to 
choose. 

\subsection{Resonance index and intersection number}
Let $\sign(\nu)$ be the sign of cycle $\nu,$ and let 
$$
  b_\nu = \left\{ \begin{matrix} 0, & \text{if} \ d_\nu \ \text{is even}, \\
                                1, & \text{if} \ d_\nu \ \text{is odd}. \end{matrix} \right.
$$
\noindent
The intersection number through a resonance point $r_\lambda$
is equal to 
\begin{equation} \label{F: inters-n number}
  \sum_{\nu=1}^m b_\nu \sign(\nu).
\end{equation}
Indeed, the value of $b_\nu$ determines whether the corresponding eigenvalue function $\lambda_\nu(s)$
makes a U-turn or not at $s = r_\lambda,$ and if it does not, 
the value of $\sign(\nu)$ shows whether the eigenvalue $\lambda_\nu(s)$ crosses 
the threshold value~$\lambda$ in positive or in negative direction. Whatever the sign of cycle~$\nu$
is, it does not contribute to the intersection number, if $b_\nu = 0.$
Further, Theorems~\ref{T: about res cycles},~\ref{T: depth of phi[nu] is d[nu]-1} 
and~\ref{T: sign of nu} imply that cycles with even~$d_\nu$ do not contribute 
to the resonance index, while cycles with odd~$d_\nu$ contribute the number $\sign(\nu).$
Hence, we proved the following 
\begin{thm} \label{T: intersection number = TRI} 
The intersection number (\ref{F: inters-n number}) of eigenvalues 
of a path $H_r, \ r \in [0,1],$ through~$\lambda$ is equal to the total resonance index. 
\end{thm}
Since each cycle~$\nu$ contributes one of the three numbers $\pm 1$ or $0$ to the TRI,
it follows that 
\begin{equation} \label{F: U-turn inequality}
  \abs{\ind_{res}(\lambda; H_{r_\lambda},V)} \leq m.
\end{equation}
This is the U-turn inequality which holds for a.e.~$\lambda$ inside 
essential spectrum too, \cite[Theorem 10.1.6]{Az9}. 

\subsection{A representation of~$P_\lambda(r_\lambda)$}

Let $T$ be an operator of rank $N < \infty.$ If $b_1, \ldots, b_N$ is a basis of $\im(T)$ 
and if $\alpha$ is an invertible $N\times N$ matrix, then there exists a unique basis $a_1, \ldots, a_N$
of $\im(T^*),$ such that 
\begin{equation} \label{F: matrix repr-n for a finite-rank op-r}
  T = \sum_{i=1}^N \sum_{j=1}^N \alpha_{ij} \scal{a_i}{\cdot} b_j.
\end{equation}
Also, if $(a_i)$ and $(b_i)$ are bases of $\im(T^*)$ and $\im(T)$ respectively, 
then there exists a unique invertible matrix $\alpha$ such that the equality 
(\ref{F: matrix repr-n for a finite-rank op-r}) holds.
In the case of a finite-rank operator~$P_\lambda(r_\lambda),$ there exists one natural Jordan basis 
$$
  \frac 1{j!}\phi_\nu^{(j)}(r_\lambda), \ \nu=1,\ldots,m, \ j=0,1,\ldots,d_\nu-1,
$$
of the vector space $\im(P_\lambda(r_\lambda)) = \Upsilon_\lambda(r_\lambda),$ 
provided by part (i) of Theorem~\ref{T: (i) and (ii) bfA phi(j)=phi(j-1)} and Corollary~\ref{C: dim Upsilon[nu]=d(nu)}. 
Since by (\ref{F: VP = QV}) and (\ref{F: P*=Q})
$\im(P^*_\lambda(r_\lambda)) = V \im(P_\lambda(r_\lambda)),$ we also have a natural basis 
$$
  \frac 1{j!}V\phi_\nu^{(j)}(r_\lambda), \ \nu=1,\ldots,m, \ j=0,1,\ldots,d_\nu-1,
$$
of the vector space $\im(P^*_\lambda(r_\lambda)) = \Psi_\lambda(r_\lambda).$ 
Hence, there exists a unique invertible $N\times N$ matrix~$\alpha,$ such that 
\begin{equation} \label{F: P = sum sum ...}
  P_\lambda(r_\lambda) = \sum_{\mu=1}^m\sum_{\nu=1}^m \sum_{k=0}^{d_\mu-1} \sum_{j=0}^{d_\nu-1}
      \frac 1{k!j!} \alpha_{\mu\nu}^{kj} \scal{V \phi_\mu^{(k)}(r_\lambda)}{\cdot} \phi_\nu^{(j)}(r_\lambda).
\end{equation}
Since~$P_\lambda(r_\lambda)$ is an idempotent, we have 
$P_\lambda(r_\lambda)\phi_\nu^{(j)}(r_\lambda)= \phi_\nu^{(j)}(r_\lambda).$ 
Therefore, the matrix $\alpha$ is the transpose-inverse of the $N\times N$ matrix 
$$
  (b_{\mu\nu}^{kj}) := \brs{\frac 1{k!j!}\scal{V \phi_\mu^{(k)}(r_\lambda)}{\phi_\nu^{(j)}(r_\lambda)}}.
$$
Previous results imply that the numbers $b_{\mu\nu}^{kj}$ are real. 
By Corollary~\ref{C: second proof}, this matrix has the following property:
\begin{equation} \label{F: matrix b}
  b_{\mu\nu}^{kj} = \delta_{\mu\nu}b_{\mu\nu}^{kj},
\end{equation}
Hence, the matrix $b$ is a direct sum of~$m$ matrices $b_\nu^{kj}$ of size $d_\nu\times d_\nu.$ 

Further, by part (ii) of Theorem~\ref{T: (i) and (ii) bfA phi(j)=phi(j-1)}, the scalar product 
$$
  \frac 1{k!j!} \scal{V \phi_\nu^{(k)}(r_\lambda)}{\phi_\nu^{(j)}(r_\lambda)}
$$
depends only on $k+j,$ and if $k+j\leq d_\nu-2,$ then it is zero. 
Hence, the matrix $(b_\nu^{kj})$ is a Hankel matrix with zeros above the skew-diagonal. 
It follows that the inverse of $(b_\nu^{kj})$ is a Hankel matrix with zeros below the skew-diagonal. 
Thus, we have proved the following theorem.
\begin{thm} \label{T: P = sum sum} 
The idempotent operator~$P_\lambda(r_\lambda)$ can be written in the form 
(\ref{F: P = sum sum ...}), where the $N \times N$ matrix $\alpha$ is a direct 
sum of self-adjoint skew-upper triangular Hankel matrices of sizes $d_1, \ldots, d_m.$ 
\end{thm}

\subsection{Signature of the resonance matrix $V P_\lambda(r_\lambda)$}
In this subsection we prove the equality
\begin{equation} \label{F: ind res=sign VP}
  \ind_{res}(\lambda; H_{r_\lambda},V) = \sign V P_\lambda(r_\lambda).
\end{equation}
This equality was proved in \cite{Az9} in a more general setting of~$\lambda$ from the essential spectrum. 
Here we give a new proof, which easily follows from previous results. 
The proof is based on the following two well-known lemmas. 

\begin{lemma} \label{L: L1} The signature of a self-adjoint skew-upper triangular $d\times d$ Hankel matrix 
is zero, if~$d$ is even, and is equal to the sign of the skew-diagonal entry, if $d$ is odd. 
\end{lemma}

\begin{lemma} \label{L: L2} If $b_1, \ldots, b_d$ are linearly independent vectors in a Hilbert space
and $\alpha$ is a $d\times d$ self-adjoint matrix then the signature of the self-adjoint operator 
$$
  T = \sum_{i=1}^d \sum_{j=1}^d \alpha_{ij} \scal{b_i}{\cdot} b_j
$$
is equal to the signature of $\alpha.$ 
\end{lemma}


\begin{thm} \label{T: ind res=sign VP} 
The equality (\ref{F: ind res=sign VP}) holds. 
\end{thm}
\begin{proof} 
For each $\nu=1,\ldots,m,$ signs of the main skew-diagonal entries of the $d_\nu\times d_\nu$ 
Hankel matrices (\ref{F: matrix b}) and 
$\brs{\alpha_{\mu\nu}^{kj}}$ from (\ref{F: P = sum sum ...}) are equal. 
Hence, by Lemma~\ref{L: L1} and Theorems~\ref{T: sign of nu} and~\ref{T: intersection number = TRI},
the total resonance index is equal to the signature of the matrix $\brs{\alpha_{\mu\nu}^{kj}}.$

We have from (\ref{F: P = sum sum ...})
$$
  V P_\lambda(r_\lambda) = \sum_{\mu,k}\sum_{\nu,j} \frac 1{k!j!} 
      \alpha_{\mu\nu}^{kj} \scal{V \phi_\mu^{(k)}(r_\lambda)}{\cdot} V\phi_\nu^{(j)}(r_\lambda).
$$
By Lemma~\ref{L: L2}, signature of this operator is equal to the signature of the matrix~$\alpha.$

\end{proof}

\section{On stability of resonance index}
\label{S: On stability of res-index}

In this section we study the behaviour of the resonance index $\ind_{res}(\lambda; H_{r_\lambda},V)$ 
as a function of the~perturbation~$V.$

\subsection{Topology of the vector space of directions}

In previous sections 
we worked with a fixed direction~$V$ and for this reason there was no 
need in having some topology in the real vector space of directions~$\clA_0.$
Now we are going to consider stability of resonance index $\ind_{res}(\lambda; H_{\rlmb},V)$ 
with respect to small perturbations of~$V$ and therefore we need to discuss the topology of~$\clA_0.$

The approach taken is to impose conditions on the topology 
of the affine space~$\clA$ which allow to prove stability results. 
The conditions imposed on the topology of~$\clA$ hold trivially if the vector space of directions~$\clA_0$ 
consists of bounded self-adjoint operators
and if the topology of~$\clA_0$ is the norm topology or stronger. 

\smallskip
\begin{assump} \label{A: norm of clA(0)}
The real vector space~$\clA_0$ is endowed with a norm $\norm{\cdot}_{\clA_0}$ such that 
for some non-real~$z$ and for some $H_0 \in \clA$ the following three conditions hold: 
\begin{enumerate}
  \item[(VR)] the function~$\clA_0 \ni V \mapsto V R_z(H_0) \in \clB(\hilb)$ is continuous. 
  \item[(VRV1)] the product $V_1 R_z(H_0)V_2$ is compact. 
  \item[(VRV2)] the product $V_1 R_z(H_0)V_2$ is a continuous function of~$V_1$ and~$V_2.$ 
\end{enumerate}
\end{assump}

Since Assumption \ref{A: Main Assumption}(2) implies compactness of $V_1 \Im R_z(H_0)V_2,$
the condition (VRV1) is equivalent to compactness of $V_1 \Re R_z(H_0)V_2.$

Topology in~$\clA_0$ induces a topology in the affine space~$\clA.$ We assume that $\clA$ 
is endowed with this topology. 

\begin{lemma} \label{L: VR} 
\begin{enumerate}
  \item[(i)] For any $V \in \clA_0$ and any $H \in \clA$ \ $\lim_{y \to \infty} \norm{V R_{\lambda+iy}(H)} = 0.$
  
  \noindent Further, if~$\clA_0$ is endowed with a norm which satisfies property~(VR), then 
  \item[(ii)] \label{item (B)} for any non-real~$z$ and for any $H \in \clA$ 
  the operator~$V R_z(H)$ jointly continuously depends 
  on~$z \in \mbC\setminus \mbR,$~$V \in \clA_0$ and~$H \in \clA,$ and 
  \item[(iii)] \label{item (A)} for any non-real~$z$ the operator~$R_z(H)$ 
    jointly continuously depends on~$z \in \mbC\setminus \mbR$ and on~$H \in \clA.$
\end{enumerate}
\end{lemma}
\begin{proof} (i) We have 
$V R_{\lambda+iy}(H) = V R_{\lambda+i}(H) \frac{H-\lambda-i}{H-\lambda-iy}.$
The operator $V R_{\lambda+i}(H)$ is compact, and the operator 
$\frac{H-\lambda-i}{H-\lambda-iy}$ converges to zero in $*$-strong topology as $y \to \infty.$ 
Hence, \cite[Lemma 6.1.3]{Ya} completes proof. 

(ii) The first resolvent identity combined with (VR) implies that $VR_z(H_0)$ 
depends jointly continuously on~$z$ and~$V.$
Further, by the second resolvent identity we have
\begin{equation} \label{F: VR(H)=VR(H0)(...)(-1)}
  V R_z(H) = V R_z(H_0) (1+(H-H_0)R_z(H_0))^{-1}.
\end{equation}
For non-real $z$ the operator $1+(H-H_0)R_z(H_0)$ is invertible. 
Hence, this equality shows that the operator $VR_z(H)$ depends continuously on~$z,$ $V$ and~$H.$
The item (iii) is proved by the same argument. 
\end{proof}

\smallskip
\begin{lemma} \label{L: VRV}
If~$\clA_0$ is endowed with a norm which satisfies Assumption~\ref{A: norm of clA(0)}, then 
the operator $V_1 R_z(H) V_2$ is compact for any non-real $z$ and any $H\in\clA,$
and it jointly continuously depends on $z \in \mbC\setminus\mbR,$ $V_1, V_2 \in \clA_0$ and $H \in \clA.$ 
\end{lemma}
\begin{proof} 
By the first resolvent identity we have 
\begin{equation} \label{F: good}
  V_1 R_z(H_0)V_2 - V_1 R_w(H_0)V_2 = (z-w) V_1 R_z(H_0)\,R_w(H_0)V_2.
\end{equation}
Since by Assumption~\ref{A: Main Assumption} the right hand side is compact, the operator $V_1 R_z(H_0)V_2$ is compact 
for any non-real value of~$z,$ provided it is compact for some value of~$z.$

For $H = H_0 + V_3$ we have 
\begin{equation} \label{F: Neumann for V1 R(H) V2}
  \begin{split}
     V_1 R_z(H)V_2 & = V_1 R_z(H_0) (1+V_3 R_z(H_0))^{-1} V_2
     \\ & = V_1 R_z(H_0) \sum_{k\geq 0} (-1)^k (V_3 R_z(H_0))^k V_2 
     \\ & = V_1 R_z(H_0)V_2 + \sum_{k\geq 1} (-1)^k (V_3 R_z(H_0))^{k-1}  \cdot \SqBrs{V_3 R_z(H_0)V_2}.
  \end{split}
\end{equation}
For large $y = \Im z$ convergence of the series and compactness of this operator follow 
from Lemma~\ref{L: VR}(i) and (VRV1). 
The first summand and the product in the pair of square brackets are continuous by the assumption (VRV2).
Since for large enough~$y$ the geometric series converges uniformly, 
for such $y$ the last series depends continuously on~$V_3$ by assumption (VR). 
For other values of $y$ the claim can now be inferred from (\ref{F: good}).
\end{proof}

Usually we denote a resonance point of a triple $(\lambda; H_0,V),$ where~$H_0$
is a $\lambda$-regular operator, by~$r_\lambda.$ But in this section for convenience we 
assume that~$r_\lambda = 0,$ so that the operator~$H_0$ itself is $\lambda$-resonant. 

\begin{lemma} \label{L: regular dir-s are open} 
If~$\clA_0$ is endowed with a norm which satisfies Assumption~\ref{A: norm of clA(0)}
then the sets of all (a) regular and (b) simple directions at a resonance point~$H_{0}$
are open in the norm of~$\clA_0.$ 
\end{lemma}
\begin{proof} (a) Let~$V$ be a regular direction. 
By continuity of the mapping $V \mapsto R_z(H)V$ (Lemma~\ref{L: VR}(ii)), there exists a neighbourhood $O_V$ 
of~$V$ in~$\clA_0,$ such that for all $W \in O_V$ 
\begin{equation} \label{F: norm(R(V-W))<1}
  \norm{R_\lambda(H_{0}+V) (V-W)} < 1.
\end{equation}
Hence, by the second resolvent identity, for all $W \in O_V$ 
$$
  R_\lambda(H_{0}+W) = (1 + R_\lambda(H_{0}+V) (V-W))^{-1} R_\lambda(H_{0}+V),
$$
where the inverse exists due to (\ref{F: norm(R(V-W))<1}).
It follows that all directions from $O_V$ are regular. 

(b) 
In this proof we use this characterisation of simple directions: 
a regular direction~$V$ at a resonance point~$H_{0}$
is simple if and only if the algebraic multiplicity~$N$ of the eigenvalue $s^{-1}$ 
of the compact operator $R_{\lambda}(H_0+sV)V$ is equal to the geometric 
multiplicity~$m$ of that eigenvalue.

Since a simple direction is regular, by part (a) there exists a neighbourhood $O_V$ of 
a simple direction~$V$ such that all directions~$W$ from $O_V$ are regular. 
By Lemma~\ref{L: VR}(ii), the operator $R_\lambda(H_{0}+W)W$ depends 
continuously on~$W \in \clA_0$ in some neighbourhood of~$V.$
Further, since~$V$ is simple, the operator $R_\lambda(H_0+sV)V$ has~$s^{-1}$  as an eigenvalue 
of algebraic and geometric multiplicity~$m.$ Perturbation of the direction~$V$ may not change the geometric 
multiplicity of the eigenvalue~$s^{-1},$ 
since this number is the dimension of the eigenspace $\clV_\lambda$ of~$H_{0}$ 
and thus it depends only on~$H_{0}.$ Hence, a perturbation of~$V$ may not decrease 
the algebraic multiplicity of~$s^{-1},$ but it may increase it.

The operator $R_\lambda(H_{0}+sW)W$ also has~$s^{-1}$ 
as an eigenvalue of geometric multiplicity~$m.$ 
Since the operator $R_\lambda(H_{0}+s W)W$ is close to $R_\lambda(H_{0}+s V)V$ in the 
operator norm for all~$W$ close enough to~$V$ in the norm of~$\clA_0,$ 
it follows that there exists a neighbourhood~$\tilde O_V$ of~$V$ such that 
for all~$W$ from~$\tilde O_V$ the algebraic multiplicity~$m$ of~$s^{-1}$ does not increase 
for $R_\lambda(H_{0}+s W)W.$
Thus, all directions from $\tilde O_V$ are simple. 
\end{proof}

\subsection{Continuous dependence of $P_\lambda$ and $V P_\lambda$ on simple directions}

\begin{lemma} \label{L: P(H0,V) is cont-s} 
Let $H_{0}$ be a resonant point and let~$V$ be a regular direction
of order~$1.$ Then the idempotent $P_\lambda(H_{0},V)$ depends on the direction~$V$ continuously,
that is, for any $\eps>0$ there exists $\delta > 0$ such that if~$W$ is a regular direction 
with $\norm{V-W}_{\clA_0} < \delta$ then 
\begin{equation} \label{F: norm(P(V)-P(W))<eps}
  \norm{P_\lambda(H_{0},V) - P_\lambda(H_{0},W)} < \eps.
\end{equation}
\end{lemma}
\begin{proof} Let~$m$ be the geometric multiplicity of the resonance point $H_{0}.$ 
By definition of the idempotent~$P_\lambda,$ we have  
$$
  P_\lambda(H_{0},V) = \frac 1 {2\pi i} \oint_C R_\lambda(H_0+sV)V\,ds,
$$
where the contour $C$ encloses only the resonance point $s = 0$ of the path $H_0 + sV,$
and this resonance point has both geometric and algebraic multiplicity~$m,$ since~$V$ is a simple 
direction.
By upper semi-continuity of spectrum we can choose $\delta > 0$ small enough so that, 
for all~$W$ with $\norm{V-W}_{\clA_0} < \delta,$ inside the contour $C$ there will 
be only one resonance point $s = 0$ of the path $H_{0}+s W,$ and this resonance point 
will have geometric and algebraic multiplicities both equal to~$m$ 
(indeed, the total algebraic multiplicity of all resonance points inside $C$ is at least~$m$
since $H_{0}$ has geometric multiplicity~$m$ and the total algebraic multiplicity is at most~$m$
due to upper semi-continuity of spectrum).
Now, compactness of the contour~$C$ and joint continuity of~$R_\lambda(H_0+sV)W$ 
(Lemma~\ref{L: VR}(ii)) imply that there exists a possibly smaller $\delta>0,$ if necessary, such that 
(\ref{F: norm(P(V)-P(W))<eps}) holds as long as $\norm{V-W}_{\clA_0} < \delta.$
\end{proof}

For directions of order greater than~$1$ this proof does not work, since in this case the algebraic multiplicity
of the resonance point $s = 0$ is greater than~$m,$ and as a consequence of this,
while for the perturbed path $H_{0} + s W$ the point $s = 0$ will still have geometric 
multiplicity~$m,$ other resonance points can appear inside 
the contour~$C$ which could have split from the resonance point $s = 0.$ 

\begin{thm} \label{T: res matrix is V-stable}
Let~$H_{0}$ be a resonance point and let~$V$ be a regular direction
of order~$1.$ Then the resonance matrix $V P_\lambda(H_{0},V)$ depends on the direction~$V$ continuously,
that is, for any $\eps>0$ there exists $\delta > 0$ such that if~$W$ is a regular direction 
such that $\norm{V-W}_{\clA_0} < \delta$ then 
$$
  \norm{V P_\lambda(H_{0},V) - W P_\lambda(H_{0},W)} < \eps.
$$
\end{thm}
\begin{proof} 
The mapping $V \mapsto V R_\lambda(H_{0}+sV)V$ is continuous by Lemma~\ref{L: VRV}.
Since the direction~$V$ is simple, it has a neighbourhood consisting of simple directions.
Hence, inside a small enough contour~$C,$ enclosing the resonance point~$s=0,$
for all~$W$ from the neighbourhood there will be no other resonance points of the triple 
$(\lambda; H_{0},W).$ Hence, the formula
$$
  V P_\lambda(H_{0},V) = \frac 1 {2\pi i} \oint_C V R_\lambda(H_{0}+s V)V\,ds
$$
completes the proof. 
\end{proof}

\subsection{Homotopy stability of total resonance index}

We recall some definitions from previous sections.
A point (that is, a self-adjoint operator)~$H_{0}$ of the affine
space~$\clA$ is \emph{resonant} if a fixed real number~$\lambda$ which does not
belong to the common essential spectrum $\sigma_{ess}$ of operators from~$\clA$ is an eigenvalue of~$H_{0}.$
A resonant point~$H_{0}$ is \emph{simple}, if the eigenvalue~$\lambda$ has multiplicity one.

A regular direction~$V$ is \emph{simple} at a resonance point~$H_{0}$ if~$V$
is not tangent to~$\Rset$ at~$H_{0};$ by this we mean that~$V$ is not the tangent vector
of any differentiable path in~$\Rset$ which passes through~$H_{0}.$

\begin{thm} \label{T: stability of res index for order 1 V}
Let~$H_{0}$ be a resonance point and let~$V$ be a simple direction.
Resonance index $\ind_{res}(\lambda; H_{0},V)$ is stable under small perturbations of~$V$
within any finite-dimensional subspace of~$\clA_0.$ 
\end{thm}
\begin{proof} 
Since order of~$V$ is equal to~$1,$ the rank~$N$ of the resonance matrix $VP_\lambda(H_{0},V)$
is equal to~$m,$ 
and, by Theorem~\ref{T: V simple iff transversal}, the direction~$V$ is transversal.
Hence, there exists a small enough convex neighbourhood of~$V$ in the finite-dimensional subspace 
such that all directions from that neighbourhood are also transversal, and therefore, by 
Theorem~\ref{T: V simple iff transversal}, have order $1.$ Hence, the ranks of the resonance matrices 
$W P_\lambda(H_{0},W)$ for all directions~$W$ from the neighbourhood are equal to~$m.$ 
Since by Theorem~\ref{T: res matrix is V-stable} the resonance matrix depends continuously on~$V$ 
for simple directions $V,$ it follows that signature of the resonance matrix 
$VP_\lambda(H_{0},V)$ is stable under small perturbations of~$V.$
\end{proof}

While the resonance index of a direction is stable if the direction is simple, 
in general this is not true. 
Geometrically, the reason is that a tangent direction 
may cross the resonance set however small a perturbation of that direction is. 
It leads to a sudden change of the intersection number of that direction. 
This connection of resonance index with intersection number of eigenvalues 
was discussed earlier. 
Analytically, the reason for the instability of the resonance index is that 
the resonance point may split into two or more resonance points as a non-simple 
direction~$V$ is perturbed. 
In other words, as a non-simple direction~$V$ is perturbed to a close direction $W,$ 
near a resonance point $s = 0$ there may appear other resonance points, 
which may ``take away'' part of the resonance index.
Thus, while the resonance index is not stable, the total resonance index is. 
In this subsection we prove the corresponding theorems. 

\begin{lemma} \label{L: if r then so is bar r} 
If a non-real complex number~$r_\lambda^j$ is a resonance point of the triple $(\lambda; H_{0},V)$
then the conjugate of~$r_\lambda^j$ is also a resonance point of the triple $(\lambda; H_{0},V)$
and moreover it has the same algebraic multiplicity. 
\end{lemma}
This lemma is \cite[Corollary 3.1.5]{Az9}.

\smallskip
Let $H_{0}$ be a resonance point and let~$V$ be a regular direction. 
With every pair $(H_{0},V)$ we can associate the set of resonance points of the pair.
If the direction~$V$ is slightly perturbed 
and if a resonance point is degenerate then it can split. The resonance point $s = 0$ itself
will not move when~$V$ is changed, since $H_{0}$ is resonant, but some other resonance points 
may break away from $s = 0,$ if the direction~$V$ is not simple. The following figure shows one 
of the possible scenarios. 

\bigskip

\hskip 2.8 cm 
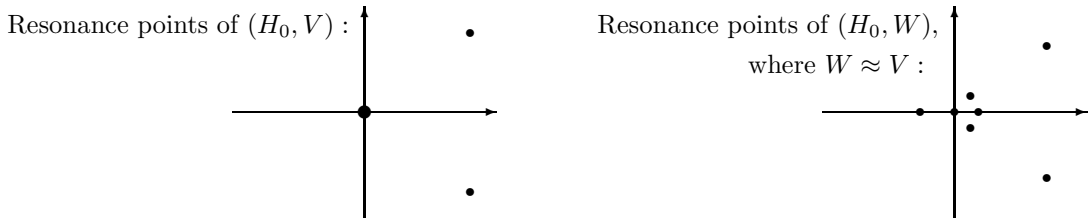
\begin{picture}(100,80)
\put(-85,70){\small Resonance points of $(H_{0},V):$}
\put(0,40){\vector(1,0){100}}
\put(50,0){\vector(0,1){80}}
\put(50,40){\circle*{5}}
\put(90,70){\circle*{3}}
\put(90,10){\circle*{3}}
\end{picture}
\hskip 4.2 cm 
\begin{picture}(100,80)
\put(-85,70){\small Resonance points of $(H_{0},W),$}
\put(-28,55){\small where $W\approx V:$}
\put(0,40){\vector(1,0){100}}
\put(50,0){\vector(0,1){80}}
\put(50,40){\circle*{3}}
\put(37,40){\circle*{3}}
\put(59,40){\circle*{3}}
\put(56,46){\circle*{3}}
\put(56,34){\circle*{3}}
\put(85,65){\circle*{3}}
\put(85,15){\circle*{3}}
\end{picture}

\begin{thm} \label{T: total res. index for (H0,W)} 
Let~$V$ be a regular direction at a resonance point~$H_{0}.$
Let~$W$ be a small perturbation of~$V$ and let 
$r^1_\lambda(H_{0},W), \ r^2_\lambda(H_{0},W), \ \ldots$
be resonance points of the triple $(\lambda; H_{0},W)$ which 
belong to the group of the resonance point $s = 0$ of the triple $(\lambda; H_0,V),$
where $H_r = H_{0} + r V.$
Then the resonance index 
$\ind_{res} (\lambda; H_{0},V)$
is equal to the sum of resonance indices
$$
  \sum_{j} \ind_{res} (\lambda; H_{r_\lambda^j},W),
$$
where the sum is taken over real resonance points of the group of $s = 0.$
\end{thm}
\begin{proof} The resonance index 
$
  \ind_{res} (\lambda; H_{0},V)
$
is equal to the difference $N_+ - N_-,$ where $N_\pm$ is the number of resonance points 
of the triple $(\lambda + iy; H_{0},V)$ for small enough~$y$ which belong to the group of~$s = 0$ 
and lie in~$\mbC_\pm.$ 
If the direction~$V$ is deformed to~$W,$ the resonance point $s = 0$ of the pair $(H_{0},V)$ 
will in general split to some number of resonance points including the original resonance point $s = 0;$
the algebraic multiplicity of this resonance point may decrease but the geometric multiplicity will stay the same. 
We shall also refer to these resonance points of the pair $(H_{0},W)$ as resonance points of the group of $s = 0.$ 
Some of these resonance points can be real and some can be non-real.

If~$\lambda$ is perturbed slightly to~$\lambda+iy$ with small positive~$y$ the non-real resonance points 
of the pair $(H_{0},W)$ which belong to the group of $s = 0$ will stay in the same half-plane, and 
the real resonance points of the group of $s = 0$ 
$$
  r^1_\lambda(H_{0},W), \ r^2_\lambda(H_{0},W), \ \ldots
$$
will shift from the real axis, thus giving a sum of resonance indices for the pair $(H_{0},W).$ 
We have to show that this sum is equal to $N_+ - N_-.$ 

Let $M_+$ (respectively, $M_-$) be the number of resonance points of the pair $(H_{0},W)$ 
which belong to the group of $s = 0$ and which appear in the upper 
(respectively, lower) half plane, as~$\lambda$ is shifted to~$\lambda+iy,$ with small $y>0.$
The difference $M_+ - M_-$ is equal to the resonance index of the triple $(\lambda; H_{0},V),$ 
since we can deform~$V$ to~$W$ with $y>0$ fixed, and as we do so resonance points of~$V$ will get deformed 
to resonance points of~$W$ without crossing~$\mbR.$
Hence, equality of the total resonance index of the pair $(H_{0},W)$ to the resonance index of $(H_{0},V)$
follows from Lemma~\ref{L: if r then so is bar r} according to which the numbers of non-real resonance points 
of the group of $s = 0$ in both half-planes are the same.
\end{proof}

\begin{thm} \label{T: total res. index for (H1,V)} 
Let~$V$ be a regular direction at a resonance point~$H_{0}.$
Let~$H_{0}'$ be a small perturbation of~$H_{0}$ and let 
$r^1_\lambda(H'_{0},V), \ r^2_\lambda(H'_{0},V), \ \ldots$
be resonance points of the triple $(\lambda; H'_{0},V)$ which belong to the group of the resonance
point $s = 0$ of the triple $(\lambda; H'_0,V),$ where $H'_r = H'_{0} + r V.$ 
Then the resonance index 
$\ind_{res} (\lambda; H_{0},V)$
is equal to the sum of resonance indices
$$
  \sum_{j} \ind_{res} (\lambda; H'_{r_\lambda^j},V),
$$
where the sum is over real resonance points of the group of $s = 0.$
\end{thm}
Proof of this theorem follows almost verbatim the argument used 
in the proof of Theorem~\ref{T: total res. index for (H0,W)},
and therefore is omitted. The only difference is that the resonance point $s = 0$ 
itself may not only split but also shift. 

\smallskip

\hskip 2.8 cm 
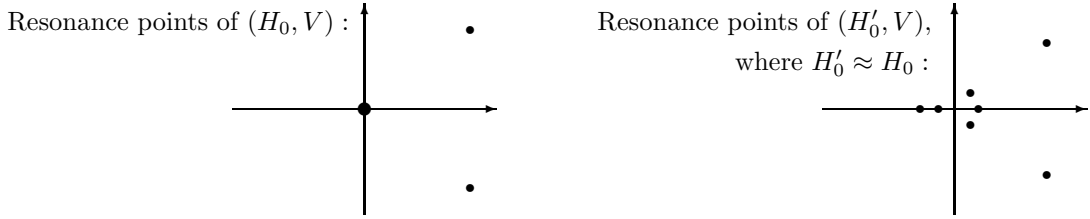
\begin{picture}(100,80)
\put(-85,70){\small Resonance points of $(H_{0},V):$}
\put(0,40){\vector(1,0){100}}
\put(50,0){\vector(0,1){80}}
\put(50,40){\circle*{5}}
\put(90,70){\circle*{3}}
\put(90,10){\circle*{3}}
\end{picture}
\hskip 4.2 cm 
\begin{picture}(100,80)
\put(-85,70){\small Resonance points of $(H'_{0},V),$}
\put(-28,55){\small \!\!\!where $H'_{0}\approx H_{0}:$}
\put(0,40){\vector(1,0){100}}
\put(50,0){\vector(0,1){80}}
\put(44,40){\circle*{3}}
\put(37,40){\circle*{3}}
\put(59,40){\circle*{3}}
\put(56,46){\circle*{3}}
\put(56,34){\circle*{3}}
\put(85,65){\circle*{3}}
\put(85,15){\circle*{3}}
\end{picture}


\begin{thm} \label{T: homotopy stability of res index} Let $H_0, H_1$ be two operators from~$\clA$ such that
$H_0$ and $H_1$ are not resonant at~$\lambda \notin \sigma_{ess}.$ Then there exist neighbourhoods
$\euU_0$ and $\euU_1$ of~$H_0$ and $H_1$ respectively such that for all $H'_0 \in \euU_0$
and all $H'_1 \in \euU_1$
$$
  \sum_{r \in [0,1]} \ind_{res}(\lambda; H_r,V) = \sum_{r \in [0,1]} \ind_{res}(\lambda; H'_r,V'),
$$
where~$V' = H'_1 - H'_0$ and $H'_r = H'_0+rV'.$
\end{thm}
\begin{proof}
This theorem follows immediately from Theorems~\ref{T: total res. index for (H0,W)}
and~\ref{T: total res. index for (H1,V)}.
\end{proof}

\subsection{Robbin-Salamon axioms for spectral flow and resonance index}

In \cite{RoSa} it was shown that for operators with compact resolvent
the spectral flow of a path of operators can be uniquely characterised as a mapping
which satisfies five axioms: Homotopy, Constancy, Catenation, Direct Sum, and Normalisation.
In this subsection we show that the total resonance index satisfies these Robbin-Salamon axioms,
following closely their paper~\cite{RoSa}.

Assume Assumptions~\ref{A: Main Assumption} and~\ref{A: norm of clA(0)}. 

For any real numbers $a$ and $b,$ $a<b,$ and any two operators $H_a$ and $H_b$ from~$\clA$
let $PC^1([a,b], H_a,H_b)$ be the set of all continuous piecewise-$C^1$ paths $\set{H_s, s \in [a,b]},$ 
of operators from~$\clA$ such that (1)
$\lambda$ does not belong to the spectrum of $H_a$ and $H_b,$ (2) 
$\lambda$ does not belong to the spectrum of the vertex points 
of the $C^1$-subpaths of~$H_s,$
and (3) all operators $H_s, s \in [a,b],$ belong to a finite dimensional subspace of~$\clA.$ 
The conditions (2) and (3) are not necessary, but they will allow
to avoid unnecessary technical complications. 

By definition, the total resonance index of a path $\set{H_s, s \in [a,b]}$ from $PC^1[a,b]$
is the total resonance index of a continuous piecewise linear path $K_s$ from $PC^1[a,b]$ which has 
the same end-points $H_a$ and $H_b$ and such that $\sup_{s \in [a,b]}\norm{H_s - K_s}_{\clA_0}$ is small enough. 
Using a standard compactness argument and the homotopy invariance of the total resonance index,
one can show that this definition is correct. 

\smallskip
Let~$\mu$ be a mapping
$PC^1([a,b], H_a, H_b) \to \mbZ,$
which satisfies the following Robbin-Salamon axioms for spectral flow:

\begin{enumerate}
  \item[(i)] (Homotopy) \ If paths~$f$ and~$g$ from $PC^1([a,b], H_a,H_b)$ are homotopic, then $\mu(f) = \mu(g).$
  \item[(ii)] (Constancy) \ If a path~$f$ from $PC^1([a,b], H_a,H_b)$ is constant then $\mu(f) = 0.$
  \item[(iii)] (Catenation) \ If $f \in PC^1([a,b], H_a,H_b)$ and $g \in PC^1([b,c], H_b,H_c)$ 
    and $f(b) = g(b),$ then
  $
    \mu(f \catn g) = \mu(f) + \mu(g),
  $
  where $\catn$ denotes concatenation of paths. 
  \item[(iv)] \ (Direct Sum) \ If $f \in PC^1([a,b], H_a,H_b)$ and if $g \in PC^1([a,b], \tilde H_a,\tilde H_b),$
  where $\tilde H_a, \ \tilde H_b$ are operators from another affine space $\tilde \clA,$ satisfying the above conditions,
  then
  $\mu(f \oplus g) = \mu(f) + \mu(g).$
  \item[(v)] \ (Normalisation) \ Let $\hilb = \mbC$
  be a one-dimensional Hilbert space, let $a < \lambda < b$ and let 
  $f(t) = t,$ $f \in PC^1([a,b], a,b).$ Then
  $
    \mu(f) = 1.
  $
\end{enumerate}

\begin{thm} \label{T: Tot res index meets R-S axioms} 
Total resonance index satisfies all five Robbin-Salamon axioms.
\end{thm}
\begin{proof} All axioms except Homotopy Axiom are trivially satisfied by the total resonance index.
The Homotopy Axiom is satisfied by Theorem~\ref{T: homotopy stability of res index}.

The Normalisation Axiom is also trivially satisfied, but nevertheless we shall check it.
So, let $\hilb = \mbC,$ $a < \lambda < b,$ $H_s = s.$ Then $V = 1,$ 
and 
$$
  A_{\lambda+i y}(s) = R_{\lambda+i y}(H_s)V = \brs{s - \lambda - i y}^{-1}.
$$
For $y>0$ the only pole $s = \lambda + iy$ of this meromorphic function, which belongs 
to the group of the pole $s = \lambda$ of $A_{\lambda}(s),$ is situated 
in the upper half-plane, and therefore $\ind_{res}(\lambda; H_\lambda,V) = 1.$
\end{proof}

\begin{thm} \label{T: Robbin-Salamon uniqueness theorem}
A mapping~$\mu$ which satisfies Robbin-Salamon axioms exists and is unique.
\end{thm}
\begin{proof} 
Original proof of this theorem from \cite{RoSa} was given for operators with compact resolvent,
but an inspection of that proof shows that with minor obvious changes it applies verbatim for this case too.
Nevertheless, for readers' convenience here we outline this proof. 

In \cite{RoSa} proof of the existence part is based on checking that the intersection number satisfies 
Robbin-Salamon axioms. Here for the existence part we can also refer 
to Theorem~\ref{T: Tot res index meets R-S axioms}.

For the uniqueness, in \cite{RoSa} it is shown that any putative spectral flow mapping coincides with the 
intersection number. Again, here we show that any putative spectral flow mapping coincides with total resonance
index. 

(A) For every path $H(s)$ from $PC^1([a,b],H_a,H_b)$ 
there exist an integer $n$ and a path $B(s)$ of self-adjoint $n \times n$
matrices such that $H \oplus B$ is homotopic to a constant path. 

Proof. Firstly, the homotopy axiom allows to replace the path $H(s)$ by a path with only simple crossings
(that is, resonance points with algebraic multiplicity $N=1$).  
Further, the catenation axiom allows to reduce the problem to the case where $H(s)$ has only one simple crossing. 

Without loss of generality we assume that $a=-1,$ $b=1,$ $\lambda = 0,$ 
and that the crossing point is $s = 0.$ 
Let~$\chi$ be the eigenvector of $H_{0},$ that is, $H_{0} \chi = 0.$ 
Let $B(s) = -s,$ $B \in PC^1([-1,1],\mbC),$ and let $\phi(s), s \in [-1,1],$ 
be an eigenpath of $H_s,$ that is, $H_s \phi(s) = 0.$
Let 
$$
  \tilde H_{s,t} = 
  \left(\begin{matrix}
    H_{s} & t \phi(s) \\
    t \phi^*(s) & - s 
  \end{matrix}\right)
$$
Then $\tilde H_{s,0} = H_s \oplus B(s)$ and for $t>0$ the operator $\tilde H_{s,t}$ is invertible
for all $s \in [-1,1].$ Indeed, assume the contrary. 
Then, since $\lambda = 0$ does not belong to the essential spectrum,
there exists a non-zero vector $\tilde f = f \oplus x \in \hilb \oplus \mbC$ such that $\tilde H_{s,t}\tilde f = 0.$
It follows that $t\scal{\phi(s)}{f} = s x,$ and $H_s f = - t x \phi(s).$ The latter equality implies that 
$$
  - tx \scal{\phi(s)}{\phi(s)} = \scal{H_s \phi(s)}{f} = 0,
$$
and hence, $x = 0.$ Combining this with the former equality gives $\phi(s) \perp f.$ Also, $H_s f = 0,$
and therefore, since~$\lambda$ is a simple eigenvalue,~$f$ is co-linear with $\phi(s).$ Hence, $f = 0.$


So, the path $H_s \oplus B(s)$ is homotopic to a path $K_s$ without resonance points. 
Such a path can be continuously deformed to a constant path.

(B)
Let $\tilde \mu$ be a putative spectral flow mapping. A piecewise linear path of self-adjoint matrices
$B(s), s \in [a,b],$ is homotopic to a path of diagonal matrices.
Hence, by the homotopy, direct sum and normalisation
axioms both $\tilde \mu$ and the total resonance index of the path $B(s)$ are equal to 
$$
  \frac 12 \sign B(b) - \frac 12 \sign B(a).
$$
Now let $H(s)$ be any curve from $PC^1([a,b],H_a,H_b)$ 
and choose $B(s)$ as in part (A). Then it follows from homotopy and constant axioms 
that $\tilde \mu(H \oplus B) = 0.$ 
Hence, by the direct sum axiom, 
$$
  \tilde \mu(H) = - \tilde \mu(B) = - \mathrm{TRI}(B)= \mathrm{TRI}(H).
$$
\end{proof}

The proof in \cite{RoSa} does not use the catenation axiom, which therefore follows from the other four axioms.
In the proof above we used the catenation axiom for simplicity, though it is not necessary. 

Since spectral flow is deemed to be characterised by Robbin-Salamon axioms,
Theorem~\ref{T: Tot res index meets R-S axioms} shows that 
total resonance index and spectral flow are identical notions. 
Nevertheless, in subsection~\ref{SS: res ind and Fred ind} we give 
a direct proof of the equality
\begin{center}
``total resonance index = total Fredholm index''. 
\end{center}

\subsection{Geometric properties of the resonance set}

In this subsection we give proofs of some well-known geometric properties of the resonance set,
with the aim to provide an intuitive interpretation of spectral flow in terms of resonance set. 
\begin{thm} \label{T: codim = 1} Assume Assumption~\ref{A: Main Assumption}.
The resonance set $\Rset$ has co-dimension~1.
\end{thm}
\begin{proof} 
If codimension of $\Rset$ is $\geq 2,$ then there exists a two-dimensional affine plane
in~$\clA$ which intersects the resonance set transversally at a $\lambda$-resonant point~$H_{0}.$
A point $H_{0}+V$ on a small circle neighbourhood of~$H_{0}$ in this plane can be deformed
to $H_{0}-V$ along the circle.
By Theorem~\ref{T: stability of res index for order 1 V},
all deformations are simple and have constant resonance index.
This contradicts the equality 
$
 \ind_{res}(\lambda; H_{0},-V) = -\ind_{res}(\lambda; H_{0},V).
$
\end{proof}
At the same time, for~$\lambda$ inside essential spectrum the spectral flow
is not path-independent, see \cite[\S 8.3]{Az3v6}. This indicates that
the resonance set may have co-dimensions greater than~1 for~$\lambda$ inside essential spectrum.

Theorem~\ref{T: codim = 1} implies that the intersection of the resonance set~$\Rset$ with
any $k$-dimensional affine space which passes through a resonance point~$H_{0}$ and which 
is parallel to a simple direction~$V$ has dimension~$k-1.$

\smallskip

According to Theorem~\ref{T: codim = 1},
the resonance set $\Rset$ divides a small enough neighbourhood of any simple resonance point~$H_{0}$
into two parts. The operators in one of those parts have an eigenvalue slightly larger than~$\lambda,$
the operators in the other part have an eigenvalue slightly smaller than~$\lambda.$
We shall call these parts positive and negative. 
Resonance hyper-surfaces divide a small enough neighbourhood of a resonant operator~$H_{0}$
into several parts, which will be called \emph{cells}.
If~$V$ is a regular direction at a resonance point $H_{0},$
then it belongs to one and only one of those cells, by which we mean 
$\exists \eps>0\,\forall s \in (0,\eps)$ the operator $H_{0} + s V$
belongs to the cell.







\begin{thm} If a plane section of the resonance set consists of only simple 
curves then the number of curves in a neighbourhood of $H_{0}$ is not greater 
than the geometric multiplicity~$m.$
\end{thm}
\begin{proof} Each curve divides the plane into two parts: positive and negative.
If there are $M>m$ curves then some points of the plane section will be positive for all~$M$
curves and resonance negative for none, and some points 
will be negative for all~$M$ curves and positive for none.
The total resonance index of a continuous piece-wise linear path
starting from one of the latter points to one of the former points will be greater than~$m.$
Those two points can also be connected by a continuous piece-wise linear path which has only 
one resonance point~$H_{0}.$ By homotopy stability of the total resonance index, this path 
will have total resonance index greater than~$m$ too.
This contradicts to the U-turn inequality (\ref{F: U-turn inequality}). 
\end{proof}

We say that a plane section of the resonance set is \emph{simple}, if the section 
does not have non-simple resonance curves.

\begin{cor} In any simple plane section of the resonance set there are no more 
than~$2m$ resonance cells in a neighbourhood of a resonance point~$H_{0},$ where~$m$ is the geometric 
multiplicity of~$H_{0}.$
\end{cor}

\begin{thm} The resonance set does not have cusps. 
\end{thm}
\begin{proof} If a cusp exists then at a vertex of a cusp 
there exists a simple direction~$V$ which can be continuously deformed to the direction 
$-V$ in the set of simple directions. This implies that resonance indices of~$V$ and $-V$
are equal, which is false. 
\end{proof}

\section{Resonance index and Fredholm index}
\label{S: res index and Fredholm index}
In this section we consider relationship of the total resonance index with a traditional 
definition of spectral flow, the total Fredholm index.

\subsection{Resonance matrix as direction reduction}
\label{SS: res matrix as dir-n reduction} 

In this subsection for convenience we denote a resonance point by~$H_0,$
instead of usual~$H_{r_\lambda}.$ 

Given a~$\lambda$-resonant operator~$H_0$ and a regular direction~$V,$
to a triple $(\lambda; H_0,V)$ we can assign a finite-rank
self-adjoint operator $VP_\lambda.$ Results of this subsection demonstrate
that it is the operator $VP_\lambda$ which is responsible for spectral flow generated by 
the direction~$V.$

\begin{thm} \label{T: VP is regular}
If~$V$ is a regular direction at a resonance point~$H_{0},$
then the direction~$V P_\lambda$ is also regular.
\end{thm}
\begin{proof}
Assume the contrary. Then for any $s \in \mbR$ there exists non-zero vector $f(s)$
such that 
\begin{equation} \label{F: (H0+sVP)f=lambda f}
  (H_{0} + sV P_\lambda)f(s) = \lambda f(s).
\end{equation}
This equality can be rewritten as 
$
  (H_{0}-\lambda)f(s) = - s V P_\lambda f(s).
$
This equality implies that the vector $V P_\lambda f(s)$
is orthogonal to the eigenspace $\clV_\lambda.$ 
It also implies that 
$$
  f(s) = - s S_\lambda P_\lambda f(s) + \text{order 1 vector}.
$$
Since $V P_\lambda f(s) \perp \clV_\lambda$ and since $P_\lambda f(s)$
is a resonance vector, it follows from Theorem~\ref{T: TFAE depth 1 criterion} that the vector 
$S_\lambda P_\lambda f(s)$ is a resonance vector. Hence, by the last equality, so is the vector $f(s).$ 
That is, $P_\lambda f(s) = f(s).$
Combined with (\ref{F: (H0+sVP)f=lambda f}), this gives 
$
  (H_{0} + s V) f(s) = \lambda f(s).
$
This equality contradicts the regularity of~$V.$
\end{proof}

\begin{thm} \label{T: R(H0+sVP)VP=R(H0+sV)VP}
  For any regular direction $V$ at a resonance point $H_{0}$ and for any non-resonant $s \in \mbC$
\begin{equation}
  R_\lambda (H_{0} + sV P_\lambda) V P_\lambda = R_\lambda (H_{0}+sV) V P_\lambda,
\end{equation}
where $P_\lambda=P_\lambda(H_{0},V).$ 
\end{thm}
\begin{proof} Applying the second resolvent identity (\ref{F: second resolvent identity (2)}) 
to the pair of self-adjoint operators
$H_{0} + sV$ and $H_{0}+sV P_\lambda = H_{0} + sV - sV(1-P_\lambda)$ gives 
$$
  (E) := R_\lambda (H_{0}+sV P_\lambda) V P_\lambda = \SqBrs{1-s R_\lambda(H_{0}+sV)V(1-P_\lambda)}^{-1} R_\lambda (H_{0}+sV) V P_\lambda.
$$
Using the notation $A_\lambda(s) = R_\lambda(H_s)V,$ we can rewrite this equality as follows:
$$
  (E) = \SqBrs{1-s A_\lambda(s)(1-P_\lambda)}^{-1} A_\lambda(s)P_\lambda.
$$
It follows from (\ref{F: Az (3.3.16)}) that 
$
  A_\lambda(s)(1-P_\lambda) = \tilde A_\lambda(s),
$
where $\tilde A_\lambda(s)$ is the holomorphic (at $s = 0$) part of the meromorphic function $A_\lambda(s).$
Hence, for all small enough~$s$ we can write
\begin{equation*}
  \begin{split}
     (E) & = \SqBrs{1-s \tilde A_\lambda(s)}^{-1} A_\lambda(s)P_\lambda 
       \\ & = \SqBrs{1 + s \tilde A_\lambda(s) + s^2\tilde A^2_\lambda(s)+\ldots} A_\lambda(s)P_\lambda.
  \end{split}
\end{equation*}
Since the operators~$A_\lambda(s)$ and $P_\lambda$ commute and since $\tilde A_\lambda(s)P_\lambda = 0,$
it follows that 
$(E) = A_\lambda(s)P_\lambda.$
Since both sides of this equality are holomorphic, this equality holds for all, not necessarily small,~$s.$
This is what is required. 
\end{proof}

\begin{thm} \label{T: P(H0,VP)=P and A(H0,VP)=A} 
For any regular direction $V$ at a resonance point $H_{0}$ 
$$
  P_\lambda(H_{0},VP_\lambda(H_{0},V)) = P_\lambda(H_{0},V)
$$
and 
$$
  \bfA_\lambda(H_{0},VP_\lambda(H_{0},V)) = \bfA_\lambda(H_{0},V).
$$
\end{thm}
\begin{proof} Using definition (\ref{F: Pz=oint A(s)}) of the idempotent $P_\lambda$ 
and Theorem~\ref{T: R(H0+sVP)VP=R(H0+sV)VP}, we have 
\begin{equation*}
  \begin{split}
      P_\lambda(H_{0},VP_\lambda(H_{0},V)) & = \frac 1{2 \pi i}\oint_C R_\lambda(H_{0}+sV P_\lambda)V P_\lambda\,ds
     \\ & = \frac 1{2 \pi i}\oint_C R_\lambda(H_{0}+sV)VP_\lambda\,ds
        \  = \ P^2_\lambda = P_\lambda,
  \end{split}
\end{equation*}
where~$C$ is a contour enclosing the resonance point~$s = 0.$
Proof of the second equality is the same, but uses (\ref{F: A(lamb)=oint s A(s)}) 
instead of (\ref{F: Pz=oint A(s)}).
\end{proof}

\begin{thm} \label{T: res matrices of V and VP}
Let~$V$ be a regular direction at a resonance point~$H_0.$ 
Resonance matrices of directions~$V$ and~$VP_\lambda$ are equal. 
\end{thm}
\begin{proof} The resonance matrix of the direction~$V$ is equal to 
$
  VP_\lambda
$
and the resonance matrix of the direction $VP_\lambda $ is equal to 
$
  VP_\lambda \cdot P_\lambda(H_{0},VP_\lambda).
$
By Theorem~\ref{T: P(H0,VP)=P and A(H0,VP)=A}, these two operators are equal. 
\end{proof}
Since the resonance index of a direction is equal to the signature 
of its resonance matrix (Theorem~\ref{T: ind res=sign VP}), we have 
the following corollary of Theorem~\ref{T: res matrices of V and VP}.

\begin{thm} \label{T: ind(V)=ind(VP)} 
Let~$V$ be a regular direction. Resonance indices of directions~$V$ and~$VP_\lambda$
are equal. 
\end{thm}
This theorem is important in that aspect that it often allows to replace a direction~$V$ by a finite rank
direction $VP_\lambda.$

\begin{defn}
We say that two directions $V_1$ and $V_2$ at a resonance point~$H_0$ are \emph{plane homotopic},
if $\exists \eps>0 \ \forall (s_1,s_2) \in [0,\eps]^2 \setminus \set{(0,0)}$
the operators 
$H_{0} + s_1 V_1 + s_2 V_2$ and $H_{0} - s_1 V_1 - s_2 V_2$ are regular at~$\lambda.$
\end{defn}
\noindent 
Geometrically, two directions $V_1$ and $V_2$ are plane homotopic, if one of them can be deformed to another
within the affine plane they generate so that the half-interval being deformed stays outside the resonance set.

\begin{thm} \label{T: V and VP are plane homotopic}
If~$V$ is a regular direction at a resonance point~$H_{0},$ then 
the directions~$V$ and~$V P_\lambda$ are plane homotopic.
\end{thm}
\begin{proof}
Firstly, by Theorem~\ref{T: VP is regular}, the direction $V P_\lambda$ is regular.
Let 
$$
  H_{s,t} = H_{0} + s V + t V P_\lambda.
$$
Since~$V$ is regular, for all small enough~$s$ the operator $H_{0} + s V - \lambda$
is invertible. We need to show that 
for all small enough~$s$ and $t$ the operator $H_{0} + s V + t V P_\lambda - \lambda$
is also invertible.
By the second resolvent identity, we have 
\begin{equation}
   \begin{split}
      R_\lambda(H_{s,t}) & = \brs{1 + t R_\lambda(H_s) V P_\lambda}^{-1} R_\lambda(H_s)
         \\ & = \brs{1 + t A_\lambda(s) P_\lambda}^{-1} R_\lambda(H_s),
   \end{split}
\end{equation}
whenever the right hand side makes sense. 
From this equality we can see that to prove the claim it is enough to prove the following 
assertion: if $1+ t A_\lambda(s)$ is invertible (which is equivalent to existence of 
$R_\lambda(H_{0} + (s+t)V)$), then $1+ t A_\lambda(s) P_\lambda$
is also invertible. Assume the contrary. Then, since $A_\lambda(s) P_\lambda$ is compact,
there exists a non-zero vector~$\chi$ such that 
$$
  (1+ t A_\lambda(s) P_\lambda) \chi = 0.
$$
Since~$A_\lambda(s)$ and $P_\lambda$ commute, this equality implies 
$
 \chi = - t P_\lambda A_\lambda(s) \chi.
$
It follows that $P_\lambda \chi = \chi$ and therefore
$
 (1+ t A_\lambda(s)) \chi = 0.
$
Hence, $1+ t A_\lambda(s)$ is not invertible. 
\end{proof}

Theorem~\ref{T: V and VP are plane homotopic} provides another proof of the equality of the resonance index
and of the signature of the resonance matrix for the case where~$\lambda$ does not belong to the essential spectrum.

\begin{prop} \label{P: positive directions are plane-homotopic}
Any two regular non-negative (or-non-positive) directions are plane homotopic. 
\end{prop}
\begin{proof} 
Let $V_1$ and $V_2$ be two regular and non-negative directions at a resonance point~$H_0.$
Since $V_1$ is regular, for all small enough $s>0$ the operator $H_{0} + s V_1$ 
is non-resonant and near~$\lambda$ there are only eigenvalues of $H_{0} + s V_1$ which are larger than~$\lambda.$
Adding $t V_2$ can only increase these eigenvalues. Choosing~$s$ and $t$ small enough we can also ensure 
that there are no other eigenvalues $H_{0} + s V_1 + tV_2$ near~$\lambda.$
\end{proof}

\begin{thm}
Plane homotopic directions have equal resonance indices. 
\end{thm}
This is a special case of homotopy stability of resonance index, Theorem~\ref{T: homotopy stability of res index}.

\subsection{Resonance index and Fredholm index}
\label{SS: res ind and Fred ind}
In this section we consider relationship of the resonance index with the Fredholm index.
For reader's convenience we recall here some well-known definitions and theorems. 

A bounded operator~$T$ acting from a Hilbert space
$\hilb$ to a Hilbert space $\clK$ is \emph{Fredholm}, if $T$ has closed range
and if the kernel of $T$ and the kernel of $T^*$ are finite-dimensional.
In this case the \emph{index} of $T$ is the integer number 
$$
  \ind(T) = \dim\ker(T) - \dim\ker(T^*).
$$
The word ``bounded'' in the definition of a Fredholm operator can be replaced by the word ``closed'',
and ``Hilbert'' can be replaced by ``Banach'', but we do not need this.
Since the index is sensitive to the choice of domain and range, one may also write $\ind_{\hilb,\clK}(T).$

We denote by $\clF(\hilb,\clK)$ the set of all (bounded) Fredholm operators from~$\hilb$ to~$\clK.$ 
If $\hilb=\clK$ then one writes $\clF(\hilb)$ for $\clF(\hilb,\hilb).$ 
The set $\clF(\hilb,\clK)$ has the following properties, proofs of which can be found 
in e.g. \cite[Chapter XIX]{Hoer3}.

(1) For any compact operator $K$ on a Hilbert space $\hilb$ the operator $1+K$ is Fredholm
and $\ind(1+K) = 0.$ 

(2) The set of Fredholm operators is open in the norm topology. The index $\ind(T)$ is stable 
in the norm topology, that is, $\ind$ is a locally constant function on $\clF(\hilb,\clK).$ 

(3) If $T \in \clF(\hilb,\clK)$ and $K \colon \hilb \to \clK$ is compact,
then $T+K \in \clF(\hilb,\clK)$ and $\ind(T) = \ind(T+K).$ (1) is a special case of this property. 

(4) If $T \in \clF(\hilb_1,\hilb_2)$ and $S \in \clF(\hilb_2,\hilb_3),$ 
then $ST \in \clF(\hilb_1,\hilb_3)$ and 
$\ind(ST) = \ind(S)+\ind(T).$ 

(5) $T \in \clF(\hilb,\clK)$ if and only if 
there exists a bounded operator $S\colon \clK \to \hilb$ such that $ST - 1_\hilb$ and $TS-1_\clK$ are compact.
Such an operator $S$ is called \emph{parametrix} of $T$ and it is also Fredholm. Moreover, $\ind(S) = - \ind(T).$ 

\smallskip
A pair of orthogonal projections $(P,Q)$ is called a \emph{Fredholm pair}, if 
the operator $PQ \colon Q\hilb \to P \hilb$ is Fredholm. 
The \emph{Essential co-dimension} of the pair $(P,Q)$ is the index 
of the operator~$PQ,$ it is denoted by $\ec(P,Q).$ 

A pair $(P,Q)$ is Fredholm iff $\norm{\pi(P-Q)}<1,$ where $\pi \colon \clB(\hilb) \to \clQ(\hilb)$ 
is the canonical epimorphism of the $C^*$-algebra of bounded operators $\clB(\hilb)$ 
onto the Calkin $C^*$-algebra $\clQ(\hilb) = \clB(\hilb)/\clK(\hilb),$ 
where $\clK(\hilb)$ is the norm-closed ideal of compact operators, for a proof 
see e.g. \cite[Lemma 4.1]{BCPRSW}. 
If $(P_1,P_2)$ and $(P_2,P_3)$ are Fredholm pairs, then so are the pairs $(P_1,P_3)$ and $(P_2,P_1)$
and 
$$
  \ec(P_1,P_3) = \ec(P_1,P_2) + \ec(P_2,P_3), \quad \ec(P_2,P_1)=-\ec(P_1,P_2).
$$

Now we proceed to a discussion of J.\,Phillips' definition \cite{Ph96CMB,Ph97FIC} of spectral flow 
as total Fredholm index. The theory 
of spectral flow was developed for self-adjoint operators $H_{0}$ with compact resolvent
and with some summability condition such as $p$-summability and $\theta$-summability
(though $p$-summability implies $\theta$-summability which in its turn implies compactness of resolvent,
we choose to mention both).
This assumes that the spectrum of $H_{0}$ is discrete and so it has no essential spectrum.
The spectral flow theory originates in the analysis of elliptic differential operators $\clD$ 
acting on sections of vector bundles over compact manifolds, such as Dirac operators on spin manifolds,
and these operators satisfy the condition of compact resolvent and some summability assumptions.
Our aim here is to demonstrate directly that the spectral flow as total Fredholm index and the total resonance index
are identical notions. But since the theory of the former was developed for operators with compact resolvent 
we assume here this condition. An inspection of proofs of basic theorems of spectral flow theory shows
that the summability conditions are used essentially. It is quite possible that double operator 
integral techniques may allow to adjust the theory so that it becomes applicable to operators with essential spectrum
as long as zero (or more generally a point~$\lambda$) does not belong to it, but carrying out this plan 
may or may not be straightforward. 
In any case, as it was demonstrated in \cite{ACS},
it is sufficient to assume only compactness of resolvent without summability conditions.

\smallskip
For the rest of this subsection, we assume that $H_0$ has compact resolvent 
and that~$\clA_0$ is a subspace of the algebra of bounded operators with the operator norm. 
If $H_s$ is a continuous path in~$\clA$ such that $H_0$ and $H_1$ are not~$\lambda$-resonant,
then the \emph{spectral flow} of $\set{H_s}_{s \in [0,1]}$ through~$\lambda,$ 
by definition of J.\,Phillips, is the number
$$
  \sflow(\set{H_s}) = \sum_{j=1}^n \ec(P_{s_{j-1}},P_{s_j}),
$$
where $P_s$ is the spectral projection $E_{[\lambda,\infty)}^{H_s},$ and $\set{s_j}_{j=0}^n$ is 
a partition of the interval~$[0,1].$ The spectral flow is well-defined for all partitions
with small enough diameter $\max_{j}\abs{s_j-s_{j-1}},$ and it does not depend on the choice of such a partition,
which easily follows from the additivity property of the essential co-dimension. 
Moreover, spectral flow is homotopy invariant. 
Since this is one of several definitions of spectral flow, we shall call it here \emph{total Fredholm index}. 

This preliminary material can be found 
in e.g. \cite{Ph96CMB,Ph97FIC,BCPRSW}, \cite[\S\S 1.5, 1.6]{Azbook} and \cite[Chapter XIX]{Hoer3}.

\smallskip

We give a direct proof of the equality of the total resonance index and total Fredholm index,
which does not allude to the Robbin-Salamon uniqueness Theorem~\ref{T: Robbin-Salamon uniqueness theorem}. 

By \emph{Fredholm index of a regular direction} $V$ at $H_{0}$ we mean the number
$$
  \ec\brs{E^{H_{0}-\eps V}_{[\lambda,\infty)},E^{H_{0}+\eps V}_{[\lambda,\infty)}},
$$
where~$\eps>0$ is a small enough number. General theory \cite{Ph96CMB,Ph97FIC} shows 
that this essential co-dimension is independent from the choice of sufficiently small~$\eps>0.$

\begin{thm} \label{T: TRI=TFI} 
Total resonance index coincides with total Fredholm index.
\end{thm}
\begin{proof} Let $H_r$ be a continuous piecewise linear path which connects 
two~$\lambda$-regular operators~$H_0$ and~$H_1.$ Using a small perturbation of this path,
we can modify it in such a way that all crossings of this path with the resonance set will 
occur at simple points and at simple directions. Since both 
the total resonance index and the total Fredholm index are homotopy 
stable (Theorem~\ref{T: homotopy stability of res index}), 
the Catenation Axiom (which trivially holds for both the total resonance index and the total Fredholm index)
reduces the matter to the case of a path~$H_r$ which intersects the resonance set only once at a simple point
and at a simple direction. Let~$V$ be the direction. 
By Theorem~\ref{T: V and VP are plane homotopic},
the direction~$V$ is plane homotopic to~$VP_\lambda.$
By the Homotopy Axiom, it follows that directions~$V$ and~$VP_\lambda$ have 
the same resonance index and Fredholm index. 
Since~$V$ is a simple direction at a simple point, the resonance matrix~$VP_\lambda$ 
is a rank one self-adjoint operator. Depending on whether the resonance index of~$V$ is~$+1$
or~$-1,$ the operator~$V P_\lambda$ is plane deformable to~$\scal{\chi}{\cdot}\chi$ or~$-\scal{\chi}{\cdot}\chi,$
where~$\chi$ is an eigenvector of the simple point where~$H_r$ crosses the resonance set. 
Finally, it is left to note that the Fredholm index of the regular direction~$\scal{\chi}{\cdot}\chi$ 
(respectively, $-\scal{\chi}{\cdot}\chi$) is obviously equal to~$1$ (respectively, $-1$).
\end{proof}

\section{Resonance index and spectral shift function}
\label{S: res ind=SSF}
The aim of this section is to demonstrate the equality 
$$
  \text{spectral shift function} = \text{total resonance index}
$$
outside the essential spectrum. Since this is a special case of an essentially stronger result
which asserts that the total resonance index is equal to the singular spectral shift function for a.e.~$\lambda$
\cite[\S 6]{Az9}, \cite{Az7}, 
we do not formulate any theorems here. 
It is also well-known that the spectral shift function outside essential spectrum satisfies Robbin-Salamon axioms,
and therefore the total resonance index and spectral shift function coincide 
by uniqueness Theorem~\ref{T: Robbin-Salamon uniqueness theorem}. 
Nevertheless, here we demonstrate the argument of the proof of the above-mentioned more general result
in this special case, where it simplifies quite significantly while retaining one of the key points 
of the proof. 

Given two self-adjoint operators $H_0$ and $H_1$ with trace class difference $V = H_1 - H_0,$
the \emph{spectral shift function} of the pair $(H_0,H_1)$ is a unique real-valued 
integrable function $\xi \in L_1(\mbR),$
such that for all compactly supported functions $\phi$ of class $C^2$ the \emph{Lifshitz-Krein trace formula}
\cite{Kr53MS,Li52UMN} holds:
$$
  \Tr\brs{\phi(H_1)-\phi(H_0)} = \int_{-\infty}^\infty \phi'(\lambda)\xi(\lambda)\,d\lambda.
$$
The Birman-Solomyak formula \cite{BS75SM} gives another remarkable representation for the spectral shift function,
but this formula treats the SSF as a distribution:
$$
  \xi(\phi) = \int_0^1 \Tr\brs{V \phi(H_r)}\,dr,
$$
where $H_r = H_0+rV.$ This formula indicates that the SSF is an integral of a one-form 
$\Tr\brs{V \phi(H_r)}$ on the affine space of trace class perturbations of~$H_0.$ This form is \emph{exact},
so that the straight line $H_0+rV$ connecting the operators $H_0$ and $H_1$ can be replaced by a piecewise 
smooth path \cite{AS2}. In the spectral flow theory there exist analytic integral formulas for spectral flow 
(due to Getzler and Carey-Phillips)
which can be considered as analogues of the Birman-Solomyak formula with specifically chosen distribution 
$\phi,$ though the settings of the operator theoretic spectral shift function and the differential geometric 
spectral flow were different (the former with relatively trace class conditions imposed on 
the perturbation $V$ with more or less arbitrary self-adjoint~$H_0,$ while in the latter 
the perturbation~$V$ is an arbitrary bounded self-adjoint operator with summability conditions imposed on~$H_0$).
Alan Carey indicated many times that the idea of possibility of expressing the spectral flow as an integral of one-form
belongs to I.\,M.\,Singer (1974). 

Starting with the Birman-Solomyak formula as the definition of the SSF, 
one can show that it satisfies Krein's trace formula
in the form 
$$
  \Tr\brs{\phi(H_1)-\phi(H_0)} = \int_{-\infty}^\infty \phi'(\lambda)\,dm_\xi(\lambda),
$$
where $m_\xi$ is the spectral shift measure. 
Though proof of the Birman-Solomyak formula is somewhat simpler than that of the Lifshitz-Krein formula,
it does not allow to prove that the SSF is absolutely continuous, so the original proof 
of Lifshitz-Krein formula is indispensable. Nevertheless, there are sufficient indications, 
some of which were mentioned above,
that the Birman-Solomyak formula is more fundamental.

\smallskip
Outside the essential spectrum the spectral shift function coincides with the spectral flow. This connection
was demonstrated in \cite{ACS}, though it seems unlikely that this connection was not known in some form before. 
Since the spectral flow is an inherently integer-valued function, the SSF is also integer-valued outside the essential 
spectrum, but inside the essential spectrum this is not the case. 
The celebrated Birman-Krein formula, which connects the SSF and the scattering matrix,
$$
  \det S(\lambda; H_1,H_0) = e^{-2\pi i \xi(\lambda)},
$$
indicates that the reason for non-integrality of SSF is existence of a non-trivial scattering matrix.
It turns out however \cite{Az,Az3v6} that the SSF admits a natural decomposition as a sum of two components,
the \emph{absolutely continuous}~$\xia$ and \emph{singular spectral shift functions}~$\xis,$
such that the second component~$\xis$ is integer-valued almost everywhere inside the essential spectrum too. 
The definitions of the functions $\xia$ and $\xis$ are obtained by modified Birman-Solomyak formulas:
$$
  \xia(\phi) = \int_0^1 \Tr\brs{V \phi(H^{(a)}_r)}\,dr 
$$
and 
$$  
  \xis(\phi) = \int_0^1 \Tr\brs{V \phi(H^{(s)}_r)}\,dr,
$$
where $H^{(a)}$ and $H^{(s)}$ stand for the absolutely continuous 
and singular parts of a self-adjoint operator~$H$ respectively. 
It is shown in \cite{Az3v6} that for trace class perturbations the measure $\xia$ is absolutely continuous
and its density $\xia(\lambda)$ satisfies a modified Birman-Krein formula
$$
  \det S(\lambda; H_1,H_0) = e^{-2\pi i \xia(\lambda)}.
$$
This formula combined with Birman-Krein formula implies integrality of $\xis.$

Definitions of functions $\xia$ and $\xis$ indicate that the generalised spectral flow one-form
$\Tr(V\delta(H))$ naturally splits into two components: absolutely continuous 
$\Tr(V\delta(H^{(a)}))$ and singular $\Tr(V\delta(H^{(s)}))$ spectral flow one-forms.
But unlike the spectral flow one-form, these two components are not exact, see a counter-example in 
\cite[\S 8.3]{Az3v6}. Outside the essential spectrum the absolutely continuous spectral flow one-form drops out,
and thus in this case the remaining singular part becomes exact. In terms of the resonance set $\euR(\lambda)$
exactness of the singular spectral flow one-form is equivalent 
to the equality $\mathrm{codim}\,\euR(\lambda) = 1,$ and while outside the essential spectrum 
the resonance set $\euR(\lambda)$ always has co-dimension one, inside the essential spectrum this is not the case. 

The definition of $\xis$ is hardly suitable for calculation of this function. 
In \cite{Az7} (see also \cite[Section 6]{Az9}) it was shown that the singular spectral shift function
is equal to the total resonance index, which is something far easier to work with. Proof of this equality
for values of~$\lambda$ inside essential spectrum 
cannot be explained in a short space, but if~$\lambda$ is outside the essential spectrum then the proof
simplifies sufficiently to present it here.

It is probably worthwhile to stress here that the notion of resonance index was discovered in the course
of work on the singular spectral shift function and scattering theory. Since the singular spectral shift function
coincides with spectral shift function outside the essential spectrum, it was immediately clear that the total 
resonance index coincides with spectral flow, and this is how this paper was originated. 

We assume that a self-adjoint perturbation~$V$ of a self-adjoint operator~$H_0$ is trace class. 
This assumption is not necessary for the equality in question:
$$
  \text{if \ \ $\lambda \notin \sigma_{ess}$ \ \ then } \ \ \xi(\lambda) = \sum_{r \in [0,1]} \ind_{res}(\lambda; H_r,V),
$$ 
but it will significantly simplify the proof. For relatively trace class perturbations the proof, which includes
the main essential spectrum case too, will appear
in \cite{AzDa}. We conjecture that this equality holds as long as~$V$ is a relatively compact perturbation of~$H_0.$ 

The Birman-Solomyak formula implies the equality 
$$
  \frac 1 \pi \int_\mbR \Im R_{\lambda+iy}(x)\xi(x)\,dx = \frac 1 \pi \int_0^1 \Tr \brs{V \Im R_{\lambda+iy}(H_r)}\,dr.
$$
The left hand side of this equality is the Poisson integral of the function $\xi.$
Hence, a well-known property of the Poisson integral implies that as $y \to 0^+,$
the left hand side converges to $\xi(\lambda)$ for a.e.~$\lambda.$ 
Hence, we will be done if we show that the right hand side converges to the total resonance index. 
Since~$\lambda$ lies outside the common essential spectrum of operators~$H_r,$ the operator
$\Im R_{\lambda+iy}(H_r)$ has zero limit for all values of~$r$ from the interval $[0,1],$
except some special values of $r$ for which the resolvent $R_{\lambda+i0}(H_r)=R_{\lambda}(H_r)$ does not exist. 
Since~$\lambda$ is outside the essential spectrum, these values of $r$ are those for which~$\lambda$
is an eigenvalue of~$H_r,$ that is, they are resonance points. 
In other words, a hindrance for taking the limit $y \to 0^+$ in the right hand side is presence of resonance points.
If there were no resonance points, the limit of the right hand side would be zero, 
which agrees with the fact that no eigenvalue of~$H_r$ reached the point~$\lambda$ 
and therefore spectral flow through~$\lambda$ is zero. 

So, the presence of resonance points in the domain $[0,1]$ of integration 
is thus a hindrance, but it is the presence of resonance points which makes 
the RHS non-zero and interesting. To overcome this hindrance we note that the integrand
of the RHS
$$
  \Tr \brs{V \Im R_{\lambda+iy}(H_r)}
$$
is a meromorphic function of~$r.$ A small neighbourhood of the interval $[0,1]$
in the coupling constant complex plane for sufficiently small $y>0$ contains poles of the functions
$VR_{\lambda+iy}(H_s)$ and $VR_{\lambda-iy}(H_s),$ and these poles converge to the poles $r_\lambda^{j}$
of $V R_{\lambda}(H_s)$ from $[0,1]$ as $y \to 0.$ We represent the path of integration $[0,1]$
as the sum of a path $L_1,$ shown below, which circumvents the poles from above, and closed contours~$C_+^j,$
encircling poles of the group of $r_\lambda^j$ lying in~$\mbC_+.$ Clearly, as $y \to 0^+$ the integral 
over~$L_1$ vanishes: 
there are no obstructions in the form of poles on $L_1$ in the limit and the integrand is zero when $y=0.$
(The following figures are taken from \cite{Az11})

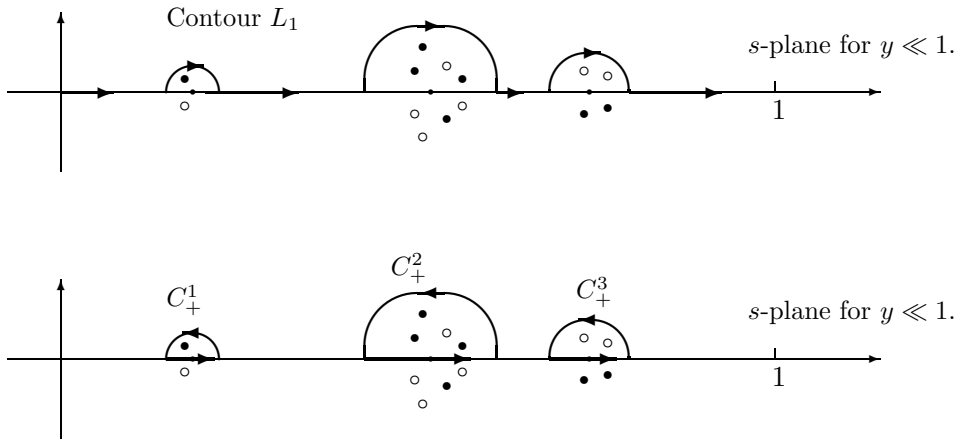
\begin{picture}(200,100)
\put(280,65){\small $s$-plane for $y \ll 1.$}
\put(60,75){\small Contour $L_1$}
\put(0,50){\vector(1,0){330}}

\put(20,20){\vector(0,1){60}}
\put(290,50){\line(0,1){4}}
\put(289,40){$1$}

\put(70,50){\circle*{2}}
\put(67,55){\circle*{3}} \put(67,45){\circle{3}}

\put(160,50){\circle*{2}}
\put(154,58){\circle*{3}} \put(154,42){\circle{3}}
\put(157,67){\circle*{3}} \put(157,33){\circle{3}}
\put(172,55){\circle*{3}} \put(172,45){\circle{3}}
\put(166,60){\circle{3}}  \put(166,40){\circle*{3}}

\put(220,50){\circle*{2}}
\put(227,56){\circle{3}} \put(227,44){\circle*{3}}
\put(218,58){\circle{3}} \put(218,42){\circle*{3}}

\thicklines
\put(235,50){\vector(1,0){35}}
\put(220,50){\oval(30,30)[t]}
\put(220,65){\vector(1,0){4}}
\put(185,50){\vector(1,0){10}}
\put(160,50){\oval(50,50)[t]}
\put(160,75){\vector(1,0){4}}
\put(75,50){\vector(1,0){35}}
\put(70,50){\oval(20,20)[t]}
\put(70,60){\vector(1,0){4}}
\put(20,50){\vector(1,0){20}}
\end{picture}

\begin{picture}(200,100)
\put(280,65){\small $s$-plane for $y \ll 1.$}
\put(0,50){\vector(1,0){330}}

\put(20,20){\vector(0,1){60}}
\put(290,50){\line(0,1){4}}
\put(289,40){$1$}

\put(70,50){\circle*{2}}
\put(67,55){\circle*{3}} \put(67,45){\circle{3}}

\put(160,50){\circle*{2}}
\put(154,58){\circle*{3}} \put(154,42){\circle{3}}
\put(157,67){\circle*{3}} \put(157,33){\circle{3}}
\put(172,55){\circle*{3}} \put(172,45){\circle{3}}
\put(166,60){\circle{3}}  \put(166,40){\circle*{3}}

\put(220,50){\circle*{2}}
\put(227,56){\circle{3}} \put(227,44){\circle*{3}}
\put(218,58){\circle{3}} \put(218,42){\circle*{3}}

\thicklines
\put(215,73){\small $C_+^3$}
\put(205,50){\vector(1,0){25}}
\put(220,50){\oval(30,30)[t]}
\put(220,65){\vector(-1,0){4}}

\put(145,82){\small $C_+^2$}
\put(135,50){\vector(1,0){40}}
\put(160,50){\oval(50,50)[t]}
\put(160,75){\vector(-1,0){4}}

\put(60,70){\small $C_+^1$}
\put(70,50){\oval(20,20)[t]}
\put(70,60){\vector(-1,0){4}}
\put(60,50){\vector(1,0){18}}
\end{picture}

We show that 
$$
  \frac 1\pi \oint_{C_+^j} \Tr \brs{V \Im R_{\lambda+iy}(H_s)}\,ds = \ind_{res}(\lambda; H_{r_\lambda^j},V)
$$
and this will complete the proof. It is enough to do this for one contour, so we omit the superscript
index. We have 
\begin{equation*}
  \begin{split}
     LHS & = \frac 1{2\pi i} \oint_{C_+} \Tr (V R_{\lambda+iy}(H_s) - V R_{\lambda-iy}(H_s))\,ds 
      \\ & = \Tr\brs{\frac 1{2\pi i} \oint_{C_+} R_{\lambda+iy}(H_s)V\,ds} - \Tr\brs{\frac 1{2\pi i} \oint_{C_+} R_{\lambda-iy}(H_s)V\,ds}.
  \end{split}
\end{equation*}
We note that these two integrals are equal to  $P^\uparrow_{\lambda+iy}(r_\lambda)$ 
and $P^\uparrow_{\lambda-iy}(r_\lambda).$ It is left to note that traces of these idempotents are $N_+$
and $N_-$ respectively. 


\mathsurround 0pt


\bigskip

\begin{center} \large Index
\end{center}

\noindent
$\aaa,$ restriction of $V$ to the eigenspace $\clV_\lambda,$ p.\,\pageref{Page: aaa}\\
$\clA,$ real affine space of self-adjoint operators $H,$ p.\,\pageref{Page: clA}\\
$\clA_0,$ real vector space of self-adjoint operators $V,$ p.\,\pageref{Page: clA0}\\
$A_z(s),$ the operator $R_z(H)V$ \\
$\bfA_z(r_z),$ nilpotent operator, p.\,\pageref{Page: bfA(z)} \\
$\bfA_z(r_\lambda),$ nilpotent operator, p.\,\pageref{Page: Az(r lmd)}\\
$\bfA_\lambda^{[\nu]},$ restriction of $\bfA_\lambda(r_\lambda)$ 
     to the vector space $\Upsilon_\lambda^{[\nu]},$ p.\,\pageref{Page: bfA(nu)1}, \pageref{Page: bfA(nu)2} \\
$B_z(s),$ the operator $V R_z(H)$ \\
$d,$ order of a resonance point \\
$d_\nu,$ size of a $\nu$th Jordan cell, size of a $\nu$th resonance cycle \\
$\tilde d_\nu,$ order of eigenpath $\phi_\nu(s)$ \\
$\euD_\lambda(s),$ the $(2,2)$ entry of resolvent $(H_r-\lambda)^{-1},$ (\ref{F: euD lambda(r)}), p.\,\pageref{Page: euD lambda(r)} \\
$D_j,$ Laurent coefficients of $\euD_\lambda(s),$ (\ref{F: Laurent for euD}), p.\,\pageref{Page: Laurent for euD} \\
$\euF_z(s),$  (\ref{F: euF(r)=1-r Az(r)}), p.\,\pageref{Page: euF(r)=1-r Az(r)} \\
$\hilb,$ Hilbert space \\
$\hat \hilb,$ Hilbert space, orthogonal complement of eigenspace $\clV_\lambda,$ p.\,\pageref{Page: hat hilb} \\
$H,$ self-adjoint operator from an affine space $\clA$ \\
$\hat H,$ restriction of $H$ to $\hat \hilb,$ p.\,\pageref{Page: hat H0} \\
$H_s,$ a path of operators $H_0 + sV$ \\
$H(s),$ a path of operators for which~$\lambda$ is an eigenvalue, pp.\,\pageref{Page: H(s) 1},\,\pageref{Page: H(s) 2},\,\pageref{Page: H(s) 3} \\ 
$\ind_{res}(\lambda; H_{r_\lambda},V)),$ resonance index \\
$m,$ geometric multiplicity of eigenvalue~$\lambda$ \\
$N,$ algebraic multiplicity of eigenvalue~$\lambda$ \\
$\hat P,$ orthogonal projection onto $\hat \hilb,$ p.\,\pageref{Page: hat P}\\
$P_z(r_z),$ idempotent operator, p.\,\pageref{Page: Pz(rz)}\\
$P_\lambda^{[\nu]},$ idempotent, (\ref{F: P(lambda)[nu]:=}), p.\,\pageref{Page: P(lambda)[nu]:=} \\
$P_z(r_\lambda),$ idempotent operator, p.\,\pageref{Page: Pz(r lmd)}\\
$\euR(\lambda),$ the resonance set, the set of operators $H$ from $\clA$ for which~$\lambda$ is an eigenvalue \\
$r,$ coupling constant, usually a real number \\
$r_\lambda,$ a resonance point, p.\,\pageref{Page: res point} \\
$r_z,$ a resonance point, p.\,\pageref{Page: res point} \\
$r_\nu^{(j)}(z),$ a resonance point from $\nu$th cycle, p.\,\pageref{Page: r(nu)(.)(z)} \\
$r_\nu^{(\cdot)}(z),$ a cycle of resonance points, p.\,\pageref{Page: r(nu)(.)(z)} \\
$s,$ coupling constant, a complex number \\
$S_\lambda,$ the operator $R_\lambda(\hat H_{\rlmb})V,$ (\ref{F: S(lambda)}), p.\,\pageref{Page: S(lambda)} \\
$V,$ self-adjoint operator from the real vector space $\clA_0$ \\
$\hat V,$ restriction of $V$ to $\hat \hilb,$ p.\,\pageref{Page: hat V}\\
$\clV_\lambda,$ the eigenspace of a resonant operator, p.\,\pageref{Page: clV(lamb)}\\
$v,$ the operator $\hat P V \hat P^\perp,$ the $(1,2)$ matrix element of $V,$ p.\,\pageref{Page: v}\\
$y,$ the imaginary part of spectral parameter $z$ \\
$Y_j,$ p.\,\pageref{Page: Y(j)} \\ 
$z,$ spectral parameter, a complex number outside the common essential spectrum $\sigma_{ess}$ \\
$\gamma_\chi,$ a curve of resonant operators, p.\,\pageref{Page: gamma chi} \\
$\lambda,$ an eigenvalue, a real number outside the common essential spectrum $\sigma_{ess}$ \\
$\lambda_\nu(s),$ a path of eigenvalues of $H_0+sV$ \\ 
$\nu,$ eigenvalue function index, resonance cycle index, Jordan cell index \\
$\sigma_\lambda(s),$ equal to $(s-r_\lambda)^{-1},$ eigenvalue of $A_\lambda(s)$ \\
$\Upsilon_z(r_z),$ the range of $P_z(r_z),$ p.\,\pageref{Page Upsilon z rz} \\
$\Upsilon_\lambda^{[\nu]},$ the range of $P_\lambda^{[\nu]},$ p.\pageref{Page: Upsilon lambda nu} \\
$\phi_\nu(s),$ a path of eigenvectors of $H_s=H_0+sV$ \\
$\chi(s),$ a path of operators for which~$\lambda$ is an eigenvalue 

\smallskip
\noindent Curve, \\
\mbox{ }\quad --- regular, p.\,\pageref{Page: regular path} \\
\mbox{ }\quad --- resonant, p.\,\pageref{Page: resonant path} \\
\noindent Direction, \\
\mbox{ }\quad --- of order~$d$, p.\,\pageref{Page: simple direction} \\
\mbox{ }\quad --- regular, p.\,\pageref{Page: regular direction} \\
\mbox{ }\quad --- simple, p.\,\pageref{Page: simple direction} \\
\mbox{ }\quad --- tangent, p.\,\pageref{Page: tangent direction(0)},\,\pageref{Page: tangent direction} \\
\mbox{ }\quad --- tangent to order~$k,$ pp.\,\pageref{Page: tangent to order k direction} \\
\mbox{ }\quad --- transversal, pp.\,\pageref{Page: transversal direction(0)},\,\pageref{Page: transversal direction} \\
\noindent Path, \\
\mbox{ }\quad --- regular, p.\,\pageref{Page: regular path} \\
\mbox{ }\quad --- resonant, p.\,\pageref{Page: resonant path} \\
\mbox{ }\quad --- standard, p.\,\pageref{Page: standard path(0)},\,\pageref{Page: standard path} \\
\noindent Point, \\
\mbox{ }\quad --- of geometric multiplicity~$m$ \\
\mbox{ }\quad --- resonance, p.\,\pageref{Page: resonance point} \\ 
\mbox{ }\quad --- simple, p.\,\pageref{Page: simple point} \\
\noindent Order of eigenpath \ p.\,\pageref{Page: property U(k)} \\
\mbox{ }\quad --- strict, p.\,\pageref{Page: strict U(k)}\\
\noindent Resonance index, \ p.\pageref{Page: res index}\\
\mbox{ }\quad --- total, p.\,\pageref{Page: total res index}\\
\noindent Resonance vector, \ p.\,\pageref{Page: res vector}\\
\mbox{ }\quad --- of depth $k$, p.\,\pageref{Page: depth of vector},\,\pageref{Page: depth of vector(2)}\\
\mbox{ }\quad --- of order~$k$, p.\,\pageref{Page: order of res vector}\\
TRI, total resonance index, p.\,\pageref{Page: total res index}\\

\end{document}